\documentclass[a4paper, reqno, 11pt, notitlepage]{amsart}
\usepackage{fullpage}
\usepackage[utf8]{inputenc}

\usepackage{latexsym}
\usepackage{amsmath}
\usepackage{amssymb}
\usepackage{amsthm}
\usepackage{amscd}
\usepackage{mathrsfs}
\usepackage[all]{xy}
\usepackage{graphicx}
\usepackage{comment}
\usepackage{arydshln}
\usepackage{mathtools}
\usepackage{dsfont}
\usepackage{adjustbox}
\usepackage{multicol}
\usepackage{quiver}
\usepackage{changepage} 
\usepackage[activate={true,nocompatibility},final,tracking=true,kerning=true,spacing=true]{microtype}
\microtypecontext{spacing=nonfrench}
\usepackage{eucal}
\usepackage{yfonts}

\numberwithin{equation}{subsection}

\makeatletter

\renewcommand{\tocsection}[3]{%
  \indentlabel{\@ifnotempty{#2}{\bfseries\ignorespaces#1 #2\quad}}\bfseries#3}

\renewcommand{\tocsubsection}[3]{%
  \indentlabel{\@ifnotempty{#2}{\ignorespaces#1 #2\quad}}#3}

\newcommand\@dotsep{4.5}
\def\@tocline#1#2#3#4#5#6#7{\relax
  \ifnum #1>\c@tocdepth 
  \else
    \par \addpenalty\@secpenalty\addvspace{#2}%
    \begingroup \hyphenpenalty\@M
    \@ifempty{#4}{%
      \@tempdima\csname r@tocindent\number#1\endcsname\relax
    }{%
      \@tempdima#4\relax
    }%
    \parindent\z@ \leftskip#3\relax \advance\leftskip\@tempdima\relax
    \rightskip\@pnumwidth plus1em \parfillskip-\@pnumwidth
    #5\leavevmode\hskip-\@tempdima{#6}\nobreak
    \leaders\hbox{$\m@th\mkern \@dotsep mu\hbox{.}\mkern \@dotsep mu$}\hfill
    \nobreak
    \hbox to\@pnumwidth{\@tocpagenum{\ifnum#1=1\bfseries\fi#7}}\par
    \nobreak
    \endgroup
  \fi}
\AtBeginDocument{%
\expandafter\renewcommand\csname r@tocindent0\endcsname{0pt}
}
\def\l@subsection{\@tocline{2}{0pt}{2.5pc}{5pc}{}}
\makeatother

    \makeatletter
    \def\paragraph{\@startsection{paragraph}{4}%
    \z@\z@{-\fontdimen2\font}%
    {\normalfont\bfseries}}
    \makeatother

\usepackage[T1]{fontenc}

\usepackage{tikz-cd}
\usepackage{nicefrac}
\usepackage{authblk}
\usepackage{textcomp}
\usepackage{graphicx}
\usepackage{framed} \usepackage{color}  

\usepackage[pagebackref]{hyperref} 
    \definecolor{darkblue}{rgb}{0,0,.85} 
    \definecolor{darkred}{rgb}{0.84,0,0}
    \hypersetup{colorlinks = true,
            linkcolor = darkblue,
            urlcolor  = darkblue,
            citecolor = darkred,
            anchorcolor = darkblue,
            bookmarksdepth=3}

\usepackage{accents}

\usepackage[shortlabels]{enumitem}

\usepackage[foot]{amsaddr}

\usepackage{bm}

\usepackage{stmaryrd}

\usepackage{graphics}

\usepackage{mathtools}

\newtheorem{thm}{Theorem}[section]
\newtheorem{prop}[thm]{Proposition}
\newtheorem{thmi}{Theorem}
\newtheorem{propi}{Proposition}
\newtheorem{lem}[thm]{Lemma}

\newtheorem{cor}[thm]{Corollary}

\newtheorem{construction}[thm]{Construction}

\theoremstyle{definition}
\newtheorem{example}[thm]{Example}

\newtheorem{nota}[thm]{Notation}

\theoremstyle{plain}

\newtheorem*{thmn}{Theorem}
\newtheorem*{corn}{Corollary}

\makeatletter
\newtheoremstyle{examplestyle}
  {1em}
  {1em}
  {\addtolength{\@totalleftmargin}{1.0em}
   \addtolength{\linewidth}{-1.0em}
   \parshape 1 1.0em \linewidth}
  {}
  {\bfseries}
  {.}
  {.5em}
  {}

\theoremstyle{examplestyle}

\newtheorem{rem}[thm]{Remark}
\newtheorem{remi}{Remark}

\newtheorem{defn}[thm]{Definition}

\DeclareMathOperator{\Spec}{Spec}

\DeclareMathOperator{\Gal}{Gal}
\DeclareMathOperator{\id}{id}

\DeclareMathOperator{\Hom}{Hom}

\DeclareMathOperator{\Spf}{Spf}

\DeclareMathOperator{\Aut}{Aut} 
\DeclareMathOperator{\Fil}{Fil}

\DeclareMathOperator{\Spa}{Spa}

\DeclareMathOperator{\Sht}{Sht}
\DeclareMathOperator{\Spd}{Spd}

\DeclareMathOperator{\Frob}{Frob}
\DeclareMathOperator{\Gr}{Gr}
\DeclareMathOperator{\GL}{GL}

\newcommand{\sht}{\mathrm{sht}}

\newcommand{\colim@}[2]{%
  \vtop{\m@th\ialign{##\cr
    \hfil$#1\operator@font colim$\hfil\cr
    \noalign{\nointerlineskip\kern1.5\ex@}#2\cr
    \noalign{\nointerlineskip\kern-\ex@}\cr}}%
}
\newcommand{\colim}{%
  \mathop{\mathpalette\colim@{}}\nmlimits@
}

\newcommand{\B}{\mathrm B}

\newcommand{\F}{\mathbb{F}}

\newcommand{\bb}[1]{\mathbb{#1}}

\newcommand{\mc}[1]{\mathcal{#1}}
\newcommand{\mbb}[1]{\mathbb{#1}}

\newcommand{\mf}[1]{\mathfrak{#1}}

\usepackage{relsize}

\newcommand{\wh}{\widehat}
\newcommand{\wt}{\widetilde}

\newcommand{\GVect}{\mathcal{G}\text{-}\mathbf{Vect}}

\newcommand{\GLoc}{\mathcal{G}\text{-}\mathbf{Loc}}
\newcommand{\GRflx}{\mathcal{G}\text{-}\mathbf{Rflx}}
\newcommand{\GSht}{\mathcal{G}\text{-}\mathbf{Sht}}
\newcommand{\GShtmu}{\mathcal{G}\text{-}\mathbf{Sht}_{\bm{\mu}}}

\newcommand{\hyphen}{\mathchar`-}

\newcommand{\Frac}{\mathrm{Frac}}
\newcommand{\ur}{\mathrm{ur}}

\newcommand{\Z}{\mathbb{Z}}

\newcommand{\Q}{\mathbb{Q}}

\renewcommand{\ll}{\llbracket}
\newcommand{\rr}{\rrbracket}

\newcommand{\lt}{\mathrm{lt}}
\newcommand{\rflx}{\mathrm{rflx}}

\makeatletter
\renewcommand{\email}[2][]{%
  \ifx\emails\@empty\relax\else{\g@addto@macro\emails{,\space}}\fi%
  \@ifnotempty{#1}{\g@addto@macro\emails{\textrm{(#1)}\space}}%
  \g@addto@macro\emails{#2}%
}
\makeatother

\newcommand{\triv}{\mathrm{triv}}

\newcommand{\stacks}[1]{\cite[\href{https://stacks.math.columbia.edu/tag/#1}{Tag~#1}]{StacksProject}}

\newcommand{\aff}{\mathrm{aff}}

\newcommand{\an}{\mathrm{an}}

\newcommand{\cat}[1]{\mathbf{#1}}

\DeclareMathOperator{\Isom}{Isom}
\newcommand{\uIsom}{\underline{\Isom}}
\newcommand{\ms}[1]{\mathscr{#1}}

\newcommand{\fl}{\mathrm{fl}}
\newcommand{\Zar}{\mathrm{Zar}}
\newcommand{\fpqc}{\mathrm{fpqc}}

\newcommand{\proet}{\mathrm{pro\acute{e}t}}
\newcommand{\qproet}{\mathrm{qpro\acute{e}t}}
\newcommand{\Isoc}{\mathbf{Isoc}}

\newcommand{\perf}{\mathrm{perf}}

\def\Item(#1){\item[\llap{(}\refstepcounter{enumi}$\bullet$] #1)}

\DeclareMathOperator*{\twolim}{2-lim}

\DeclareMathAccent{\wtilde}{\mathord}{largesymbols}{"65}

\newcommand*\isomto{%
        \xrightarrow{\raisebox{-0.2 em}{\smash{\ensuremath{\sim}}}}%
    }

    \newcommand{\ov}[1]{\overline{#1}}
    \newcommand{\et}{\mathrm{\acute{e}t}}

    \newcommand{\Qsyn}{\mathrm{QSYN}}
    \newcommand{\qsyn}{\mathrm{qsyn}}
    \newcommand{\pris}{\mathrm{pris}}
    \newcommand{\PRIS}{\mathrm{PRIS}}
    \newcommand{\qrsp}{\mathrm{qrsp}}
    \newcommand{\perfd}{\mathrm{perfd}}
    
    \newcommand{\Ainf}{\mathrm{A}_\mathrm{inf}}
    \newcommand{\Acrys}{\mathrm{A}_\mathrm{crys}}
    \newcommand{\Bcrys}{\mathrm{B}_\mathrm{crys}}
    \newcommand{\Bdr}{\mathrm{B}_\mathrm{dR}}
    
   \newcommand{\defeq}{\vcentcolon=}
    \newcommand{\efdeq}{=\vcentcolon}
    \newcommand{\be}{\begin{equation*}}
    \newcommand{\ee}{\end{equation*}}
    \newcommand{\bx}{\begin{equation*}\xymatrix}
    \newcommand{\ex}{\end{equation*}}

\usepackage{relsize}
\usepackage[bbgreekl]{mathbbol}
\usepackage{amsfonts}

\DeclareSymbolFontAlphabet{\mathbbl}{bbold}
\newcommand{\Prism}{{\mathlarger{\mathbbl{\Delta}}}}
\newcommand{\prism}{{\mathlarger{\mathbbl{\Delta}}}}

\newcommand{\smallprism}{{{\mathsmaller{\Prism}}}}
\newcommand{\smallN}{{{\mathsmaller{\mc{N}}}}}

\renewcommand{\inf}{\mathrm{inf}}

\newcommand{\crys}{\mathrm{crys}}
\newcommand{\strcrys}{{{\mathsmaller{\Prism}}\text{-}\mathrm{gr}}}
\newcommand{\mr}{\mathrm}
\newcommand{\ad}{\mathrm{ad}}

\newcommand{\dR}{\mathrm{dR}}

\title{A Tannakian framework for prismatic $F$-crystals}
\author{Naoki Imai$^{(1)}$}
\address[1]{\scriptsize Graduate School of Mathematical Sciences, The University of Tokyo,
    3-8-1 Komaba, Meguro-ku, Tokyo, 153-8914, Japan}
\email[1]{\scriptsize naoki@ms.u-tokyo.ac.jp}
\author{Hiroki Kato$^{(2)}$}
\address[2]{\scriptsize 
Institut des Hautes Études
Scientifiques, 35 route de Chartres, 91440 Bures-sur-Yvette, France
}
\email[2]{\scriptsize hiroki@ihes.fr}
\author{Alex Youcis$^{(3)}$}

\address[3]{\scriptsize University of Toronto, Bahen Centre, Room 6290
40 St. George St., Toronto, ON, M5S 2E4, Canada}

\email[3]{\scriptsize alex.youcis@gmail.com}

\date{\today}

\begin{document}
\begin{abstract}
    We develop the Tannakian theory of (analytic) prismatic $F$-crystals on a smooth formal scheme $\mf{X}$ over the ring of integers of a discretely valued field with perfect residue field. Our main result gives an equivalence between the $\mc{G}$-objects of prismatic $F$-crystals on $\mf{X}$ and $\mc{G}$-objects on a newly-defined category of $\Z_p$-local systems on $\mf{X}_\eta$: those of \emph{prismatically good reduction}. Additionally, we develop a \emph{shtuka realization functor} for (analytic) prismatic $F$-crystals on $p$-adic (formal) schemes and show it satisfies several compatibilities with previous work on the Tannakian theory of shtukas over such objects.
\end{abstract}

\maketitle

\tableofcontents

\section*{Introduction}

Since its inception, (integral) $p$-adic Hodge theory has provided an immensely powerful tool for studying how the $p$-adic geometry of objects varies in families. Central to this idea is the notion of crystalline $\Z_p$-local systems on $\mf{X}_\eta$, where $\mf{X}$ is a smooth formal scheme over the ring of integers $\mc{O}_K$ 
of a discretely valued $p$-adic field,\footnote{More precisely, $K$ is a discretely valued field of mixed characteristic $(0,p)$ and which has perfect residue field.} which promises to capture the idea of those $\Z_p$-local systems which have `good reduction' relative to $\mf{X}$. This guiding principle has recently been put on very firm footing via the introduction of the category $\cat{Vect}^\varphi(\mf{X}_\Prism)$ of \emph{prismatic $F$-crystals on $\mf{X}$} in the work \cite{BhattScholzeCrystals} of Bhatt and Scholze, and related objects, which have been shown to give `models' for crystalline $\Z_p$-local systems.

Additionally, the category $\cat{Vect}^\varphi(\mf{X}_\Prism)$ also provides a good approximation to (what should be) the category of $\Z_p$-motives over $\mf{X}$. This makes them an attractive notion for the study of integral (local) Shimura varieties, which ought to parameterize such motivic objects. That said, to rigorously utilize prismatic $F$-crystals in applications to Shimura varieties, one really needs a \emph{Tannakian theory}: a good theory of the category of $\mc{G}$-objects in $\cat{Vect}^\varphi(\mf{X}_\Prism)$ and related categories, where $\mc{G}$ is initially any smooth group scheme over $\Z_p$, but is assumed reductive for many of the main results of this article.

The goal of this paper is to provide such a Tannakian theory in some cases, with our main theorem being the following equivalence of categories.

\begin{thmi}[{see Theorem \ref{thm:equiv-G-Vect-and-G-strcrys}}]\label{thmi:main} If $\mc{G}$ is reductive, then there is an equivalence of categories
\begin{equation*}
T_\et\colon \GVect^\varphi(\mf{X}_\Prism)\isomto \GLoc_{\Z_p}^\strcrys(\mf{X}_\eta)\subseteq \GLoc_{\Z_p}^\crys(\mf{X}_\eta),
\end{equation*}
where $\cat{Loc}_{\Z_p}^\strcrys(\mf{X}_\eta)$ is the category of \emph{$\Z_p$-local systems of prismatically good reduction}.
\end{thmi}

The applicability of Theorem \ref{thmi:main} to the study of integral Shimura varieties is not hypothetical. In fact, this article may be seen as a companion paper to \cite{IKY2} as well as providing some foundational results needed for \cite{IKY3}. There we show the existence of a prismatic ($F$-gauge) realization functor on integral Shimura varieties at hyperspecial level, which has several important consequences for the theory of such integral Shimura varieties (e.g., an analogue of the Serre--Tate theorem). Theorem \ref{thmi:main} plays a pivotal role in \cite{IKY2}.

In the rest of this introduction we explicate the difficulty in proving Theorem \ref{thmi:main} and discuss related Tannakian results established in this paper.

\medskip

\paragraph{The Tannakian theory of (analytic) prismatic $F$-crystals} Throughout the following we fix a mixed characteristic complete discrete valuation ring $\mc{O}_K$ with perfect residue field $k$ of characteristic $p$, and a smooth formal $\mc{O}_K$-scheme $\mf{X}$ (see Notation and conventions for our convention on smoothness) with generic fiber $X$. 

In \cite{BhattScholzeCrystals}, Bhatt and Scholze define a category $\cat{Vect}^\varphi(\mf{X}_\Prism)$ of prismatic $F$-crystals on $\mf{X}$, which presents the correct notion of a `deformation' of an $F$-crystal on $\mf{X}_k$ to $\mf{X}$. They further construct a functor $T_\et\colon \cat{Vect}^\varphi(\mf{X}_\Prism)\to\cat{Loc}_{\Z_p}(X)$, the \'etale realization functor, which is a prismatic analogue of the Riemann--Hilbert correspondence. When $\mf{X}=\Spf(\mc{O}_K)$, they show that $T_\et$ induces an equivalence between $\cat{Vect}^\varphi(\mf{X}_\Prism)$ and the category $\cat{Loc}^\crys_{\Z_p}(X)$ of crystalline $\Z_p$-local systems on $X$ (i.e.\@, $\Z_p$-lattices $\bb{L}$ in crystalline $\Q_p$-local systems).

In general, $\cat{Vect}^\varphi(\mf{X}_\Prism)$ is not sufficient to recover every crystalline lattice on $X$ (see \cite[Example 3.36]{DLMS}), and Guo and Reinecke in \cite{GuoReinecke} consider an enlargement $\cat{Vect}^{\an,\varphi}(\mf{X}_\Prism)$ consisting of so-called analytic prismatic $F$-crystals. The functor $T_\et$ extends to this larger category, and \cite{GuoReinecke} showed that the \'etale realization functor $T_\et$ forms an equivalence between $\cat{Vect}^{\an,\varphi}(\mf{X}_\Prism)$ and $\cat{Loc}^\crys_{\Z_p}(X)$ (cf.\@ the results of \cite{DLMS}). 

As above, for applications of these $p$-adic Hodge theoretic ideas to the study of integral (local) Shimura varieties it is important to have Tannakian version of these results. To this end, let us fix a smooth affine group $\Z_p$-scheme $\mc G$, and for an exact $\Z_p$-linear $\otimes$-category $\mc{C}$ (see \S\ref{section:torsors via tensors}) we define $\mc{G}\text{-}\mc{C}$ to be the category of $\mc{G}$-objects in $\mc{C}$ (i.e., exact $\Z_p$-linear $\otimes$-functors $\omega\colon \cat{Rep}_{\Z_p}(\mc{G})\to \mc{C}$).

\begin{propi}[{see Propositions \ref{prop:crystalline-can-be-tested-on-faithful-rep} and \ref{prop:GR-DLMS-exact}}]\label{propi:GR-DLMS-exact} Let $\mf{X}$ be a smooth formal $\mc{O}_K$-scheme and $\mc{G}$ a smooth affine group $\Z_p$-scheme.
\begin{enumerate} 
\item Both $T_\et\colon \cat{Vect}^{\varphi,\an}(\mf{X}_\Prism)\isomto \cat{Loc}_{\Z_p}^\crys(X)$ and its quasi-inverse are exact. 
\item There is an equivalence
\begin{equation*}
T_\et\colon \GVect^{\varphi,\an}(\mf{X}_\Prism)\isomto \GLoc^\crys_{\Z_p}(X),\quad \omega\mapsto T_\et\circ\omega.
\end{equation*}
\item An object $\omega$ of $\GLoc_{\Z_p}(X)$ lies in the full subcategory $\GLoc^{\crys}_{\Z_p}(X)$ if and only if $\omega(\Lambda_0)$ is crystalline for one faithful representation $\mc{G}\hookrightarrow \GL(\Lambda_0)$.
\end{enumerate}
\end{propi}

That said, the difference between $\cat{Vect}^\varphi(\mf{X}_\Prism)$ and $\cat{Vect}^{\an,\varphi}(\mf{X}_\Prism)$ suggests a strengthening of the notion of crystalline. Namely, let $\cat{Loc}^\strcrys_{\Z_p}(X)$ be the category of \emph{prismatically good reduction $\Z_p$-local systems}: those crystalline $\Z_p$-local systems $\bb{L}$ such that $T_\et^{-1}(\bb{L})$ is a prismatic $F$-crystal. In fact, it is this category of prismatically good reduction $\Z_p$-local systems which plays
a more important role in many applications to Shimura varieties (e.g., see \cite{IKY2}) as it is $\cat{Vect}^\varphi(\mf{X}_\Prism)$ which is closer to motivic $\Z_p$-objects over $\mf{X}$, e.g., $p$-divisble groups (see \cite{AnschutzLeBrasDD}).

That said, the analogue of Proposition \ref{propi:GR-DLMS-exact} is far more subtle in this case. Namely, while the functor $T_\et\colon \cat{Vect}^\varphi(\mf{X}_\Prism)\to\cat{Loc}^\strcrys_{\Z_p}(X)$ is an exact $\Z_p$-linear $\otimes$-equivalence, its quasi-inverse is not exact, as it involves the extension of a vector bundle on an open subset of some space over its non-trivial closed complement (e.g., see \cite[Example 4.1.4]{LiuDifferentUnif}). So, there is no a priori reason for $T_\et^{-1}\circ\nu\colon \cat{Rep}_{\Z_p}(\mc{G})\to\cat{Vect}^\varphi(\mf{X}_\Prism)$ to be exact for $\nu$ in $\GLoc^\strcrys_{\Z_p}(X)$.

Despite this, a Tannakian equivalence still holds, at least if $\mc{G}$ is reductive, as in the following result which, in particular, refines Theorem \ref{thmi:main}. 

\begin{thmi}[{see Theorem \ref{thm:equiv-G-Vect-and-G-strcrys} and Corollary \ref{cor:strcrys-faithful-condition}}]\label{thmi:Tannakian} Let $\mf{X}$ be a smooth formal $\mc{O}_K$-scheme and $\mc{G}$ a reductive group $\Z_p$-scheme. 
\begin{enumerate}
    \item The functor
\begin{equation*}
T_\et\colon \GVect^\varphi(\mf{X}_\Prism)\to\GLoc^\strcrys_{\Z_p}(X),\qquad \omega\mapsto T_\et\circ\omega,
\end{equation*}
is an equivalence, i.e.\@, $T_\et^{-1}\circ\nu$ is exact for any $\nu$ in $\GLoc^\strcrys_{\Z_p}(X)$. 
\item An object $\omega$ of $\GLoc_{\Z_p}(X)$ belongs to the subcategory $\GLoc^\strcrys_{\Z_p}(X)$ if and only if $\omega(\Lambda_0)$ is of prismatically good reduction for some faithful representation $\mc{G}\hookrightarrow \GL(\Lambda_0)$.
\end{enumerate}
\end{thmi}

We further remark that both $\GVect^\varphi(\mf{X}_\Prism)$ and $\GLoc_{\Z_p}(X)$ may be interpreted in terms of torsors (with extra structure, e.g., a Frobenius $\varphi$). In particular, there are equivalences
\begin{equation*}
    \GVect^\varphi(\mf{X}_\Prism)\simeq \cat{Tors}^\varphi_\mc{G}(\mf{X}_\Prism),\qquad \GLoc_{\Z_p}(X)\simeq \cat{Tors}_{\mc{G}(\Z_p)}(X_\proet),
\end{equation*}
where $X_\proet$ is Scholze's pro-\'etale site on rigid spaces as in \cite{ScholzepadicHT} (see Propositions \ref{prop:varphi-equivariant-GVect-and-tors-identification} and \ref{prop:G-for-proet-adic}, respectively). Thus, Theorem \ref{thmi:Tannakian} can be interpreted as an equivalence of categories
\begin{equation*}
    T_\et\colon \cat{Tors}_\mc{G}^\varphi(\mf{X}_\prism)\isomto \cat{Tors}_{\mc{G}(\Z_p)}^\strcrys(X_\proet)\subseteq \cat{Tors}_{\mc{G}(\Z_p)}(X_\proet)
\end{equation*}
where the target of the first map is the full subcategory of $\mc{G}(\Z_p)$-torsors $\mc{P}$ on $X_\proet$ such that $\rho_\ast(\mc{P})$ is a $\Z_p$-local system of prismatically good reduction for any (equiv.\@, one faithful) representation $\rho\colon \mc{G}\to\GL(\Lambda)$.

\begin{remi} The main technical result needed to prove Theorem \ref{thmi:Tannakian} may be proven independently of many of the main results of \cite{GuoReinecke} and \cite{DLMS}, and in greater generality (e.g., including power series rings), using an adaptation of an idea of Kisin. See \S\ref{ss:complementary-results} for details.
\end{remi}

\begin{remi} A natural question is how the Tannakian theory developed here extends and interacts with the stacks $X^\smallprism$, $X^\smallN$, and $X^\mr{syn}$ developed by Drinfeld and Bhatt--Lurie. This question is largely answered by combining the contents of this paper with \cite[\S1]{IKY2}.
\end{remi}

\medskip

\paragraph{Relation to the Tannakian theory of shtukas} With applications to integral Shimura varieties in mind, it is important to understand the relationship between the Tannakian theory of (analytic) prismatic $F$-crystals and the Tannakian theory of shtukas. Indeed, much of the recent work on studying the $p$-adic Hodge theory of integral Shimura varieties, notably \cite{PappasRapoportI}, uses the theory of shtukas (as developed in \cite{ScholzeBerkeley}) in a central way.

More precisely, recall that in \cite{PappasRapoportI} there is developed a theory of shtukas over $\ms{X}$, where $\ms{X}$ is either a (formal) $\Z_p$-scheme or a $\Q_p$-scheme. Roughly, such an object is a morphism of $v$-stacks
\begin{equation*}
    \ms{X}^?\to \GSht,
\end{equation*}
where the target is the $v$-stack of $\mc{G}$-shtukas and $\ms{X}^?$ is a certain $v$-sheaf associated to $\ms{X}$.\footnote{The notation $?$ is due to the fact that exactly which $v$-sheaf one takes depends on precisely what type of object $\ms{X}$ is (e.g., it differs between $\Q_p$-schemes and $\Z_p$-schemes).}

Suppose now that $\ms{X}=X$ is a smooth $\Q_p$-scheme. For a smooth affine group $\Z_p$-scheme $\mathcal{G}$, there is developed in op.\@ cit.\@ a functor 
\begin{equation*}
    U_\mr{sht}\colon \GLoc^{\mr{dR}}_{\Z_p}(X^\mr{an})\to \GSht(X),
\end{equation*}
where the source is the category of de Rham $\mc{G}(\Z_p)$-local systems on $X^\mr{an}$ (i.e., $\mc{G}(\Z_p)$-lattices in de Rham $\mc{G}(\Q_p)$-local systems on $X^\mr{an}$).

The following result shows that there is a precise way to realize a $\mc{G}$-object in analytic prismatic $F$-crystals in the category of $\mc{G}$-shtukas, and that the functor $U_\mr{sht}$ intertwines this shtuka realization with the \'etale realization.

\begin{thmi}[{see Theorem \ref{thm:prismatic-F-crystal-shtuka-relationship}}]\label{thmi:shtuka-realization} Let $\mf{X}$ be a smooth formal $\mc{O}_K$-scheme and $\mc{G}$ a smooth affine group $\Z_p$-scheme. Then, there is a \emph{shtuka realization functor}
\begin{equation*}
    T_\mr{sht}\colon \GVect^{\varphi,\an}(\mf{X}_\Prism)\to \GSht(\mf{X}),
\end{equation*}
and a comparison isomorphism
\begin{equation*}
    \varrho\colon (T_\mr{sht})_\eta\isomto U_\mr{sht}\circ T_\et.
\end{equation*}
\end{thmi}

Using the comparison isomorphism in Theorem \ref{thmi:shtuka-realization} we can actually upgrade our shtuka realization functor from smooth formal $\mc{O}_K$-schemes $\mf{X}$ to also include smooth $\mc{O}_K$-schemes $\ms{X}$. Namely, as in Definition \ref{defn:prismatic-model}, let $\GVect^{\varphi,\an}(\ms{X}_\Prism)$ be the category of $\mc{G}$-objects in analytic prismatic $F$-crystals on $\ms{X}$. Essentially by definition, an object of this category consists of a triple $(\omega_\et,\omega_\Prism,\iota)$ where
\begin{itemize}
    \item $\omega_\et$ is a de Rham $\mc{G}(\Z_p)$-local system on $\ms{X}_K^\mr{an}$,
    \item $\omega_\Prism$ is an object of $\GVect^{\varphi,\an}(\wh{\ms{X}}_\Prism)$ (with $\wh{\ms{X}}$ the $p$-adic completion of $\ms{X}$),
    \item $\iota\colon T_\et\circ \omega_\Prism\isomto \omega_\et|_{\wh{\ms{X}}_\eta}$ is an isomorphism in $\GLoc_{\Z_p}(\wh{\ms{X}}_\eta)$.
\end{itemize}
We then obtain a \emph{shtuka realization functor}
\begin{equation*}
    T_\mr{sht}\colon \GVect^{\varphi,\an}(\ms{X}_\Prism)\to \GSht(\ms{X}),
\end{equation*}
given by the rule 
\begin{equation*} 
T_\mr{sht}(\omega_\et,\omega_\Prism,\iota)=(U_\sht(\omega_\et),T_\sht(\omega_\Prism),U_\sht (\iota) \circ \varrho_{\omega_\Prism}),
\end{equation*}
where $\varrho_{\omega_\Prism}$ is the isomorphism induced by the comparison isomorphism $\varrho$ from Theorem \ref{thmi:shtuka-realization}.
In \cite{IKY2}, Theorem \ref{thmi:shtuka-realization} plays an important role in relating prismatic realization functors to shtuka realization functors, especially in proving the Pappas--Rapoport conjecture on the shtuka realization functors for Shimura varieties of abelian type at hyperspecial level. 


\medskip

\paragraph{A remark on reductive hypotheses}
As discussed above, while many of the results in this paper only require $\mathcal{G}$ to be a smooth affine group $\mathbb{Z}_p$-schemes, many of the main results (e.g., Theorem \ref{thmi:main}) require $\mathcal{G}$ to be reductive.

Pivotally, this makes use of an observation of Colliot-Th\'el\`ene--Sansuc. To state it precisely, let $X$ be an integral Noetherian scheme, $U\subseteq X$ an open subscheme such that $X-U$ has depth at least $2$, $\mc{G}$ a smooth affine group $X$-scheme, and $\rho\colon\mathcal{G}\hookrightarrow \GL_n$ a faithful representation. 

We then have the following result (see \S\ref{ss:reflexive-pseudo-torsors} for a more comprehensive discussion).

\begin{thmn}[Colliot-Th\'el\`ene--Sansuc] Let notation be as above, and assume that $\mathcal{G}$ is a reductive group $X$-scheme. Then, if $\mc{P}$ is a $\mathcal{G}$-torsor on $U$ such that the vector bundle $\rho_\ast\mathcal{P}$ extends to $X$ as a vector bundle, then the $\mc{G}$-torsor $\mathcal{P}$ extends to $X$.
\end{thmn}
Ultimately, this theorem boils down to an application of Hartog's lemma and the fact that the reductivity of $\mathcal{G}$ implies that the quotient $\GL_n/\mathcal{G}$ is affine. As a specific example of this, if $(A,\mf{m})$ is a regular local ring of dimension $2$, then the condition on $\rho_\ast\mathcal{P}$ is always satisfied by the Auslander--Buchsbaum formula, and so we obtain the following corollary. 

\begin{corn}Suppose that $(A,\mf{m})$ is a two-dimensional regular local ring and $\mc{G}$ a reductive group $A$-scheme. Then, any $\mathcal{G}$-torsor on $\Spec(A)-\{\mf{m}\}$ extends to a $\mathcal{G}$-torsor on $\Spec(A)$.
\end{corn}

It is unreasonable to expect the theorem of Colliot-Th\'el\`ene--Sansuc to hold when $\mathcal{G}$ is a general smooth affine group $X$-scheme. Namely, write $Y=\mathrm{GL}_n/\mathcal{G}$ for the quotient algebraic space. In this generality the best one can hope for is that $Y$ is quasi-affine (e.g., see \cite[Corollary 11.7]{PappasZhu} for the case considered in the above corollary), and so we assume this and set $Y^\mathrm{aff}=\Spec(\mc{O}(Y))$. Then, any section of $U\to Y$ extends to a section of $X\to Y^\mr{aff}$, but there is no reason to believe this avoids the boundary $Y^\mr{aff}-Y$. In fact, by considering the universal case $U=Y$ and $X=Y^\mathrm{aff}$, one may show that the theorem of Colliot-Th\'el\`ene--Sansuc fails whenever $\mathcal{G}$ is not reductive. 

That said, from the perspective of Shimura varieties, the most interesting case beyond the reductive case is when $\mathcal{G}$ is a parahoric group scheme over $\mathbb{Z}_p$. In this case, while the theorem of Colliot-Th\'el\`ene--Sansuc fails to hold, we expect that the above-stated corollary concerning regular local rings of dimension $2$ \emph{does hold}. Strong evidence in this direction is provided by \cite{Anschutz}, which shows that the result does hold for certain $A$, notably $A=\mathfrak{S}_\mathcal{O}$ for mixed characteristic discrete valuation rings $\mathcal{O}$. 

Finally, while the cases of this observation of Colliot-Th\'el\`ene--Sansuc relevant to proving Theorem \ref{thmi:main} are not two-dimensional, it may still be possible to bootstrap from the two-dimensional case (or even the more specific result of Ansch\"{u}tz) to prove such an analogue of Theorem \ref{thmi:main} at parahoric level.



\medskip

\paragraph{Acknowledgments} The authors would like to heartily thank Abhinandan, Piotr Achinger, Bhargav Bhatt, Alexander Bertoloni Meli, K\k{e}stutis \v{C}esnavi\v{c}ius, Patrick Daniels, Ofer Gabber, Ian Gleason, Haoyang Guo, Kentaro Inoue, Mark Kisin, Emanuel Reinecke, Peter Scholze, and Koji Shimizu, for helpful discussions and comments. They would also like to thank an anonymous referee for many helpful comments which improved the readability of the paper. Part of this work was conducted during a visit to the Hausdorff Research Institute for Mathematics, funded by the Deutsche Forschungsgemeinschaft (DFG, German Research Foundation) under Germany's Excellence Strategy – EXC-2047/1 – 390685813. This work was supported by JSPS KAKENHI Grant Numbers 22KF0109, 22H00093 and 23K17650, the European Research Council (ERC) under the European Union’s Horizon 2020 research and innovation programme (grant agreement No. 851146), and funding through the Max Planck Institute for Mathematics in Bonn, Germany (report numbers MPIM-Bonn-2022, MPIM-Bonn-2023, MPIM-Bonn-2024). 

\medskip

\paragraph{Notation and conventions}

\begin{itemize}
    \item The symbol $p$ will always denote a (rational) prime.
    \item All rings are assumed commutative and unital unless stated otherwise.
    \item Reductive group schemes are assumed (by definition) to have connected fibers.
    \item Formal schemes are assumed to have a locally finitely generated ideal sheaf of definition.
    \item For a property $P$ of morphisms of schemes, an adic morphism of formal schemes $\mf{X}\to\mf{Y}$, where $\mf{Y}$ has an ideal sheaf of definition $\mc{I}$, is \emph{adically $P$ (or $\mc{I}$-adically $P$)} if the reduction modulo $\mc{I}^n$ is $P$ for all $n$. If $A\to B$ is an adic morphism of rings with the $I$-adic topology, for $I\subseteq A$ an ideal, then we make a similar definition.
    \item For \'etaleness/smoothness of a morphism of formal schemes, we follow the conventions in \cite[Chapter I, \S5.3.(b)--(c)]{FujiwaraKato}. In particular, smooth morphisms are adic. 
    \item For a non-archimedean field $K$, a \emph{rigid $K$-space} $X$ is an adic space locally of finite type over $K$. We denote the set of \emph{classical points} by $|X|^\mathrm{cl}:=\left\{x\in X: [k(x):K]<\infty\right\}$.
    \item For two categories $\ms{C}$ and $\ms{D}$, the notation $(F,G)\colon \ms{C}\to\ms{D}$ means a pair of functors $F\colon \ms{C}\to \ms{D}$ and $G\colon\ms{D}\to\ms{C}$, with $F$ being right adjoint to $G$.
    \item For an $R$-module $M$ and an ideal $I\subseteq R$ we write $M/I$ as shorthand for $M/IM$.
    \item A filtration always means a decreasing, separated, and exhaustive $\Z$-filtration. 
    \item For a ring $A$ which is $a$-adically separated for $a$ in $A$, denote by $\Fil^\bullet_a$ the filtration with $\Fil^r_a=a^r A$ for $r> 0$, and $\Fil^r_a=A$ for $r\leqslant 0$. Define $\Fil^\bullet_\triv\defeq \Fil^\bullet_0$.
    \item A filtration of modules (of sheaves) is \emph{locally split} if its graded pieces are locally free.
    \item For an $\bb{F}_p$-algebra $R$ (resp.\@ $\bb{F}_p$-scheme $X$), we denote by $F_R$ (resp.\@ $F_X$) the absolute Frobenius of $R$ (resp.\@ $X$).
\end{itemize}

\section{The Tannakian framework for prismatic objects}\label{s:G-objects-prismatic-crystals}

In this section we discuss the fundamental Tannakian theory of (analytic) prismatic $F$-crystals on a quasi-syntomic $p$-adic formal scheme $\mf{X}$. See Appendix \ref{s:Tannakian-appendix} for our conventions concerning, topoi, formal schemes, and Tannakian formalism.

\medskip

\begin{nota}\label{nota:Section-1} We fix the following notation:
\begin{itemize}
\item $k$ is a perfect field extension of $\bb{F}_p$, $W\defeq W(k)$, and $K_0\defeq \mr{Frac}(W)$,
\item $K$ is a finite totally ramified extension of $K_0$, with ring of integers $\mc{O}_K$ and ramification index $e$,
\item $\pi$ is a uniformizer of $K$, which we take to be $p$ if $K=K_0$, and $E=E(u)$ in $W[u]$ is the minimal polynomial for $\pi$ over $K_0$,
\item $\ov{K}$ is an algebraic closure of $K$ and $C$ is its $p$-adic completion,
\item $\pi^\flat$ and $p^\flat$ in $C^\flat$ are as in \cite[Lemma 6.2.2]{ScholzeBerkeley},
\item $\varepsilon\defeq (1,\zeta_p,\ldots)$ in $C^\flat$ is a compatible system of $p^\text{th}$-power roots of $1$, 
\item $q=[\varepsilon]$ in $\Ainf(\mc{O}_C)$, and $t\defeq\log(q)=-\sum_{n\geqslant 1} \frac{(1-q)^n}{n}$ an element of $\Acrys(\mc{O}_C)$, where $\Ainf(\mc{O}_C)$ and $\Acrys(\mc{O}_C)$ are as in \cite[\S1.2.3]{FontainePeriodes} and \cite[\S2.2]{FontainePeriodes}, respectively,
\item $\xi_0\defeq p-[p^\flat]$ and $\tilde{\xi}_0\defeq p-[p^\flat]^p$, elements of $\Ainf(\mc{O}_C)$,
\item $(\Lambda_0,\mathds{T}_0)$ is a tensor package over $\Z_p$, with $\mc{G}\defeq \mathrm{Fix}(\mathds{T}_0)$ smooth over $\Z_p$ (see \S\ref{section:torsors via tensors}),
\item $G\defeq \mc{G}_{\Q_p}$.
\end{itemize}
\end{nota}
\subsection{The absolute prismatic and quasi-syntomic sites}\label{ss:absolute-prismatic-and-qsyn-sites} 

We now record notation and basic results about the prismatic and quasi-syntomic sites developed in \cite{BMS-THH}, \cite{BhattScholzePrisms}, and \cite{BhattScholzeCrystals}.

\subsubsection{Prisms}\label{subsubsec: prisms}
 For a $\Z_{(p)}$-algebra $A$, a \emph{$\delta$-structure} is a map $\delta\colon A\to A$ with $\delta(0)=\delta(1)=0$ and
\begin{equation*}
    \delta(xy)=x^p\delta(y)+y^p\delta(x)+p\delta(x)\delta(y),\qquad \delta(x+y)=\delta(x)+\delta(y)+\tfrac{1}{p}(x^p+y^p-(x+y)^p).
\end{equation*}
Associated to $\delta$ is a Frobenius lift $\phi\colon A\to A$ given by $\phi(x)=x^p+p\delta(x)$, which we call the \emph{Frobenius}. If $A$ is $p$-torsion-free then any Frobenius lift $\phi$ on $A$ defines a $\delta$ structure by $\tfrac{1}{p}(\phi(x)-x^p)$, establishing a bijection between the two types of structures, and we conflate the two notions. We call the pair $(A,\delta)$ a \emph{$\delta$-ring}. We often suppress $\delta$ from the notation, writing $\delta_A$ (or $\phi_A$) when we want to be clear. A morphism of $\delta$-rings is a ring map that intertwines the $\delta$-structures.

A \emph{prism} is a pair $(A,I)$ where $A$ is a $\delta$-ring and $I\subseteq A$ is an invertible ideal with $A$ derived $(p,I)$-adically complete (see \cite[\S6.2]{BMSI}), and such that $I+\phi(I)$ contains $p$. Thus, $I$ is finitely generated and $\Spec(A)-V(I)$ is affine (see \stacks{07ZT}), and we denote by $A[\nicefrac{1}{I}]$ the global sections of the structure sheaf. Let $j_{(A,I)}$ denote the inclusion of $U(A,I)\defeq \Spec(A)-V(p,I)$ into $\Spec(A)$. For a prism $(A,I)$, unless stated otherwise, we view $A$ as being equipped with the $(p,I)$-adic topology. 

A prism $(A,I)$ is \emph{bounded} if $A/I$ has bounded $p^\infty$-torsion. The following two results will be used without comment in the sequel.

\begin{lem} Let $(A,I)$ be a bounded prism. Then, $A$ is $(p,I)$-adically complete, and $A/I^r$ is $p$-adically complete for all $r\geqslant 1$.
\end{lem}
\begin{proof}
    The first claim is precisely \cite[Lemma 3.7 (1)]{BhattScholzePrisms}. For the second claim, let $I^r=(d_1,\ldots,d_n)$. Then, $A/I^r=\mathrm{coker}(f)$, where $f\colon A^n\to A$ is given by $f(a_1,\ldots,a_n)=\sum_{i=1}^n d_ia_i$. As $A^n$ and $A$ are $p$-adically complete, we know by \cite[Lemma 3.4.14]{BhattScholzeProetale} that $A/I^r$ is derived $p$-adically complete. But, it is then $p$-adically complete by the argument of Ogus given in the comments of \stacks{0BKF}.
\end{proof}

A \emph{morphism} $(A,I)\to (B,J)$ is a morphism $A\to B$ of $\delta$-rings mapping $I$ into $J$. By the \emph{rigidity property of morphisms} of prisms (see \cite[Proposition 3.5]{BhattScholzePrisms}), if $(A,I)\to (B,J)$ is a morphism of prisms, then $I\otimes_A B$ maps isomorphically onto $J$ and, in particular, $J=IB$. A morphism $(A,I)\to (B,IB)$ is \emph{$(p,I)$-completely (faithfully) flat (resp.\@ \'etale, smooth}) when $B\otimes^L_A (A/(p,I))$ is concentrated in degree $0$ and $A/(p,I)\to B\otimes^L_A (A/(p,I))$ is (faithfully) flat (resp.\@ \'etale, smooth).

\begin{lem}\label{lem:adically-and-completely-agree} Let $(A,I)$ be a bounded prism. Let $B$ be a  $(p,I)$-adically complete $A$-algebra. Then 
$A \to B$ is $(p,I)$-completely (faithfully) flat (resp.\@ \'etale, smooth) if and only if $\Spf(B)\to\Spf(A)$ is adically (faithfully) flat (resp.\@ \'etale, smooth).
\end{lem}
\begin{proof} We put $J=IB$. If $(A,I)\to (B,J)$ is $(p,I)$-completely (faithfully) flat then \cite[Theorem 4.3]{YekutieliI} implies the map $A/(p,I)^n\to B/(p,J)^n$ is (faithfully) flat for all $n$, and so $\Spf(B)\to\Spf(A)$ is adically (faithfully) flat. Suppose that $\Spf(B)\to\Spf(A)$ is adically (faithfully) flat. Then \cite[Theorem 7.3]{YekutieliII} implies the ideal $(p,I)\subseteq A$ is weakly proregular in the sense of op.\@ cit.\@ Therefore, we deduce that $A\to B$ is $(p,I)$-completely (faithfully) flat by completeness of these modules and \cite[Theorem 6.9]{YekutieliI}. The second claim follows from this as $A\to B$ is $(p,I)$-completely \'etale (resp.\@ smooth) if and only if it is $(p,I)$-completely flat and $A/(p,I)\to B/(p,J)$ is \'etale (resp.\@ smooth), and $\Spf(B)\to\Spf(A)$ is adically \'etale if and only if it is adically flat and $A/(p,I)\to B/(p,J)$ is \'etale (resp.\@ smooth).
\end{proof}

\begin{prop}\label{prop:etale-prism-extension} Let $(A,I)$ be a bounded prism and $\Spf(B)\to\Spf(A)$ an adically \'etale map, where $B$ is $(p,I)$-adically complete. Then, there exists a unique $\delta$-structure on $B$ such that $(A,I)\to (B,IB)$ is a morphism of bounded prisms.
\end{prop}
\begin{proof}
By \cite[Lemma 2.18 and Lemma 3.7 (3)]{BhattScholzePrisms}, we obtain a unique $\delta$-structure so that $(A,I)\to (B,IB)$ is a map of prisms. Then by \cite[Lemma 3.4.14]{BhattScholzeProetale}, $B/IB$ is derived $p$-adically complete. Hence $(B,IB)$ is a bounded prism by \cite[Corollary 4.8 (1)]{BMS-THH}. 
\end{proof}

A prism $(A,I)$ is \emph{perfect} if $\phi_A$ is an isomorphism, in which case it is bounded (see \cite[Lemma 2.34]{BhattScholzePrisms}). For a perfectoid ring $R$, we have the perfect prism $(\Ainf(R),\ker(\theta_R))$ where $\Ainf(R)=W(R^\flat)$ is Fontaine's ring, which comes equipped with a natural Frobenius $\phi_R$, and $\theta_R\colon \Ainf(R)\to R$ is Fontaine's map. We also have the perfect prism $(\Ainf(R),\ker(\wt{\theta}_R))$, where $\wt{\theta}_R\defeq \theta_R\circ \phi^{-1}_R$, which is isomorphic to $(\Ainf(R),\ker(\theta_R))$ via $\phi_R$ (see Remark \ref{rem:choice-of-twist} to see why we introduce these two choices). We often fix a generator $\xi_R$ of $\ker(\theta_R)$ and set $\tilde{\xi}_R\defeq \phi_R(\xi_R)$ so that $\ker(\wt{\theta}_R)=(\tilde{\xi}_R)$. When $R$ is clear from context we shall omit the decoration of $R$ at all places. When $R$ is an $\mc{O}_C$-algebra, we may take $\xi=\xi_0$ and $\tilde\xi=\tilde\xi_0$ which we often implicitly do.

\subsubsection{The absolute prismatic site} Let $\mf{X}$ be a $p$-adic formal scheme. Consider the category $\mf{X}_\Prism^\mathrm{op}$ of triples $(A,I,s)$ where $(A,I)$ is a bounded prism, and  $s\colon \Spf(A/I)\to \mf{X}$ is a morphism, and where morphisms are maps of prisms commuting with the maps to $\mf{X}$. We often omit $s$ from the notation. The \emph{absolute prismatic site} $\mf{X}_
\Prism$ of $\mf{X}$ (see \cite[Definition 2.3]{BhattScholzeCrystals}) is the opposite category of $\mf{X}_\Prism^\mathrm{op}$, endowed with the topology where $\{\alpha_i\colon (A,I)\to (B_i,J_i)\}$ in $\mf{X}_\Prism^\mathrm{op}$ corresponds to a cover if $\{\Spf(\alpha_i)\colon \Spf(B_i)\to \Spf(A)\}$ is a cover in $\Spf(A)_{\fl}$ (see \S\ref{ss:torsors-and-vb-on-formal-schemes}). That this is a site follows from the argument in \cite[Corollary 3.12]{BhattScholzePrisms}, which also shows that for a diagram in $\mathfrak{X}_\Prism^\mathrm{op}$
\begin{equation*}
    (B,IB)\leftarrow (A,I)\to (C,IC),
\end{equation*}
with one of the maps adically faithfully flat, then its cofibered product is $(B\wh{\otimes}_A C,I(B\wh{\otimes}_A C))$ with the obvious $\delta$-structure and map to $\mf{X}$. We often abuse notation when working in $\mf{X}_\Prism$, writing objects and morphisms as in $\mf{X}_\Prism^\mathrm{op}$. We also shorten $\Spf(R)_\Prism$ to $R_\Prism$, in which case we often write $s\colon \Spf(A/I)\to \Spf(R)$ as $s\colon R\to A/I$.

By \cite[Corollary 3.12]{BhattScholzePrisms}, the presheaves $\mc{O}_{\mf{X}_\Prism}(A,I)\defeq A$ and $\ov{\mc{O}}_{\mf{X}_\Prism}(A,I)\defeq A/I$ are sheaves. By the rigidity property of morphisms of prisms, $\mc{I}_{\mf{X}_\Prism}(A,I)\defeq I$ is a quasi-coherent ideal sheaf of $\mc{O}_{\mf{X}_\Prism}$. The association
$\mc{O}_{\mf{X}_\Prism}[\nicefrac{1}{\mc{I}_\Prism}](A,I)\defeq A[\nicefrac{1}{I}]$ is a sheaf of rings as $A\to A[\nicefrac{1}{I}]$ is flat. Define
\begin{equation*}
     \mc{O}_{\mf{X}_\Prism}[\nicefrac{1}{\mc{I}_\Prism}]^{\wedge}_p\defeq \varprojlim_n  \mc{O}_{\mf{X}_\Prism}[\nicefrac{1}{\mc{I}_\Prism}]/p^n \mc{O}_{\mf{X}_\Prism}[\nicefrac{1}{\mc{I}_\Prism}]=R\varprojlim_n  \mc{O}_{\mf{X}_\Prism}[\nicefrac{1}{\mc{I}_\Prism}]/p^n \mc{O}_{\mf{X}_\Prism}[\nicefrac{1}{\mc{I}_\Prism}],
\end{equation*}
(see \cite[Remark 2.4]{BhattScholzeCrystals} and \cite[Proposition 3.1.10]{BhattScholzeProetale} for the second equality). The morphisms $\phi_A$ collectively form a morphism of sheaves of rings $\phi_{\mf{X}_\Prism}\colon \mc{O}_{\mf{X}_\Prism}\to \mc{O}_{\mf{X}_\Prism}$. We often omit the $\mf{X}$ from the above notation when it is clear from context, writing $\mc{O}_\Prism$, $\ov{\mc{O}}_\Prism$, $\mc{I}_\Prism$, $\mc{O}_\Prism[\nicefrac{1}{\mc{I}_\Prism}]$, $\mc{O}_\Prism[\nicefrac{1}{\mc{I}_\Prism}]^\wedge_p$, and $\phi$. 

For a morphism $f\colon \mf{X}\to \mf{Y}$, the association $((A,I),s)\mapsto ((A,I),f\circ s)$ is cocontinuous, so gives a morphism of topoi $(f_{\Prism\ast},f^\ast_\Prism)\colon \cat{Shv}(\mf{X}_\prism)\to \cat{Shv}(\mf{Y}_\Prism)$. 

\subsubsection{The quasi-syntomic site}\label{ss:qsyn-site} 
As in \cite{BMS-THH}, call a $p$-adically complete ring $R$ \emph{quasi-syntomic} if it has bounded $p^\infty$-torsion and the cotangent complex $L_{R/\Z_p}$ has $p$-complete Tor-amplitude in $[-1,0]$ (see \cite[Definition 4.1]{BMS-THH}), which we consider as having the $p$-adic topology. A map $R\to S$ of $p$-adically complete rings with bounded $p^\infty$-torsion is called a \emph{quasi-syntomic morphism (resp.\@ cover)} if it is adically flat (resp.\@ adically faithfully flat)\footnote{By \cite[Corollary 4.8]{BMS-THH}, $R\to S$ is adically (faithfully) flat if and only if it is $p$-completely (faithfully) flat.} and $L_{S/R}$ has $p$-complete Tor-amplitude in $[-1,0]$. By \cite[Proposition 4.19]{BMS-THH} a perfectoid ring $R$ is quasi-syntomic.

One extends these definitions to (maps of) $p$-adic formal schemes by working affine locally. For a quasi-syntomic $p$-adic formal scheme $\mf{X}$ the \emph{big (resp.\@ small) quasi-syntomic site} of $\mf{X}$, denoted $\mf{X}_\Qsyn$ (resp.\@ $\mf{X}_\qsyn$), has objects maps (resp.\@ quasi-syntomic maps) $\Spf(R)\to\mf{X}$ for $R$ a $p$-adically complete ring with bounded $p^\infty$-torsion, morphisms given by $\mf{X}$-morphisms, and covers given by quasi-syntomic covers (see \cite[Lemma 4.17]{BMS-THH}). 

The functor
\begin{equation*}
    u\colon \mf{X}_\Prism\to\mf{X}_\Qsyn,\qquad ((A,I),s)\mapsto s,
\end{equation*}
is cocontinuous (see \cite[Corollary 3.24]{AnschutzLeBrasDD}), and therefore gives rise to a morphism of topoi $(u_\ast,u^{-1})\colon \cat{Shv}(\mf{X}_\Prism) \to \cat{Shv}(\mf{X}_\Qsyn)$. The inclusion $\mf{X}_{\qsyn}\to \mf{X}_{\Qsyn}$, while continuous, may not induce a morphism of sites (see \cite[\S4.1]{AnschutzLeBrasDD}), but we may still consider the functor
\begin{equation*}
    v_\ast\colon \cat{Shv}(\mf{X}_\Prism)\to\cat{Shv}(\mf{X}_\qsyn), \qquad \mc{F}\mapsto u_\ast(\mc{F})|_{\mf{X}_\qsyn}.
\end{equation*}
Following \cite[Definition 4.1]{AnschutzLeBrasDD}, we use $\mc{O}^\pris_\mf{X}$ to denote $v_\ast(\mc{O}_{\mf{X}_\Prism})$ and $\mc{I}_\mf{X}^\pris$ for $v_\ast(\mc{I}_{\mf{X}_\prism})$. Define 
\begin{equation*}
\mc{O}^\pris_{\mf{X}}[\nicefrac{1}{\mc{I}^\pris_\mf{X}}](R)\defeq \mc{O}^\pris_{\mf{X}}(R)[\nicefrac{1}{\mc{I}^\pris_\mf{X}}(R)]=v_\ast(\mc{O}_{\mf{X}_\Prism}[\nicefrac{1}{\mc{I}_{\mf{X}_\Prism}}]). 
\end{equation*}
There is a morphism of sheaves of rings $\phi_\mf{X}^\pris\defeq v_\ast(\phi_{\mf{X}_\Prism})\colon \mc{O}^\mathrm{pris}_\mf{X}\to \mc{O}^\mathrm{pris}_\mf{X}$. When no confusion will arise, we omit $\mf{X}$ from notation, writing $\mc{O}^\pris$, $\mc{I}^\pris$, $\mc{O}_\pris[\nicefrac{1}{\mc{I}^\pris}]$, and $\phi$. There are also obvious analogues of these objects using $u_\ast$ in place of $v_\ast$, which we denote by $\mc{O}^\PRIS$, etc.

A quasi-syntomic ring $R$ is called \emph{quasi-regular semi-perfectoid} (see \cite[Definition 4.20 and Remark 4.22]{BMS-THH}), abbreviated \emph{qrsp}, if there exists a surjection $S\to R$ with $S$ a perfectoid ring. 

\begin{lem}\label{lem:flatness-preserves-boundedness}Let $R \to R'$ be a $p$-adically flat morphism of $p$-adically complete rings. If $N\geqslant 1$ is such that $R[p^N]=R[p^\infty]$, then $R'[p^N]=R'[p^\infty]$.
\end{lem}
\begin{proof}
Let $r'$ be in $R'[p^{\infty}]$ and $n \geqslant N$ such that $p^nr'=0$. As $R'$ is $p$-adically separated, it suffices to show that $p^Nr'=0\mod p^m$ for all $m\geqslant 0$. Note that the image of $(R/p^{m+n})[p^n]$ in $(R/p^m)[p^n]$ is contained in $(R/p^m)[p^N]$, as $R[p^n]=R[p^N]$. Note also that $(R'/p^\ell)[p^s]=(R'/p^\ell) \otimes_{R/p^\ell} (R/p^\ell)[p^s]$ for any $\ell,s\geqslant 0$ by $p$-adic flatness. Thus, the image of $(R'/p^{m+n})[p^n]$ in $(R'/p^m)[p^n]$ is contained in $(R'/p^m)[p^N]$. This implies that the image of $r'$ (or equivalently of $r'\mod p^{m+n}$) in $(R'/p^m)[p^n]$ is contained in $(R'/p^m)[p^N]$.
\end{proof}

\begin{lem}\label{lem:etale-over-qrsp}
    If $R$ is qrsp and $\Spf(R')\to \Spf(R)$ is an adically \'etale map, then $R'$ is qrsp.
\end{lem}
\begin{proof} By \cite[Remark 4.22]{BMS-THH} it is sufficient to show that $R'$ is quasi-syntomic, $R'/p$ is semi-perfect, and that there exists a perfectoid ring $S$ and a morphism $S\to R'$. The last of these is clear. To prove the first claim, we observe that as $R'$ has bounded $p^\infty$-torsion by Lemma \ref{lem:flatness-preserves-boundedness}, that $R\to R'$ is $p$-completely flat by \cite[Corollary 4.8]{BMS-THH} and so $R'\otimes^L_R (R/p)=R'/p$. 
Thus, by \stacks{08QQ} we have that $L_{R'/R}\otimes^L_{R'} (R'/p)=L_{(R'/p)/(R/p)}=0$. For semi-perfectness, observe that as $\Frob_{R/p}\colon R/p\to R/p$ is surjective, thus so is the induced map $R'/p\to R'/p\otimes_{R/p,\Frob_{R/p}}R/p$. But, this map is identified with $\Frob_{R'/p}\colon R'/p\to R'/p$ as $R/p\to R'/p$ is \'etale.
\end{proof}

Denote by $\mf{X}_{\qrsp}$ the full subcategory of $\mf{X}_\qsyn$ consisting of qrsp objects, with the induced topology. By \cite[Lemma 4.28 and Proposition 4.31]{BMS-THH}, $\mf{X}_\qrsp$ is a basis for $\mf{X}_\qsyn$. By a \emph{qrsp cover} $\{\Spf(R_i)\to\mf{X}\}$, we mean a quasi-syntomic cover where each $R_i$ is qrsp.

\subsubsection{Initial prisms}\label{ss: initial prisms} When $R$ is a qrsp ring, the category $R_\Prism$ has an initial object $(\Prism_R,I_R)$, necessarily unique up to unique isomorphism (see \cite[Proposition 7.2]{BhattScholzePrisms}).

\begin{example} If $R$ is perfectoid, then by \cite[Lemma 4.8]{BhattScholzePrisms}, $(\Ainf(R),(\xi),\text{nat.})$, or the isomorphic $(\Ainf(R),(\tilde{\xi}),\wt{\mr{nat}}.)$, are initial objects. Here we denote by $\mathrm{nat.}\colon R\isomto \Ainf(R)/(\xi)$ (resp.\@ $\wt{\mr{nat}}.\colon R\isomto \Ainf(R)/(\tilde{\xi})$) the natural isomorphism induced by $\theta$ (resp.\@ $\tilde\theta$). 
\end{example}

\begin{example}\label{example:initial-prism-Acrys}Let $R$ be a qrsp $\bb{F}_p$-algebra. If $R^\flat\defeq \varprojlim_{F_R} R$ and $J$ denotes the kernel of the composition $W(R^\flat)\to R^\flat\to R$, set $\Acrys(R)\defeq W(R^\flat)[\{\tfrac{x^n}{n!}:x\in J\}]^\wedge_p$ to be the $p$-completed divided power envelope of $(W(R^\flat),J)$, which constitutes the universal pro-(PD thickening) of $R$ over $W$ (see \cite[Th\'eor\`eme 2.2.1]{FontainePeriodes}). Let $\phi_R\colon \Acrys(R)\to \Acrys(R)$ denote the morphism induced from $F_{R}$ by this universality. Then, by \cite[Theorem 8.14]{BMS-THH}, $\Acrys(R)$ is $p$-torsion-free and $\phi_R$ is a Frobenius lift on $\Acrys(R)$ and so $(\Acrys(R),(p))$ is a prism.

Under the natural Frobenius-equivariant morphism $W(R^\flat)\to \Acrys(R)$, the ideal $\phi_R(J)$ maps to $(p)$. Indeed, as $\phi_R(J)$ contains $p$ it suffices to show that if $x$ is in $J$, then $p$ divides $\phi_R(x)$ in $\Acrys(R)$. But, observe that $x^p= p(p-1)!\tfrac{x^p}{p!}$ and so $\phi_R(x)=x^p=0\mod p\Acrys(R)$. We then obtain a morphism $\wt{\mr{nat}}.\colon R\to \Acrys(R)/p$ obtained as the composition
\begin{equation*}
    R\isomto W(R^\flat)/J\xrightarrow{\phi_R}W(R^\flat)/\phi_R(J)\to \Acrys(R)/p.
\end{equation*}
So, $(\Acrys(R),(p),\wt{\mr{nat}}.)$ constitutes an element of $R_\Prism$, and is initial by \cite[Lemma 3.27]{AnschutzLeBrasDD} and \cite[Proposition 7.10]{BhattScholzePrisms}.
\end{example}

\begin{example}\label{example:initial-prism-ainf-to-acrys}Let $R$ be a perfectoid ring, and set $\Acrys(R)\defeq \Acrys(R/p)$. Then, the universal property of $\Acrys(R)$ implies that the map $\theta\colon \Ainf(R)\to R$ extends to a map $\theta\colon \Acrys(R)\to R$ which is an initial object of $(R/W)_\crys$ (see \S\ref{ss:filtered-f-isocrystals} for this notation). We write $\phi_R$ instead of $\phi_{R/p}$. From Example \ref{example:initial-prism-Acrys}, we have a morphism of prisms $(\Ainf(R),(\tilde{\xi}),\wt{\mathrm{nat}}.)\to (\Acrys(R),(p),\wt{\mr{nat}}.)$. This is injective by \cite[Lemma 4.1.7]{ScholzeWeinstein} as $R/p$ is $f$-adic by \cite[Proposition 4.1.2]{ScholzeWeinstein} since $\ker(R^\flat\to R/p)$ is generated by the image of $\xi$ under $\Ainf(R)=W(R^\flat)\to R^\flat$.
\end{example}

\begin{rem}\label{rem:choice-of-twist} The morphism $(\Ainf(R),(\tilde{\xi}),\wt{\mathrm{nat}}.)\to (\Acrys(R),(p),\wt{\mr{nat}}.)$ in Example \ref{example:initial-prism-ainf-to-acrys} justifies the appearance of the element $\tilde{\xi}$. One cannot generally undo these Frobenius twists as while $\phi_R$ is an isomorphism on $\Ainf(R)$, it is not on $\Acrys(R)$.
\end{rem}

If $\Spf(R)\to\mf{X}$ is an object of $\mf{X}_{\qrsp}$ then $u^{-1}(h_{\Spf(R)})$ is equal to $h_{(\Prism_R,I_R)}$. Using this and the cocontinuity of $u$ (see \cite[Corollary 3.24]{AnschutzLeBrasDD}), one deduces the following (cf.\@ Lemma \ref{lem:cocont-inverse-image-locally-non-empty}).

\begin{prop}\label{prop:cover-of-final-object-qrsp} If $\{\Spf(R_i)\to \mf{X}\}$ is a qrsp cover, then $\{(\Prism_{R_i},I_{R_i})\}$ covers $\ast$ in $\cat{Shv}(\mf{X}_\Prism)$.
\end{prop}

\subsubsection{Small and base \texorpdfstring{$\mc{O}_K$-algebras}{OK algebras}}\label{ss:small-and-base-rings} We now discuss the rings of main interest in this article.

\medskip

\paragraph{First definitions} Let $R$ be a $p$-adically complete $\mc{O}_K$-algebra with $\Spec(R)$ connected. Call $R$ a \emph{base $\mc{O}_K$-algebra} if $R=R_0\otimes_W\mc{O}_K$ where $R_0$ is a $J_{R_0}$-adically complete ring for which the pair $(R_0,J_{R_0})=(A_n,I_n)$, is obtained from the following iterative procedure. For some $d\geqslant 0$, let $A_0$ be $T_d\defeq W\langle t_1^{\pm 1},\ldots,t_d^{\pm 1}\rangle$, and set $I_0=(p)$. For each $i=0,\ldots,n-1$ iteratively form the pair $(A_{i+1},I_{i+1})$ by one of the following operations:
\begin{itemize}
    \item $A_{i+1}$ is the $p$-adic completion of an \'etale $A_i$-algebra,\footnote{By Elkik's theorem (cf.\@ \stacks{0AKA}) we may replace this with: `a $p$-adically \'etale $A_i$-algebra'.} and $I_{i+1}=I_iA_{i+1}$,
    \item $A_{i+1}$ is the $p$-adic completion of a localization $A_i\to (A_i)_\mf{p}$ at a prime $\mf{p}$ containing $p$, and $I_{i+1}=I_iA_{i+1}$,
    \item $A_{i+1}$ is the $I$-adic completion of $A_i$ with respect to an ideal $I\subseteq A_i$ containing $p$, and $I_{i+1}=(I_i,I)A_{i+1}$.
\end{itemize}
While the discussion of base $\mc{O}_K$-algebras implicitly entails other topologies, we always think of a base $\mc{O}_K$-algebra as being equipped with the $p$-adic topology.

A map $t\colon T_d\to R_0$ (where we implicitly have $R=R_0\otimes_W \mc{O}_K$) of the form constructed above is called a \emph{presentation}.  In the following we use terminology from \cite[\S2.2]{KimBK}. Moreover, for a map of rings $R\to S$ and an ideal $J\subseteq S$, we say that $R\to S$ is \emph{$J$-formally \'etale} if the following solid diagram 
\begin{equation*}
\begin{tikzcd}
	S & {A/I} \\
	R & A
	\arrow[from=1-1, to=1-2]
	\arrow[dashed, from=1-1, to=2-2]
	\arrow[from=2-1, to=1-1]
	\arrow[from=2-1, to=2-2]
	\arrow[from=2-2, to=1-2]
\end{tikzcd}
\end{equation*}
can be uniquely completed along the dotted arrow assuming that $A$ is a ring and $I\subseteq A$ is a square-zero ideal such that the image $J$ in $A/I$ is nilpotent, and similarly for $J$-formally smooth.

\begin{prop}\label{prop:base-algebra-good-properties} A base $\mc{O}_K$-algebra $R$ is excellent and regular, and $R/\pi$ has a finite $p$-basis. For a presentation $t\colon T_d\to R_0$, the $k$-algebra $R_0/J_{R_0}$ is finite type, and $t$ is $J_{R_0}$-formally \'etale.
\end{prop}
\begin{proof} Fix a presentation $t\colon T_d\to R_0$. We first prove that $R$ is excellent. As $R_0\to R$ is finite type, it suffices to prove that $R_0$ is excellent (see \stacks{07QU}). We prove that $A_i$ is excellent by induction on $i$. As $T_d$ is the $p$-adic completion of $W[t_1^{\pm 1},\ldots,t_d^{\pm 1}]$, which is excellent by loc.\@ cit.\@, we know that $T_d$ is excellent by \cite[Main Theorem 2]{KuranoShimomoto}. If $A_i$ is excellent, then $A_{i+1}$ is excellent regardless of which of the three constructions is applied to $A_i$ by combining \stacks{07QU} and \cite[Main Theorem 2]{KuranoShimomoto}. To prove that $t$ is $J_{R_0}$-formally \'etale, it suffices by induction to prove that $A_i\to A_{i+1}$ is $I_{i+1}$-formally \'etale for each $i=0,\ldots,n-1$, but this is clear. Thus, $T_d\otimes_W \mc{O}_K\to R$ is $J_{R_0}$-formally \'etale, and thus it follows from \cite[Theorem 6.4.2]{MajadasRodicio} that $T_d\otimes_W \mc{O}_K\to R$ is regular and thus $R$ is regular.\footnote{For a map $f\colon (R,\mf{m})\to (S,\mf{n})$ of local rings, the phrase `formally smooth' in \cite[Theorem 6.4.2]{MajadasRodicio} means $\mf{n}$-formally smooth in our terminology.} The final claims concerning $R/\pi=R_0/p$ and $R_0/J_{R_0}$ may be checked iteratively, the latter being obvious. The former being preserved by our three operations requires \cite[Lemma 1.1.3]{deJongCrystalline}, and the observation that if $\{x_\alpha\}$ is a $p$-basis for an $\bb{F}_p$-algebra $A$, and $A\to B$ is \'etale then $\{x_\alpha\}$ is a $p$-basis for $B$ as $F_A\otimes 1$ is identified with $F_B$ via the isomorphism $A \otimes_{F_A,A} B \isomto B$. 
\end{proof}

We call a decomposition $R=R_0\otimes_W \mc{O}_K$ and a $J_{R_0}$-formally \'etale map $t\colon T_d\to R_0$, where $J_{R_0}$ is any ideal coming from a presentation, a \emph{formal framing}. For a formal framing $t$, the ring $R_0$ carries a unique Frobenius lift $\phi_{t}$ such that $\phi_t\circ t=t \circ \phi_0$, where $\phi_0$ is the Frobenius lift on $T_d$ acting as usual on $W$ and with $\phi_0(t_i)=t_i^p$. 

\medskip

\paragraph{Perfectoid type prisms} Let $R$ be a base $\mc{O}_K$-algebra. Write $\mc{K}$ for $\Frac(R)$. Fix an algebraic closure $\ov{\mc{K}}$ containing $\ov{K}$, and denote by $\mc{K}^\ur$ the maximal subfield of $\ov{\mc{K}}$ unramified along $R[\nicefrac{1}{p}]$ (cf.\@ \stacks{0BQJ}). Denote $\Gal(\mc{K}^\ur/\mc{K})$ by $\Gamma_R$, which agrees with $\pi_1^\et(\Spec(R[\nicefrac{1}{p}]),\overline{x})$ where $\overline{x}$ is the geometric point determined by $\ov{\mc{K}}$ (see \stacks{0BQM}). Denote by $\ov{R}$ (resp.\@ $R^\ur$) the integral closure of $R$ in $\ov{\mc{K}}$ (resp.\@ $\mc{K}^\ur$), and set $\wt{R}$ (resp.\@ $\check{R}$) to be its $p$-adic completion.\footnote{Note that our notation here differs from some references (e.g.\@ \cite{DLMS}) where the notation $\ov{R}$ is used for what we denote by $\check{R}$.} The $\mc{O}_C$-algebras $\check{R}$ and $\wt{R}$ are perfectoid by the following lemma. 

\begin{lem}[{\cite[Proposition 2.1.8]{CesnaviciusScholze}}] Let $A$ be a $p$-torsion-free $\mc{O}_C$-algebra which is $p$-integrally closed in $A[\nicefrac{1}{p}]$ (i.e.\@, if $x^p$ is in $A$ for $x$ in $A[\nicefrac{1}{p}]$, then $x$ is in $A$). Then, the $p$-adic completion $\widehat{A}$ is perfectoid if and only if $A/p$ is semi-perfect.
\end{lem}

Thus, we have the objects $(\Ainf(\check{R}),(\tilde{\xi}),\wt{\mr{nat}}.)$ and $(\Acrys(\check{R}),(p),\wt{\mr{nat}}.)$ of $R_\Prism$, and their $\wt{R}$ counterparts.

\medskip

\paragraph{Breuil--Kisin type prisms} Let $R=R_0\otimes_W \mc{O}_K$ be a base $\mc{O}_K$-algebra and $t\colon T_d\to R_0$ a formal framing.  We consider the following objects of $R_\prism$.
\begin{enumerate}[itemsep=7pt,leftmargin=.7cm]
    \item The object $(R_0^{(\phi_t)},(p),q)$. Here $R_0^{(\phi_t)}$ is the ring $R_0$ equipped with the $\delta$-strucure corresponding to the Frobenius lift $\phi_t$, and $q$ is the quotient map $R\to R/\pi\isomto R_0/p$.
    \item\textbf{(Breuil--Kisin prism)} The object $(\mf{S}_R^{(\phi_t)},(E),\mathrm{nat.})$. Here $\mf{S}_R^{(\phi_t)}\defeq \mf{S}_R\defeq R_0\ll u\rr$ is equipped with the $\delta$-structure corresponding to the Frobenius $\phi_t\colon \mf{S}_R\to\mf{S}_R$ extending $\phi_t$ on $R_0$ and satisfying $\phi_t(u)=u^p$. The map  $\mathrm{nat.}\colon R\isomto \mf{S}_R/(E)$ is the natural one.
    \item\textbf{(Breuil prism)} Consider the Breuil ring $S_R$, defined to be the $p$-adic completion of the PD-envelope of $\mf{S}_R\twoheadrightarrow R$, which can be explicitly described as follows:
    \begin{equation}\label{eq:Breuil-ring-description}
        S_R=\left\{\sum_{m=0}^\infty a_m\frac{u^m}{\lfloor m/e\rfloor!}\in R_0[\nicefrac{1}{p}]\ll u\rr: a_m\text{ converge to }0\text{ }p\text{-adically}\right\}. 
    \end{equation}
    The ring $S_R^{(\phi_t)}\defeq S_R$ has a unique Frobenius $\phi_t$ extending that on $\mf{S}_R$, and thus an associated $\delta$-structure. We then have the triple $(S_R^{(\phi_t)},(p),\ov{i}\circ \ov{\phi}_t\circ \text{nat.})$, where $i\colon \mf{S}_R\to S_R$ is the natural inclusion, and $\ov{\phi}_t\colon \mf{S}_R/(E)\to \mf{S}_R/(\phi_t(E))$ is the map induced by $\phi_t$, and this composition makes sense as a map $R\to S_R/p$ as $\tfrac{\phi_t(E)}{p}$ is a unit in $S_R$.
\end{enumerate}
We often omit the decoration $(-)^{(\phi_t)}$ when the choice of a particular formal framing is clear or unimportant, in which case we just write $\phi$ for $\phi_t$.

\medskip

\paragraph{Various morphisms of prisms} Let $R=R_0\otimes_W\mc{O}_K$ be a base $\mc{O}_K$-algebra and choose a formal framing $t\colon T_d\to R_0$. Denote by $t^\flat$ the choice of $p^\text{th}$-power roots $t_i^\flat$ of $t_i$ in $\check{T}_d^\flat$. 

Define $R_0\to \Ainf(\check{R})$, using the $J_{R_0}$-formal \'etaleness of $t$, as the unique extension of the map $T_d\to \Ainf(\check{R})$ sending $t_i$ to $[t_i^\flat]$ and inducing the natural map $R_0\to \check{R}$ after composition with $\wt{\theta}$. We further define $\alpha_\inf=\alpha_{\inf,t^\flat}\colon \mf{S}_R\to \Ainf(\check{R})$ to be the unique extension of this map such that $\alpha_{\inf,t^\flat}(u)=[\pi^\flat]$. We then have the following diagram of prisms, where each inclusion arrow is the obvious inclusion, and all other not previously defined arrows are determined uniquely by commutativity.

\begin{equation}\label{eq:big-BK-diagram}
\begin{tikzcd}
	&& {(R_0^{(\phi_t)},(p),q)} & {(R_0^{(\phi_t)},(p),F_{R_0/p}\circ q)} \\
	{(\mathfrak{S}_R^{(\phi_t)},(E),\mathrm{nat.})} && {(\mathfrak{S}_R^{(\phi_t)},(\phi_t(E)),\ov{\phi}_t\circ \mathrm{nat.})} & {(S_R^{(\phi_t)},(p),\ov{i}\circ\ov{\phi}_t\circ \mathrm{nat.})\text{     }} \\
	&& {(\mathrm{A}_\mathrm{inf}(\check{R}),(\tilde{\xi}),\widetilde{\mr{nat}}.)} & {(\mathrm{A}_\mathrm{crys}(\check{R}),(p),\widetilde{\mr{nat}}.)} \\
	&& {(\mathrm{A}_\mathrm{inf}(\widetilde{R}),(\tilde{\xi}),\widetilde{\mr{nat}}.)} & {(\Acrys(\widetilde{R}),(p),\widetilde{\mr{nat}}.)}
	\arrow["{\alpha_{\mathrm{inf},t^\flat}}"', from=2-3, to=3-3]
	\arrow[hook, from=3-3, to=3-4]
	\arrow["{\alpha_{\mathrm{crys},t^\flat}}", from=2-3, to=3-4]
	\arrow[hook, from=3-3, to=4-3]
	\arrow[hook, from=3-4, to=4-4]
	\arrow[hook, from=4-3, to=4-4]
	\arrow["{\phi_t}", from=2-1, to=2-3]
	\arrow["{\widetilde{\alpha}_{\mathrm{inf},t^\flat}}"', from=2-1, to=3-3]
	\arrow["{\phi_t}", from=1-3, to=1-4]
	\arrow["{\gamma_{t^\flat}}",from=2-4, to=3-4]
	\arrow[hook, from=2-3, to=2-4]
	\arrow["{(\ast)}",hook, from=1-4, to=2-4]
	\arrow["\color{black}\beta_{t^\flat}", shift left=5, curve={height=-37pt}, from=1-4, to=3-4]
\end{tikzcd}
\end{equation}
Note that for all of these triples, save $(R_0^{(\phi_t)},(p),q)$ and $(R_0^{(\phi_t)},(p),F_{R_0/p}\circ q)$, the structure morphism is unambiguous given the first two entries, so we often omit it. If we write $(R_0^{(\phi_t)},(p))$ then the reader should assume we mean $(R_0^{(\phi_t)},(p),q)$.

While every morphism in Diagram \eqref{eq:big-BK-diagram} is a morphism of prisms, the arrow labeled $(\ast)$ is the only map which may not be a morphism in $R_\Prism$. In fact, this happens precisely when $\mc{O}_K=W$.

\begin{lem}\label{lem: map in R prism} For an integer $a\geqslant 0$, let $s_{a+1}$, denote the following composition:
\begin{equation*}
    R\xrightarrow{\mr{nat}.}\mf S_R/(E)\xrightarrow{\bar\phi_t} \mf{S}_R/(\phi_t(E))\xrightarrow{\ov{i}}S_R/p\xrightarrow{F_{S_R/p}^a} S_R/p.
\end{equation*} 
Then $(R_0,(p),F^{a+1}_{R_0/p}\circ q)\hookrightarrow (S_R,(p),s_{a+1})$ is a morphism in $R_\Prism$ if and only if $p^{a}\geqslant e$.
\end{lem}
\begin{proof} It suffices to observe that the following diagram commutes if and only if $p^a\geqslant e$:
\begin{equation*}
  \begin{tikzcd}[column sep=2.8em]
	{\mathfrak{S}_R/(\phi_t(E))} & {S_R/p} & {S_R/p} & {R_0/p} & {R_0/p} \\
	{\mathfrak{S}_R/(E)} && R && {R/(\pi)}
	\arrow["{\ov{i}}", from=1-1, to=1-2]
	\arrow["{F_{S_R/p}^a}", from=1-2, to=1-3]
	\arrow[from=1-4, to=1-3]
	\arrow["{F_{R_0/p}^{a+1}}"', from=1-5, to=1-4]
	\arrow["\wr"', from=2-5, to=1-5]
	\arrow[from=2-3, to=1-3]
	\arrow[from=2-3, to=2-5]
	\arrow["{\mathrm{nat.}}", from=2-3, to=2-1]
	\arrow["{\ov{\phi}_t}", from=2-1, to=1-1]
\end{tikzcd}
\end{equation*}
Note that $R=R_0[\pi]$ and that the diagram always commutes on $R_0$. Then chasing where $\pi$ is sent, one sees this commutativity holds if and only if $u^{p^{a+1}}$ is divisible by $p$ in $S_R$. Using Equation \eqref{eq:Breuil-ring-description}, one easily sees that this happens if and only if $\tfrac{p^{a+1}}{e}\geqslant p$ or, equivalently, that $p^a\geqslant e$.
\end{proof}

\medskip

\paragraph{Miscellanea} Following \cite{DLMS}, we call a Zariski connected $p$-adically complete $\mc{O}_K$-algebra $R$ \emph{small} if there exists a $p$-adically \'etale morphism $t\colon \mc{O}_K\langle t_1^{\pm 1},\ldots,t_d^{\pm 1}\rangle\to R$ for some $d\geqslant 0$, called a \emph{framing}. By the topological invariance of the \'etale site of a formal scheme, there is a unique $p$-adically \'etale morphism $T_d\to R_0$ with $R=R_0\otimes_W \mc{O}_K$ such that the composition $T_d\to R_0\to R$ is equal to $t$. We denote the map $T_d\to R_0$ also by $t$. Thus, $R$ is a base $\mc{O}_K$-algebra.

By a \emph{base formal $\mc{O}_K$-scheme} we mean a morphism of formal schemes $\mf{X}\to \Spf(\mc{O}_K)$ such that there exists an open cover $\{\Spf(R_i)\}$ of $\mf{X}$ with each $R_i$ a base $\mc{O}_K$-algebra. By the discussion in the last paragraph, this includes smooth formal schemes $\mf{X}\to\Spf(\mc{O}_K)$.

\begin{lem}[{cf.\@ \cite[Theorem 5.16]{Bha20}}]\label{lem:Bhatt-flatness} Let $R$ be a base $\mc{O}_K$-algebra. Then, $R\to \wt{R}$ is faithfully flat and quasi-syntomic.
\end{lem}
\begin{proof} The faithful flatness follows from \cite[Theorem 5.16]{Bha20}. To prove quasi-syntomicness, first observe that $L_{\wt{R}/\mc{O}_K}\otimes^L_{\mc{O}_K} (\mc{O}_K/p)$ is concentrated in degree $-1$ (cf.\@ the proof of \cite[Proposition 4.19]{BMS-THH}, using $\mc{O}_K\to \Ainf(\wt{R})\otimes_W \mc{O}_K\xrightarrow{\theta\otimes 1}\wt{R}$). On the other hand, $\mc{O}_K\to R$ is $p$-completely flat, and so $R\otimes^L_{\mc{O}_K} \mc{O}_K/p=R/p$. 
Thus, by \stacks{08QQ} we have that $L_{R/\mc{O}_K}\otimes^L_{\mc{O}_K} (\mc{O}_K/p)$ is equal to $L_{(R/p)/(\mc{O}_K/p)}$. But, as $\mc{O}_K/p\to R/p$ is formally smooth, $L_{(R/p)/(\mc{O}_K/p)}$ is concentrated in degree $0$. So, the claim follows by the triangle property for the cotangent complex.
\end{proof}

\begin{prop}\label{prop:small-open-cover-topos-cover}
    Let $\mf{X}\to\Spf(\mc{O}_K)$ be a base formal $\mc{O}_K$-scheme, and $\{\Spf(R_i)\}$ an open cover with each $R_i$ a base $\mc{O}_K$-algebra. Then, $\{(\mf{S}_{R_i},(E))\}$ is a cover of $\ast$ in $\cat{Shv}(\mf{X}_\Prism)$.
\end{prop}
\begin{proof} The maps $\wt{\alpha}_{i,\inf}\colon (\mf{S}_{R_i},(E))\to (\Ainf(\wt{R}_i),(\tilde{\xi}))$ in $\mf{X}_\Prism$, shows that $\{(\mf{S}_{R_i},(E))\to \ast\}$ refines $\{(\Ainf(\wt{R}_i),(\wt{\xi}))\to \ast\}$. But, combining Proposition \ref{prop:cover-of-final-object-qrsp} with Lemma \ref{lem:Bhatt-flatness}, we see that $\{(\Ainf(\wt{R}_i),(\wt{\xi}))\to \ast\}$ is a cover, and thus so is $\{(\mf{S}_{R_i},(E))\to \ast\}$ (see Lemma \ref{lem:condition-for-covering-of-final-object}).
\end{proof}

\subsection{(Analytic) prismatic torsors with \texorpdfstring{$F$}{F}-structure}\label{ss:analytic-prismatic-F-torsors}

We now discuss the theory of (analytic) prismatic $F$-crystals, as in \cite{BhattScholzePrisms} and \cite{GuoReinecke}. Fix a quasi-syntomic formal scheme $\mf{X}$, and an object $T$ of $\cat{Shv}(\mf{X}_\Prism)$, which we omit from the notation when $T=\ast$. 
Although we will only need the case where $T=\ast$ in the sequel, we discuss the general case for the sake of naturality, especially in order to allow us to use the framework established in Appendix \ref{s:Tannakian-appendix} (e.g., Remark \ref{rem: need general T to define local triviality}). 
As usual we confuse representable (pre)sheaves and the objects that represent them via the Yoneda embedding (as discussed in \S\ref{ss:basic-definitions-and-results}). 

\subsubsection{Prismatic \texorpdfstring{$F$-crystals}{F-crystals}} Define the category of \emph{prismatic crystals in vector bundles (resp.\@ perfect complexes)} over $T$ as follows:
\begin{equation*}
     \twolim_{(A,I)\in\mf{X}_\Prism/T}\cat{Vect}(A)\qquad \bigg(\text{resp. }\lim_{(A,I)\in\mf{X}_\Prism/T}\cat{D}_\perf(A)\bigg).
\end{equation*}
Concretely, a prismatic crystal in vector bundles (resp.\@ perfect complexes) is a collection of finite projective $A$-modules $M_{(A,I)}$ (resp.\@ perfect complexes $K^\bullet_{(A,I)}$), indexed by objects $(A,I)$ of $\mf{X}_\Prism/T$, together with (quasi-)isomorphisms $M_{(A,I)}\otimes_A B\isomto M_{(B,J)}$ (resp.\@ $K^\bullet_{(A,I)}\otimes^L_A B\isomto K^\bullet_{(B,J)}$) for any morphism $(A,I)\to (B,J)$ in $\mf{X}_\Prism$ (the \emph{crystal property}), satisfying the obvious compatibility conditions. The category of prismatic crystals in vector bundles carries the structure of an exact $\Z_p$-linear $\otimes$-category where exactness and tensor products are defined term-by-term.

\begin{prop}[{cf.\@ \cite[Proposition 2.7]{BhattScholzeCrystals}}]\label{prop:vector-bundle-limit-decomp} The global sections functor 
\begin{equation*}
    \cat{Vect}(\mf{X}_\Prism/T,\mc{O}_\Prism)\to \twolim_{(A,I)\in\mf{X}_\Prism/T}\cat{Vect}(A) 
\end{equation*}
is a bi-exact $\Z_p$-linear $\otimes$-equivalence. Moreover, the derived global sections functor
\begin{equation*}
    \cat{D}_\perf(\mf{X}_\Prism/T,\mc{O}_\Prism)\to \lim_{(A,I)\in\mf{X}_\Prism/T}\cat{D}_\perf(A) 
\end{equation*}
is an equivalence of $\infty$-categories.
\end{prop}
\begin{proof} By \cite[Proposition 2.7]{BhattScholzeCrystals}, it remains only to verify that the first functor is a bi-exact $\Z_p$-linear $\otimes$-equivalence. By \cite[Proposition 2.7]{BhattScholzeCrystals}, if $\mc{F}$ and $\mc{G}$ are objects of $\cat{Vect}(\mf{X}_\Prism/T,\mc{O}_\Prism)$ then the presheaf $(A,I)\mapsto \mc{F}(A,I)\otimes_A \mc{G}(A,I)$ is a sheaf, and so the global sections functor preserves tensor products. Indeed, as $(\mc{F}(A,I)\otimes_A \mc{G}(A,I))$ forms a prismatic crystal in vector bundles by the crystal property for $\mc{F}$ and $\mc{G}$, there exists by \cite[Proposition 2.7]{BhattScholzeCrystals} a prismatic crystal $\mc{H}$ with $\mc{H}(A,I)=\mc{F}(A,I)\otimes_A \mc{G}(A,I)$, and as $\mc{H}$ is a sheaf the claim follows. The bi-exactness claim follows easily from Lemma \ref{lem: acyclicity of prismatic crystals}.
\end{proof}

\begin{lem}[{cf.\ \cite[Corollary 3.12]{BhattScholzePrisms}}]\label{lem: acyclicity of prismatic crystals}
    Let $\mc E$ be an object of $\cat{Vect}(\mf X_\prism,\mc{O}_\prism)$. Then for any object $(A,I)$ of $\mf{X}_\Prism$, we have that $H^i((A,I),\mc E)=0$ for any $i>0$. 
\end{lem}
\begin{proof} The proof of \cite[Corollary 3.12]{BhattScholzePrisms} applies as the \v{C}ech complex for $\mathcal{E}$ is the result of tensoring the \v{C}ech complex for $\mathcal{O}_\Prism$ with the flat module $M=\mathcal{E}(A,I)$, and so is still exact.
\end{proof}

As in \cite[Definition 4.1]{BhattScholzeCrystals}, define the category $\cat{Vect}^\varphi(\mf{X}_\Prism/T)$ of \emph{prismatic $F$-crystals over $T$} to have objects $(\mc{E},\varphi_\mc{E})$ where $\mc{E}$ is an object of $\cat{Vect}(\mf{X}_\Prism/T,\mc{O}_\Prism)$ and 
\begin{equation*}
\varphi_\mc{E}\colon (\phi^\ast\mc{E})[\nicefrac{1}{\mc{I}_\Prism}]\isomto \mc{E}[\nicefrac{1}{\mc{I}_\Prism}],
\end{equation*}
is an isomorphism in $\cat{Vect}(\mf{X}_\Prism/T,\mc{O}_\Prism[\nicefrac{1}{\mc{I}_\Prism}])$, called the \emph{Frobenius}, and morphisms are morphisms in $\cat{Vect}(\mf{X}_\Prism/T,\mc{O}_\Prism)$ commuting with the Frobenii. Likewise, define the category $\cat{D}^\varphi_\perf(\mf{X}_\Prism/T)$ of \emph{prismatic $F$-crystals in perfect complexes} to be the category of pairs $(\mc{E}^\bullet,\varphi_{\mc{E}^\bullet})$ where $\mc{E}^\bullet$ is an object of $\cat{D}_\perf(\mf{X}_\Prism/T,\mc{O}_\Prism)$ together with a \emph{Frobenius isomorphism} in $\cat{D}_\perf(\mf{X}_\Prism,\mc{O}_\Prism[\nicefrac{1}{\mc{I}_\Prism}])$
\begin{equation*}
    \varphi_{\mc{E}^\bullet}\colon L\phi^\ast\mc{E}^\bullet[\nicefrac{1}{\mc{I}_\Prism}]\isomto \mc{E}^\bullet[\nicefrac{1}{\mc{I}_\Prism}],
\end{equation*}
and with morphisms being those in $\cat{D}_\perf(\mf{X}_\Prism/T,\mc{O}_\Prism)$ commuting with Frobenii. By Proposition \ref{prop:vector-bundle-limit-decomp}, when $\cat{Vect}^\varphi(A,I)$ and $\cat{D}^\varphi_{\perf}(A,I)$ are given the obvious meanings, then
\begin{equation*}
\cat{Vect}^\varphi(\mf{X}_\Prism/T)=\twolim_{\scriptscriptstyle (A,I)\in\mf{X}_\Prism/T}\cat{Vect}^\varphi(A,I),\qquad \cat{D}^\varphi_{\perf}(\mf{X}_\Prism/T)=\lim_{\scriptscriptstyle (A,I)\in\mf{X}_\Prism/T}\cat{D}^\varphi_{\perf}(A,I).
\end{equation*}
Observe that $\cat{Vect}^\varphi(\mf{X}_\Prism/T)$ inherits  from $\cat{Vect}(\mf{X}_\Prism/T,\mc{O}_\Prism)$ the structure of an exact $\Z_p$-linear $\otimes$-category. We say $(\mc{E},\varphi_\mc{E})$ is \emph{effective} if $\varphi_\mc{E}$ is induced from a morphism $\varphi_{\mc{E},0}\colon \phi^\ast\mc{E}\to\mc{E}$, which is automatically injective, and the minimal $r$ such $\mc{I}_\Prism^r$ kills $\mathrm{coker}(\varphi_{\mc{E},0})$ is the \emph{height} of $(\mc{E},\varphi_\mc{E})$.

\begin{rem}\label{rem:pullback-crystal-Frobenius} Suppose that $(A,I,s)$ is an object of $\mf{X}_\Prism$ with the property that $\phi_A(I)$ is an invertible ideal (e.g.\@, if $\phi_A$ is flat, see \stacks{02OO}, or $I=(p)$). In this case $(A,\phi_A(I),s\circ \Spf(\ov{\phi}_A))$ is an object of $\mf{X}_\Prism$, where $\ov{\phi}_A\colon A/I\to A/\phi_A(I)$ is the natural map. Indeed, if $p=i+\phi_A(i')$, with $i$ and $i'$ in $I$, then evidently $p=\phi_A(i)+\phi_A(\phi_A(i'))$ so that $p\in \phi_A(I)+\phi_A(\phi_A(I))$, and $A$ is $(p,\phi_A(I))$-adically complete as $(p,I)^p\subseteq (p,\phi_A(I))\subseteq (p,I)$. Moreover, we observe that $\phi_A\colon (A,I,s)\to (A,\phi_A(I),s\circ\Spf(\ov{\phi}_A))$ is a morphism in $\mf{X}_\Prism$ so that, by the crystal property, we have an identification of $A$-modules
\begin{equation*}
    \phi^\ast\mc{E}(A,I,s)=\phi_A^\ast(\mc{E}(A,I,s))\isomto \mc{E}(A,\phi_A(I),s\circ\Spf(\ov{\phi}_A)). 
\end{equation*}
Thus, we, in particular, see that via this identification there is an isomorphism 
\begin{equation*}
    \varphi_\mc{E}\colon (\phi^2)^\ast\mc{E}(A,I,s)[\nicefrac{1}{\phi_A(I)}]\isomto \phi^\ast\mc{E}(A,I,s)[\nicefrac{1}{\phi_A(I)}],
\end{equation*}
which provides $\phi^\ast\mc{E}(A,I,s)$ with a Frobenius-like structure.
\end{rem}

\subsubsection{Analytic prismatic \texorpdfstring{$F$-crystals}{F-crystals}} We now discuss the notion of analytic prismatic $F$-crystals as defined in \cite{GuoReinecke}, which are necessary to model every crystalline $\Z_p$-local system.

\medskip

\paragraph{Basic definitions} Following \cite[Definition 3.1]{GuoReinecke}, define the category of \emph{analytic prismatic crystals} over $T$ as follows (where we recall that $U(A,I)\defeq \Spec(A)-V(p,I)$):
\begin{equation*}
    \cat{Vect}^\mathrm{an}(\mf{X}_\Prism/T)\defeq\twolim_{\scriptscriptstyle(A,I)\in \mathfrak{X}_\Prism/T}\cat{Vect}(U(A,I)).
\end{equation*}
We denote an object of $\cat{Vect}^\an(\mf{X}_\Prism/T)$ as $\mc{V}=(\mc{V}_{(A,I)})$. By the argument given in \cite[Proposition 3.7]{GuoReinecke}, the following restriction is fully faithful: 
\begin{equation*}
    \cat{Vect}(\mf{X}_\Prism/T,\mc{O}_\Prism)\isomto \twolim_{\scriptscriptstyle(A,I)\in\mf{X}_\Prism/T}\cat{Vect}(A)\to \twolim_{\scriptscriptstyle(A,I)\in\mf{X}_\Prism/T}\cat{Vect}(U(A,I))=\cat{Vect}^\an(\mf{X}_\Prism/T). 
\end{equation*}
Endow $\cat{Vect}^\an(\mf{X}_\Prism/T)$ with the structure of a $\Z_p$-linear $\otimes$-category defined term-by-term in the two-limit. Then the above restriction is a $\Z_p$-linear $\otimes$-functor. Denote the object $(\mc{O}_{U(A,I)})$  by $\mc{O}^\an_\Prism$. As in \cite[Definition 3.5]{GuoReinecke}, we then define the category $\cat{Vect}^{\varphi,\mr{an}}(\mf{X}_\Prism)$ of \emph{analytic prismatic $F$-crystals}, denoted typically by $(\mc{V},\varphi_\mc{V})$, in the same way and endowed with the analogous structure of a $\Z_p$-linear $\otimes$-category.

We say that a sequence of analytic prismatic crystals 
    \begin{equation*}
0\to \mc{V}^1\to \mc{V}^2\to \mc{V}^3\to 0
\end{equation*}
is \emph{exact} if the sequence of vector bundles 
\begin{equation*}
0\to \mc{V}^1_{(A,I)}\to \mc{V}^2_{(A,I)}\to \mc{V}^3_{(A,I)}\to 0
\end{equation*}
on $\Spec(A)-V(p,I)$ is exact for all $(A,I)$ in $\mf{X}_\Prism$. We define a sequence of objects of $\cat{Vect}^{\varphi,\mr{an}}(\mf{X}_\Prism)$ to be exact if the underlying sequence of analytic prismatic crystals is. This endows $\cat{Vect}^\an(\mf{X}_\Prism)$ and $\cat{Vect}^{\an,\varphi}(\mf{X}_\Prism)$ with the structures of exact $\Z_p$-linear $\otimes$-categories.

\medskip

\paragraph{Criteria for exactness} Our next goal is to establish the following basic result.

\begin{prop}\label{prop:exactness-preserved-cover-analytic-prismatic-crystals} A sequence
\begin{equation*}
0\to \mc{V}^1\to \mc{V}^2\to \mc{V}^3\to 0
\end{equation*}
in $\cat{Vect}^{\mathrm{an}}(\mf{X}_\Prism)$ is exact if and only if there exists a cover $\{(A_i,I_i)\}$ of $\ast$ such that 
\begin{equation*} 
0\to \mathcal{V}^1_{(A_i,I_i)}\to \mathcal{V}^2_{(A_i,I_i)}\to \mathcal{V}^3_{(A_i,I_i)}\to 0
\end{equation*}
is exact for all $i$.
\end{prop}

\begin{rem} Proposition \ref{prop:exactness-preserved-cover-analytic-prismatic-crystals} does not easily follow using fpqc
descent. The latter condition implies (see Lemma \ref{lem:condition-for-covering-of-final-object}) that for every $(A,I)$ there exists a cover $\{(B_j,IB_j)\}$ of $(A,I)$ such that sequence over $(A,I)$ is exact when pulled back to each $U(B_j,IB_j)$. If $S=\{U(B_j,IB_j)\to U(A,I)\}$ were an fpqc cover we would be done. But, a priori all we know is that $\{\Spf(B_j)\to \Spf(A)\}$ is an adically flat cover. With Noetherian hypotheses this is sufficient (cf.\@ \stacks{0912}), but in our generality it is not even clear that $S$ is jointly surjective (although see Proposition \ref{prop:Gabber} below).
\end{rem}

Using Lemma \ref{lem:condition-for-covering-of-final-object}, and the fact that an exact sequence of vector bundles is universally exact (see \stacks{058H}), the desired result is a special case of the following.

\begin{prop}\label{prop:exactness-ff-cover} Let $f\colon A\to B$ be a map of rings, and suppose that the finitely generated ideal $J\subseteq A$ is contained in the Jacobson radical of $A$. If $f\colon A/J^n\to B/J^nB$ is faithfully flat for all $n$ then, a sequence of vector bundles
\begin{equation*}
    P\colon \quad 0\to \mc{M}_1\to\mc{M}_2\to\mc{M}_3\to 0,
\end{equation*}
on $\Spec(A)-V(J)$ is exact if and only if the sequence
    \begin{equation*}
    f^\ast P\colon \quad 0\to f^\ast\mc{M}_1\to f^\ast\mc{M}_2\to f^\ast\mc{M}_3\to 0,
\end{equation*}
of vector bundles on $\Spec(B)-V(JB)$ is exact.
\end{prop}

Proposition \ref{prop:exactness-ff-cover} follows from the following two lemmas as $\Spec(B)-V(JB)$ is quasi-compact (as $J$ is finitely generated) so that any point contains a closed point in its closure.

\begin{lem}\label{lem:local-ring-seq} Let $(Y,\mc{O}_Y)$ be a locally ringed space and $S$ a subset of $Y$ such that every point of $Y$ admits an element of $S$ as a specialization. Then, a sequence 
\begin{equation*}
    Q\colon \quad 0\to V_1\to V_2\to V_3\to 0,
\end{equation*}
of $\mc{O}_Y$-modules, where each $V_i$ is a vector bundle on $(Y,\mc{O}_Y)$, is exact if and only if for all $y$ in $S$ the induced sequence
\begin{equation*}
    0\to (V_1)_{k(y)}\to (V_2)_{k(y)}\to (V_3)_{k(y)}\to 0,
\end{equation*}
of vector spaces over the residue field $k(y)$ is exact. In particular, if $f\colon (Y',\mc{O}_Y)\to (Y,\mc{O}_Y)$ is a map of locally ringed spaces containing $S$ in its image, then $Q$ is exact if and only if 
\begin{equation*}
    f^\ast Q\colon \quad 0\to f^\ast V_1\to f^\ast V_2\to f^\ast V_3\to 0,
\end{equation*}
is exact.
\end{lem}
\begin{proof} The second claim follows easily from the first, and only the if condition of the first statement requires proof. To prove this, it suffices to show that for each $y$ in $S$, the sequence of projective $\mathcal{O}_{Y,y}$-modules
\begin{equation*}
    Q_y\colon \quad 0\to (V_1)_y\to (V_2)_y\to (V_3)_y\to 0,
\end{equation*}
is exact. Indeed, if $y'$ is an arbitrary point of $Y$ and $y$ in $S$ is a specialization, then the sequence $Q_{y'}$ is obtained by base changing the sequence $Q_y$ along the flat map $\mc{O}_{Y,y}\to \mc{O}_{Y,y'}$.

To show that $Q_y$ is exact, let $n_i$, for $i=1,2,3$, be the rank of the finite free $\mc{O}_{Y,y}$-module $(V_i)_y$. That $(V_2)_y\to (V_3)_y$ is surjective follows from Nakayama's lemma. We then see that the sequence
\begin{equation}\label{eq:local-ring-seq}
    0\to \ker((V_2)_y\to (V_3)_y)\to (V_2)_y\to (V_3)_y\to 0
\end{equation}
is exact. As $V_3$ is a vector bundle, this sequence splits. So, $\ker((V_2)_y\to (V_3)_y)$ is locally free of rank equal to $n_2-n_3$. As rank can be computed over $k(y)$, we have that $n_1=n_2-n_3$. Thus, $(V_1)_y\to \ker((V_2)_y\to (V_3)_y)$ is a map of finite free $\mc{O}_{Y,y}$-modules of the same rank. Thus, it is an isomorphism if and only if it is surjective (see \cite{Orzech} for a generalization of this fact).  But, by Nakayama's lemma, it suffices to check this claim after base change to $k(y)$. But, as \eqref{eq:local-ring-seq} is exact, and $(V_3)_y$ flat, this is equivalent to checking that $(V_1)_{k(y)}\to \ker((V_2)_{k(y)}\to (V_3)_{k(y)})$ is surjective, which is true by assumption.
\end{proof}

\begin{lem} Let $f\colon A\to B$ be a map of rings, and suppose that the finitely generated ideal $J\subseteq A$ is contained in the Jacobson radical. If $f\colon A/J^n\to B/J^nB$ is faithfully flat for all $n$, then the map $\Spec(B)-V(JB)\to \Spec(A)-V(J)$ is surjective on closed points.
\end{lem}
\begin{proof} Let $\mf{p}$ be a prime of $A$ constituting a closed point of $\Spec(A)-V(J)$. We show that $\mf{p}$ is in the image of $f$. Suppose first that $\mf{p}=(0)$. In this case we see that it suffices to show that $\Spec(B)-V(JB)$ is non-empty. If it were empty, then $JB$ is contained in the nilradical of $B$ and as $J$ is finitely generated, this implies that $J^nB=0$ for some $n$. This implies that $J^nB/J^{n+1}B$ is zero, and as $A/J^{n+1}\to B/J^{n+1}B$ is faithfully flat, this implies that $J^n/J^{n+1}$ is zero. Since $J$ is contained in the Jacobson radical of $A$, we deduce from Nakayama's lemma (see \stacks{00DV}) that $J^n$ is zero, but this implies that $\Spec(A)-V(J)$ is empty, which is a contradiction.

In general, let $\ov{A}\defeq A/\mf{p}$, $\ov{B}\defeq B/\mf{p}B$, and $\ov{J}\defeq J\ov{A}$. Then, $\ov{J}$ is in the Jacobson radical of $\ov{A}$, and $\ov{A}/\ov{J}^n\to \ov{B}/\ov{J}^n\ov{B}$ is equal to the base change of $A/J^n\to B/J^n B$ along $A\to \ov{A}$, and so faithfully flat. From the argument in the previous paragraph there exists some $\ov{\mf{q}}$ in $\Spec(\ov{B})-V(\ov{J}\ov{B})$. Let $\mf{q}$ be a prime of $B$ lying over $\ov{\mf{q}}$. Observe that $\mf{q}$ belongs to $\Spec(B)-V(JB)$, and by construction $f(\mf{q})$ lies in $V(\mf{p})\cap (\Spec(A)-V(J))$. But, as $\mf{p}$ is closed in $\Spec(A)-V(J)$, one checks that $V(\mf{p})\cap (\Spec(A)-V(J))=\{\mf{p}\}$ from where the claim follows.
\end{proof}

Finally, we include the following beautiful observation of Ofer Gabber showing that the restriction to closed points is not necessary when the rings are complete.

\begin{prop}[Gabber]\label{prop:Gabber} Let $J\subseteq A$ be a finitely generated ideal, and $f\colon A\to B$ a ring map. Suppose that $A$ (resp.\@ $B$) is complete with respect to $J$ (resp.\@ $JB$), and that $A/J^n\to B/J^nB$ is faithfully flat for all $n$. Then, the map $\Spec(B)\to\Spec(A)$ is surjective.
\end{prop}
\begin{proof} For an ideal $I$ of $A$, let us denote by $I^\mr{ec}$ the ideal $f^{-1}(IB)$, and similarly for the ring maps $f_n\colon A/J^n\to B/J^nB$. Observe that by assumption that each $f_n$ is faithfully flat, we have that $I^\mr{ec}=I$ for any ideal $I\subseteq A/J^n$. Thus, if $I$ is an ideal of $A$ closed in the $J$-adic topology, then by passing to the limit we see that $I^\mr{ec}=I$. 
\vspace*{5 pt}

\noindent\textbf{Claim 1:} For a finitely generated ideal $I\subseteq A$, we have that $\ov{I}^2\subseteq I$, where $\ov I$ is the closure of $I$. 
\begin{proof} More generally we show that for finitely generated ideals $I_1$ and $I_2$ of $A$, the product of their closures is contained in $I_1+I_2$. If $x$ is in the closure of $I_1$ we can write $x=\sum_n x_n$ with $x_n$ in $I_1\cap J^n$. Similarly if $y$ is in the closure of $I_2$ then $y=\sum_n y_n$, with $y_n$ in $I_2\cap J^n$. Thus,
\begin{equation}\label{eq:Gabber-eq}
    xy = \sum_{n\leqslant m}x_n y_m+
\sum_{n>m}x_n y_m.
\end{equation} 
Let $f_k$ be a finite system of generators of $I_1$ and write $x_n=\sum_k x_{nk} f_k$. Then the first term on the right-hand side of \eqref{eq:Gabber-eq} is $\sum_k(\sum_m(\sum_{n\leqslant m}x_{nk}) y_m)f_k$, which is in $I_1$, and similarly the second term on the right-hand side of \eqref{eq:Gabber-eq} is in $I_2$.
\end{proof}

\vspace*{5 pt}

\noindent\textbf{Claim 2:} For every ideal $I$ of $A$, one has $(I^\mr{ec})^2\subseteq I$.
\begin{proof} One reduces to the case when $I$ is finitely generated, in which case we observe that
\begin{equation*}
    (I^\mr{ec})^2\subseteq (\ov{I}^\mr{ec})^2=\ov{I}^2\subseteq I,
\end{equation*}
the second equality by our initial observations, and the second containment by Claim 1.
\end{proof}

If $\mf{p}\subseteq A$ is prime, then from Claim 2 we have $(\mf{p}^\mr{ec})^2\subseteq \mf{p}$. As $\mf{p}$ is radical this implies that $\mf{p}^\mr{ec}\subseteq \mf{p}$ and thus $\mf{p}^\mr{ec}=\mf{p}$, and so $\mf{p}$ is in the image of $\Spec(B)\to\Spec(A)$. Indeed, $B_\mf{p}/\mf{p}B_\mf{p}$ is non-zero as $A_\mf{p}/\mf{p}=A_\mf{p}/f^{-1}(\mf{p}B)A_\mf{p}$ is non-zero and embeds into this ring via $f$.
\end{proof}

\medskip

\paragraph{A criterion to be a prismatic $F$-crystal} We now observe a simple criterion for when an analytic prismatic ($F$-)crystal comes from a prismatic ($F$-)crystal.

\begin{prop}\label{prop:suff-cond-for-being-in-ess-image} Let $\mf{X}$ be a base formal $\mc{O}_K$-scheme and let $\{\Spf(R_i)\}$ be an open cover of $\mf{X}$ where $R_i$ is a (formally framed) base $\mc{O}_K$-algebra. Then, the essential image of
\begin{equation*}
    \cat{Vect}(\mf{X}_\Prism)\to\cat{Vect}^{\mr{an}}(\mf{X}_\Prism),\qquad \cat{Vect}^{\varphi}(\mf{X}_\Prism)\to\cat{Vect}^{\mr{an},\varphi}(\mf{X}_\Prism),
\end{equation*}
consists of those $\mc{W}$ and those $(\mc{V},\varphi_\mc{V})$, such that $(j_{(\mf{S}_{R_i},(E))})_\ast \mc{W}_{(\mf{S}_{R_i},(E))}$ and $(j_{(\mf{S}_{R_i},(E))})_\ast\mc{V}_{(\mf{S}_{R_i},(E))}$ are vector bundles on $\mf{S}_{R_i}$ for all $i$, respectively.
\end{prop}
\begin{proof} It suffices to prove the first claim. Moreover, we are clearly reduced to the case when $\mf{X}=\Spf(R)$.  For $i=1,2,3$, let $\mf{S}_R^{(i)}$ be the $i$-fold self-product of $\mf{S}_R$ in the topos $\cat{Shv}(\mf{X}_\Prism)$ (see \cite[Example 3.4]{DLMS}), and let $p^{(i)}_k$ for $k=1,\ldots, i$ denote the projection maps, which are flat by \cite[Lemma 3.5]{DLMS}. Let $j^{(i)}$, for $i=1,2,3$ denote the inclusion of $U(\mf{S}_R^{(i)},(E))$ into $\Spec(\mf{S}_R^{(i)})$. Observe that we have the $2$-commutative diagram of categories.
\begin{equation*}
    \begin{tikzcd}[column sep=small]
	{\cat{Vect}(\mf{X}_\Prism)} & {\mathbf{Vect}(\mf{S}_{R})} & {\mathbf{Vect}(\mathfrak{S}_R^{(2)})} & {\mathbf{Vect}(\mathfrak{S}_R^{(3)})}\\
    {\mathbf{Vect}^\mathrm{an}(\mathfrak{X}_\Prism)} & {\mathbf{Vect}(U(\mathfrak{S}_R,(E)))} & {\mathbf{Vect}(U(\mathfrak{S}_R^{(2)},(E)))} & {\mathbf{Vect}(U(\mathfrak{S}_R^{(3)},(E))).}
	\arrow[from=1-1, to=1-2]
	\arrow[shift left, from=1-2, to=1-3]
	\arrow[shift left=2, from=1-3, to=1-4]
	\arrow[from=2-1, to=2-2]
	\arrow[shift left, from=2-2, to=2-3]
	\arrow[shift left=2, from=2-3, to=2-4]
	\arrow[from=1-1, to=2-1]
	\arrow["(j^{(1)})^\ast",from=1-2, to=2-2]
	\arrow["(j^{(2)})^\ast",from=1-3, to=2-3]
	\arrow["(j^{(3)})^\ast",from=1-4, to=2-4]
	\arrow[shift right, from=1-2, to=1-3]
	\arrow[shift right, from=2-2, to=2-3]
	\arrow[from=1-3, to=1-4]
	\arrow[shift right=2, from=1-3, to=1-4]
	\arrow[from=2-3, to=2-4]
	\arrow[shift right=2, from=2-3, to=2-4]
\end{tikzcd}
\end{equation*}
The top row is obviously exact, and the bottom row is exact by \cite{Mathew} (cf.\@ \cite[Theorem 2.2]{BhattScholzeCrystals}). Observe that, by inspection, $(p,(E))$ has height $2$ in $\mf{S}_R$, and thus by the flatness of $p_k^{(i)}$, the same holds for $(p,(E))$ in $\mf{S}^{(i)}$ for $i=1,2,3$. Thus, we observe that the maps $(j^{(i)})^\ast$ are fully faithful by Proposition \ref{prop:affine-sheaf-adjunction}. Suppose now that $j^{(1)}_\ast \mc{W}_{(\mf{S}_R,(E))}$ is a vector bundle. Let us observe that 
\begin{equation*} 
(j^{(i)})^\ast (p_k^{(i)})^\ast j^{(1)}_\ast \mc{W}_{(\mf{S}_R,(E))}\simeq (p_k^{(i)})^\ast \mc{W}_{(\mf{S}_R,(E))},
\end{equation*}
for each $i=1,2,3$ and $k=1,\ldots,i$, again by Proposition \ref{prop:affine-sheaf-adjunction}. Thus, by the fully faithfulness of the maps $(j^{(i)})^\ast$, we may pull back the descent data on $\mc{W}_{(\mf{S}_R,(E))}$ relative to the covering $\{p_k^{(2)}\colon U(\mf{S}_R^{(2)},(E))\to U(\mf{S}_R,(E))\}_{k=1,2}$ to descent data on $j^{(1)}_\ast \mc{W}_{(\mf{S}_R,(E))}$ relative to the covering $\{p_k^{(2)}\colon \Spec(\mf{S}_R^{(2)})\to \Spec(\mf{S}_R)\}_{k=1,2}$. This gives an object $\mc{F}$ of $\cat{Vect}(\mf{X}_\Prism)$ whose image in $\cat{Vect}^{\mr{an}}(\mf{X}_\Prism)$ is isomorphic to $\mc{W}$.

We have shown that anything satisfying the condition on the pushforward in the proposition statement is in the essential image of $\cat{Vect}(\mf{X}_\Prism)\to\cat{Vect}^\an(\mf{X}_\Prism)$. Conversely, if $\mc{W}$ is in the essential image, then this pushforward condition holds again by applying Proposition \ref{prop:affine-sheaf-adjunction}.
\end{proof}



\subsection{Tannakian theory} We now show the category of prismatic $F$-crystals (and quasi-syntomic $F$-torsors, see \S\ref{sss:qsyn-torsor}) possesses good Tannakian properties in the sense of Appendix \ref{s:Tannakian-appendix}. We advise the reader to consult there for undefined terminology or notation, and to recall Notation \ref{nota:Section-1}.

\subsubsection{Prismatic \texorpdfstring{$F$-crystals}{F-crystals}}\label{sss:Tannakian-prismatic-F-crystals}

We begin by shortening some notation from \S\ref{section:torsors via tensors}:
\begin{equation*}
      \mc{G}_\Prism\defeq \mc{G}_{\mc{O}_\Prism},\qquad \cat{Tors}_\mc{G}(\mf{X}_\Prism/T)\defeq \cat{Tors}_{\mc{G}_\Prism}(\mf{X}_\Prism/T).
\end{equation*} 
Combining Proposition \ref{prop:vector-bundle-limit-decomp} and Theorem \ref{thm:broshi2}, we see that $\mc{G}_\Prism$ is reconstructible in $(\mf{X}_\Prism/T,\mc{O}_\Prism)$, and every object of $\GVect(\mf{X}_\Prism/T,\mc{O}_\Prism)$ is locally trivial.

\begin{rem}\label{rem: need general T to define local triviality}
    Although our main interest is the case where $T=\ast$, to make sense of the notion of local triviality, it is natural to include the case of general sheaves $T$, at least that of the representable ones. 
\end{rem}

Set $\phi\colon \mc{G}_\Prism\to\mc{G}_\Prism$ to be the morphism of group sheaves associating to every $(A,I)$ the map $\mc{G}(A)\to\mc{G}(A)$ obtained from $\phi_A\colon A\to A$. For an object $\mc{A}$ of $\cat{Tors}_\mc{G}(\mf{X}_\Prism/T)$, denote by $\phi^\ast\mc{A}$ the $\mc{G}_\Prism$-torsor $\phi_\ast\mc{A}$, in the notation of \S \ref{ss:basic-definitions-and-results}. Denote by $\mc{G}_\Prism[\nicefrac{1}{\mc{I}_\Prism}]$ the group sheaf $\mc{G}_{\mc{O}_\Prism[\nicefrac{1}{\mc{I}_\Prism}]}$ and let $\iota\colon \mc{G}_\Prism\to\mc{G}_\Prism[\nicefrac{1}{\mc{I}_\Prism}]$ be the obvious monomorphism. Finally, set $\mc{A}[\nicefrac{1}{\mc{I}_\Prism}]$ to be $\iota_\ast \mc{A}$. 

\begin{defn}
    The category $\cat{Tors}_{\mc{G}}^\varphi(\mf{X}_\Prism/T)$  of \emph{prismatic $\mc{G}$-torsors with $F$-structure} over $T$ has objects $(\mc{A},\varphi_\mc{A})$ where $\mc{A}$ is an object of $\cat{Tors}_{\mc{G}}(\mf{X}_\Prism/T)$ and $\varphi$ is a \emph{Frobenius isomorphism}
    \begin{equation*}
        \varphi_\mc{A}\colon (\phi^\ast\mc{A})[\nicefrac{1}{\mc{I}_\Prism}]\isomto \mc{A}[\nicefrac{1}{\mc{I}_\Prism}],
    \end{equation*}
and morphisms are morphisms in $\cat{Tors}_{\mc{G}}(\mf{X}_\Prism/T)$ commuting with the Frobenii.
\end{defn}

Objects of $\GVect^\varphi(\mf{X}_\Prism/T)$ are pairs $(\omega,\varphi_\omega)$ with $\omega$ an object of $\GVect(\mf{X}_\Prism/T,\mc{O}_\Prism)$ and $\varphi_\omega\colon (\phi^\ast\omega)[\nicefrac{1}{\mc{I}_\Prism}]\isomto \omega[\nicefrac{1}{\mc{I}_\prism}]$ an isomorphism in $\GVect^\lt(\mf{X}_\Prism/T,\mc{O}_\Prism[\nicefrac{1}{\mc{I}_\Prism}])$. For an object $\mc{A}$ in $\cat{Tors}_\mc{G}(\mf{X}_\Prism/T)$ there is a natural identification $\phi^\ast\omega_\mc{A}\simeq \omega_{\phi^\ast\mc{A}}$. So, for an object $(\mc{A},\varphi_\mc{A})$ of $\cat{Tors}^\varphi_\mc{G}(\mf{X}_\Prism/T)$ there is a natural Frobenius $\varphi_{\omega_\mc{A}}$, and $(\omega_\mc{A},\varphi_{\omega_\mc{A}})$ is an object of $\GVect^\varphi(\mf{X}_\Prism/T)$. 

On the other hand, define $\cat{Twist}_{\mc{O}_\Prism|_T}^\varphi(\Lambda_0,\mathds{T}_0)$ to be the groupoid consisting of pairs $((\mc{E},\varphi_\mc{E}),\mathds{T})$ where $(\mc{E},\varphi_\mc{E})$ is an object of $\cat{Vect}^\varphi(\mf{X}_\Prism/T)$, the pair $(\mc{E},\mathds{T})$ is an object of $\cat{Twist}_{\mc{O}_\Prism|_T}(\Lambda_0,\mathds{T}_0)$, and $\mathds{T}$ constitutes a set of tensors on $(\mc{E},\varphi_\mc{E})$ or, equivalently, each tensor in $\mathds{T}$ is \emph{Frobenius invariant}: fixed under the composition $\mc{E}\to \phi^\ast\mc{E}\xrightarrow{\varphi_\mc{E}}\mc{E}$, where the first map sends $x$ to $x\otimes 1$. 

Finally, observe that there is a natural identification 
\begin{equation*}
    \phi^\ast \underline{\Isom}(\mc{O}_\Prism^n,\mc{E})\isomto\underline{\Isom}(\mc{O}_\Prism^n,\phi^\ast\mc{E}),
\end{equation*} 
for an object $\mc{E}$ of $\cat{Vect}(\mf{X}_\Prism/T,\mc{O}_\Prism)$. Moreover, there is an identification between $\underline{\Isom}(\mc{O}_\Prism^n,\mc{E})[\nicefrac{1}{\mc{I}_\Prism}]$ and $\underline{\Isom}(\mc{O}_\Prism[\nicefrac{1}{\mc{I}_\Prism}]^n, \mc{E}[\nicefrac{1}{\mc{I}_\Prism}])$. So, if $((\mc{E},\varphi_\mc{E}),\mathds{T})$ is an object of $\cat{Twist}_{\mc{O}_\Prism|_T}^\varphi(\Lambda_0,\mathds{T}_0)$, then  the $\mc{G}_\Prism$-torsor $\underline{\Isom}((\Lambda_0\otimes_{\Z_p}\mc{O}_\Prism,\mathds{T}_0\otimes 1),(\mc{E},\mathds{T}))$ has a natural Frobenius also denoted $\varphi_\mc{E}$.

Our main Tannakian result concerning prismatic $F$-crystals is the following, where for a prism $(A,I)$ the category $\cat{Tors}^\varphi_\mc{G}(A,I)$ has the obvious meaning.

\begin{prop}\label{prop:varphi-equivariant-GVect-and-tors-identification} The natural functor 
\begin{equation*}
    \cat{Tors}^\varphi_\mc{G}(\mf{X}_\Prism/T)\to\twolim_{\scriptscriptstyle (A,I)\in\mf{X}_\Prism/T}\cat{Tors}^\varphi_\mc{G}(A,I),
\end{equation*}
is an equivalence. Moreover, we have a commuting triangle of equivalences
\begin{equation*}
\labelmargin-{-2pt}\xymatrixcolsep{6.5pc}\xymatrixrowsep{2.5pc}\xymatrix{\cat{Twist}^\varphi_{\mc{O}_\Prism|_T}(\Lambda_0,\mathds{T}_0)\ar[rr]^-{((\mc{E},\varphi_\mc{E}),\mathds{T})\mapsto (\underline{\Isom}((\Lambda_0\otimes_{\Z_p}\mc{O}_\Prism,\mathds{T}_0\otimes 1),(\mc{E},\mathds{T})),\varphi_\mc{E})} & & \cat{Tors}^\varphi_{\mc{G}}(\mf{X}_\Prism/T)\ar[d]^{(\mc{A},\varphi_\mc{A})\mapsto (\omega_\mc{A},\varphi_{\omega_\mc{A}})}\\ & & \GVect^\varphi(\mf{X}_\Prism/T).\ar[ull]^{(\omega,\varphi_\omega)\mapsto ((\omega(\Lambda_0),\varphi_\omega),\omega(\mathds{T}_0))\qquad\quad}}
\end{equation*}
\end{prop}

Results from Appendix \ref{s:Tannakian-appendix} (notably Proposition \ref{prop:torsors-twists-functors} applied to both $\mc{G}_\Prism$ and $\mc{G}_\Prism[\nicefrac{1}{\mc{I}_\Prism}]$) and Proposition \ref{prop:vector-bundle-limit-decomp} reduces the proof of this statement to the claim that the natural functor 
\begin{equation*}
    \cat{Tors}_{\mc{G}}(\mf{X}_\Prism/T)\to \twolim_{\scriptscriptstyle(A,I)\in\mf{X}_\Prism/T}\cat{Tors}_\mc{G}(A)
\end{equation*}
is an equivalence.

To prove this, we introduce an ancillary site. Denote by $\mf{X}_{\Prism,\et}$ the subcategory of $\mf{X}_\Prism$ with the same objects as $\mf{X}_\Prism$ and with those morphisms $(A,I)\to (B,J)$ with $\Spf(B)\to\Spf(A)$ adically \'etale, with the induced topology. By Proposition \ref{prop:etale-prism-extension} (and \stacks{00X5}), the functor $\mf{X}_{\Prism,\et}\to\mf{X}_\Prism$ induces a morphism of sites, and there are equivalences of sites $\mf{X}_{\Prism,\et}/(A,I)\isomto \Spf(A)_\et$.

\begin{prop}\label{prop:torsor-limit-decomp}
The following restriction functors are equivalences of categories 
\begin{equation*}
    \cat{Tors}_{\mc{G}}(\mf{X}_\Prism/T)\to \cat{Tors}_{\mc{G}}(\mf{X}_{\Prism,\et}/T)\to \twolim_{\scriptscriptstyle(A,I)\in\mf{X}_\Prism/T}\cat{Tors}_\mc{G}(A).
\end{equation*}
\end{prop}

\begin{proof} Fix an object $(A,I)$ of $\mf{X}_\Prism$ and set $\mf{X}'_{\Prism}/(A,I)$ to be the full subcategory of $\mf{X}_\Prism/(A,I)$ consisting of covers $(A,I)\to (B,J)$, with the induced topology. We use this notation in a similar way for other sites. The naturally defined functor $\mu_1\colon \mf{X}'_{\Prism}/(A,I)\to  \Spf(A)'_{\fl}$ is continuous and has exact pullback by \stacks{00X6} and \cite[Lemma 3.3]{DLMS}, and we have the equality $\mc{O}_{\mf{X}_\Prism/(A,I)}=\mu_{1\ast}(\mc{O}_{\Spf(A)})$. The same properties hold mutatis mutandis for $\mu_3\colon \mf{X}'_{\Prism,\et}\to\Spf(A)_\et'$. From the following commutative diagram
\begin{equation*}
    \xymatrixcolsep{1.3pc}\xymatrixrowsep{1.3pc}\xymatrix{\mf{X}'_{\Prism}/(A,I)\ar[r]^-{\mu_1}& \Spf(A)'_{\fl}\ar[r]^-{\mu_2} & \Spf(A)_\fl \\ \mf{X}'_{\Prism,\et}/(A,I)\ar[r]_-{\mu_3}\ar[u]_-{\nu_1} & \Spf(A)'_{\et}\ar[r]_-{\mu_4}\ar[u]_-{\nu_2} & \Spf(A)_{\et}\ar[u]_-{\nu_3},}
\end{equation*}
(where the vertical functors are the natural inclusions), we obtain the diagram
\begin{equation*}
  \xymatrixcolsep{1.3pc}\xymatrixrowsep{1.3pc}\xymatrix{\cat{Tors}_{\mc{G}_\Prism}(\mf{X}'_{\Prism}/(A,I))\ar[r]^-{\mu_1^\ast}& \cat{Tors}_\mc{G}(\Spf(A)'_{\fl})\ar[r]^-{\mu_2^\ast} & \cat{Tors}_\mc{G}(\Spf(A)_\fl) \\ \cat{Tors}_{\mc{G}_\Prism}(\mf{X}'_{\Prism,\et}/(A,I))\ar[r]_-{\mu_3^\ast}\ar[u]_{\nu_1^\ast} & \cat{Tors}_\mc{G}(\Spf(A)'_{\et})\ar[r]_{\mu_4^\ast}\ar[u]_{\nu_2^\ast} & \cat{Tors}_\mc{G}(\Spf(A)_{\et})\ar[u]_{\nu_3^\ast}.}
\end{equation*}
Note that $\mu_4^\ast$ and $\mu_2^\ast$ are evidently equivalences, $\nu_2^\ast$ and $\nu_3^\ast$ are equivalences by Corollary \ref{cor:flat-and-etale-torsors-agree}, and $\mu_3^\ast$ is an equivalence by above discussion. By a diagram chase we see that $\mu_1^\ast$ is essentially surjective and so it is an equivalence by Corollary \ref{cor:torsor-equiv-continuous-morphism}. Thus, $\nu_1^\ast$ is an equivalence. We deduce that for any object of $\cat{Tors}_\mc{G}(\mf{X}_\Prism/T)$, its restriction to $(A,I)$ can be trivialized on an \'etale cover, and thus the restriction  $\cat{Tors}_{\mc{G}}(\mf{X}_\Prism/T)\to \cat{Tors}_{\mc{G}}(\mf{X}_{\Prism,\et}/T)$ is an equivalence by Corollary \ref{cor:torsor-equiv-continuous-morphism}.

To prove that $\cat{Tors}_\mc{G}(\mf{X}_{\Prism,\et}/T)\to \twolim\cat{Tors}_\mc{G}(\Spf(A)_\et)$ is an equivalence we exhibit a quasi-inverse. For an object $(\mc A_{(A,I)})$ of the 2-limit category define the object $\mc{A}$ of $\cat{Tors}_\mc{G}(\mf{X}_{\Prism,\et})$ to be the one sending $(A,I)$ to $\mc{A}_{(A,I)}(A)$. This is a presheaf carrying an action of $\mc{G}_\Prism$ and for which the action on $(A,I)$-points is simply transitive whenever non-empty. Thus, we are done as it is simple to see that this presheaf is locally isomorphic to $\mc{G}_\Prism$.
\end{proof}

Finally, we mention the relationship between the above discussion and analytic prismatic $F$-crystals. By Theorem \ref{thm:broshi2}, and the ideas applied in the proof of Proposition \ref{prop:varphi-equivariant-GVect-and-tors-identification}, one has
\begin{equation}\label{eq:GVectan-torsors}
    \GVect^{\an,\varphi}(\mf{X}_\Prism/T) =\twolim_{\scriptscriptstyle(A,I)\in\mf{X}_\Prism/T} \GVect^\varphi(U(A,I))=\twolim_{\scriptscriptstyle(A,I)\in\mf{X}_\Prism/T}\cat{Tors}^\varphi_\mc{G}(U(A,I)),
\end{equation}
and the restriction map 
\begin{equation*}
    \GVect^\varphi(\mf{X}_\Prism/T)\isomto \twolim_{\scriptscriptstyle (A,I)\in\mf{X}_\Prism}\cat{Tors}^\varphi_\mc{G}(A)\to \twolim_{\scriptscriptstyle (A,I)\in\mf{X}_\Prism}\cat{Tors}^\varphi_\mc{G}(U(A,I))\isomto \GVect^{\an,\varphi}(\mf{X}_\Prism/T)
\end{equation*}
is fully faithful by \cite[Proposition 3.7]{GuoReinecke},
where $\cat{Tors}^\varphi_\mc{G}(U(A,I))$ has the obvious meaning. For this reason, we occasionally identify $\GVect^{\an,\varphi}(\mf{X}_\Prism/T)$ with the category $\cat{Tors}_\mc{G}^{\an,\varphi}(\mf{X}_\Prism)$ which, by definition, is the right-most term in \eqref{eq:GVectan-torsors}.

\subsubsection{Quasi-syntomic torsors with \texorpdfstring{$F$}{F}-structure}\label{sss:qsyn-torsor} For the sake of completeness, we compare the material from \S\ref{sss:Tannakian-prismatic-F-crystals} to the analogous theory in the quasi-syntomic setting. 

Abbreviate $\mc{G}_{\mc{O}^\pris}$ to $\mc{G}^\pris$ and $\mc{G}_{\mc{O}^\pris[\nicefrac{1}{\mc{I}^\pris}]}$ to $\mc{G}^\pris[\nicefrac{1}{\mc{I}^\pris}]$. Write $\cat{Tors}_\mc{G}(\mf{X}_\qsyn)$ instead of $\cat{Tors}_{\mc{G}^\pris}(\mf{X}_\qsyn)$. We make similar notational conventions concerning $(\mf{X}_\Qsyn,\mc{O}^\PRIS)$.

\begin{prop}\label{prop:qsyn-prism-torsor-relationship} The following are well-defined equivalences of categories functorial in $\mc{G}$:
    \begin{equation*}
\begin{tikzcd}
    & \cat{Tors}_{\mc{G}}(\mf{X}_\Prism) \arrow[d,shift left,"v_\ast"]\arrow[dl,swap,shift left,"u_\ast"]\arrow[dr,shift left,"\emph{eval.}"]\\
    \cat{Tors}_{
    \mc{G}}(\mf{X}_\Qsyn)\arrow[r,shift left,"\emph{res.}"] & \cat{Tors}_{\mc{G}}(\mf{X}_\qsyn)\arrow[r,shift left,"\emph{eval.}"] &  \displaystyle \twolim_{\scriptscriptstyle\Spf(R)\in\mf{X}_\qrsp}\cat{Tors}_{\mc{G}}(\Prism_R).
\end{tikzcd}
\end{equation*}
\end{prop}
\begin{proof} It suffices to show that $u_\ast$, $\mathrm{res.}$, and $\mathrm{eval.}\colon \cat{Tors}_{\mc{G}}(\mf{X}_\qsyn)\to \twolim\cat{Tors}_{\mc{G}}(\Prism_R)$ are equivalences. To prove that $\mathrm{eval.}$ is an equivalence, take an open cover $\{\Spf(R_i)\}$ of $\mf{X}$ and a qrsp cover $R_i\to S_i$. Let $S_i^\bullet$ denote the objects of the \v{C}ech nerve of $R_i\to S_i$, which consist of qrsp rings by \cite[Lemma 4.30]{BMS-THH}. Then, by descent we have that the natural evaluation functor $\cat{Tors}_\mc{G}(\mf{X}_\qsyn)\to \twolim \cat{Tors}_\mc{G}(\Prism_{S_i^\bullet})$ is an equivalence, where we have implicitly used \cite[Proposition 3.30]{AnschutzLeBrasDD}. The fact that $\mathrm{eval.}$ is an equivalence easily follows.

Set $\mf{X}'_{\qsyn}$ to have the same objects as $\mf{X}_\qsyn$, but whose morphisms are required to be quasi-syntomic. From \cite[Lemma 4.16]{BMS-THH} and \stacks{00X6}, the inclusions $\mf{X}'_\qsyn\to\mf{X}_\qsyn$ and $\mf{X}'_\qsyn\to\mf{X}_\Qsyn$ induce morphisms of sites. As $\mf{X}'_\qsyn$ has the same set of covers of $\mf{X}$ as $\mf{X}_{\Qsyn}$, from Corollary \ref{cor:torsor-equiv-continuous-morphism} we deduce that the functors $\cat{Tors}_{\mc{G}}(\mf{X}_\Qsyn)\to \cat{Tors}_{\mc{G}}(\mf{X}'_{\qsyn})$ and $\cat{Tors}_{\mc{G}}(\mf{X}_\qsyn)\to \cat{Tors}_{\mc{G}}(\mf{X}'_{\qsyn})$ are equivalences, and so is their composition $\mathrm{res.}$

By Corollary \ref{cor:torsor-equiv-cocontinuous-morphism} to show that $u_\ast$ is an equivalence, it suffices to show that for an object $\mc{A}$ of $\cat{Tors}_\mc{G}(\mf{X}_\Prism)$, that $u_\ast(\mc{A})$ is locally non-empty. Passing to a qrsp cover $\{\Spf(R_i)\to \mf{X}\}$ we may assume that $\mf{X}=\Spf(R)$ where $R$ is qrsp. By Proposition \ref{prop:torsor-limit-decomp}, $\cat{Tors}_{\mc{G}}(\mf{X}_\Prism)= \cat{Tors}_\mc{G}(\Prism_{R})$. Pulling back $\mc{A}$ along the closed immersion $i\colon \Spec(R)\to \Spec(\Prism_R)$ as in \cite[Theorem 3.29]{AnschutzLeBrasDD}, gives a $\mc{G}$-torsor $i^\ast\mc{A}$ on $\Spf(R)_\et$. Let $R\to S$ be a $p$-adically \'etale morphism such that $(i^\ast\mc{A})(S)$ is non-empty. By Lemma \ref{lem:etale-over-qrsp} we see $S$ is qrsp, and so $(u_\ast\mc{A})(S)=\mc{A}(\Prism_S,I_S)$. But, the pair $(\Prism_S,\ker(\Prism_S\to S))$ is Henselian by \cite[Lemma 4.28]{AnschutzLeBrasDD}, and so $\mc{A}(\Prism_S,I_S)$ is non-empty (e.g., by \cite[Theorem 2.1.6]{BCLoopTorsors}).
\end{proof}

The proof of the following is obtained mutatis mutandis from the proof of Proposition \ref{prop:qsyn-prism-torsor-relationship}.

\begin{prop}[{cf.\@ \cite[Proposition 4.4]{AnschutzLeBrasDD} and \cite[Proposition 2.13 and Proposition 2.14]{BhattScholzeCrystals}}]\label{prop:qsyn-prism-vec-relationship} 
The following are well-defined rank-preserving bi-exact $\Z_p$-linear $\otimes$-equivalences:
    \begin{equation*}
\begin{tikzcd}
    & \cat{Vect}(\mf{X}_\Prism,\mc{O}_\Prism) \arrow[d,shift left,"v_\ast"]\arrow[dl,swap,shift left,"u_\ast"]\arrow[dr,shift left,"\emph{eval.}"]\\
    \cat{Vect}(\mf{X}_\Qsyn,\mc{O}^\PRIS)\arrow[r,shift left,"\emph{res.}"] & \cat{Vect}(\mf{X}_\qsyn,\mc{O}^\pris)\arrow[r,shift left,"\emph{eval.}"] &  \displaystyle \twolim_{\scriptscriptstyle\Spf(R)\in\mf{X}_\qrsp}\cat{Vect}(\Prism_R).
\end{tikzcd}
\end{equation*}
\end{prop}

Set $\phi=v_\ast(\phi)$ on $\mc{G}^\pris=v_\ast(\mc{G}_\Prism)$. For an object $\mc{B}$ of $\cat{Tors}_{\mc{G}}(\mf{X}_\qsyn)$, define $\mc{B}[\nicefrac{1}{\mc{I}^\pris}]$ as the pushforward along $\mc{G}^\pris\to \mc{G}^\pris[\nicefrac{1}{\mc{I}^\pris}]$, and $\phi^\ast\mc{B}\defeq \phi_\ast\mc{B}$. There are identifications $v_\ast\phi^\ast\mc{A}=\phi^\ast v_\ast\mc{A}$ and $v_\ast(\mc{A}[\nicefrac{1}{\mc{I}_\Prism}])=(v_\ast\mc{A})[\nicefrac{1}{\mc{I}^\pris}]$. Define a \emph{quasi-syntomic $\mc{G}$-torsor with $F$-structure} to be a pair $(\mc{B},\varphi)$ where $\mc{B}$ is an object of $\cat{Tors}_{\mc{G}}(\mf{X}_\qsyn)$ and $\varphi$ a \emph{Frobenius isomorphism}
\begin{equation*}
    \varphi\colon (\phi^\ast \mc{B})[
    \nicefrac{1}{\mc{I}^\pris}]\isomto \mc{B}[\nicefrac{1}{\mc{I}^\pris}].
\end{equation*} 
A morphism of quasi-syntomic $\mc{G}$-torsors with $F$-structure is a morphism of $\mc{G}^\pris$-torsors commuting with the Frobenii. Denote the category of such objects by $\cat{Tors}^\varphi_{\mc{G}}(\mf{X}_\qsyn)$. One can similarly define the categories $\cat{Vect}^\varphi(\mf{X}_\qsyn)$ and $\GVect^\varphi(\mf{X}_\qsyn)$, as well as their analogues for $\mf{X}_\Qsyn$ or $\mf{X}_\qrsp$.

Using Proposition \ref{prop:varphi-equivariant-GVect-and-tors-identification}, Proposition \ref{prop:qsyn-prism-torsor-relationship}, and Proposition \ref{prop:qsyn-prism-vec-relationship} one may deduce the following.

\begin{prop}\label{prop:comparing-prismatic-and-qsyn-F-crystals}
    There is a commuting diagram of well-defined equivalences
    \begin{equation*}
        \xymatrixcolsep{1.3pc}\xymatrixrowsep{1.3pc}\xymatrix{\cat{Tors}^\varphi_{\mc{G}}(\mf{X}_\Prism)\ar[r]\ar[d]_{v_\ast} & \GVect^\varphi(\mf{X}_\Prism)\ar[d]^{v_\ast}\\ \cat{Tors}^\varphi_{\mc{G}}(\mf{X}_\qsyn)\ar[r] & \GVect^\varphi(\mf{X}_\qsyn)}
    \end{equation*} 
    which is functorial in $\mc{G}$. Similar assertions hold with $\mf{X}_\qsyn$ replaced by $\mf{X}_\Qsyn$ or $\mf{X}_\qrsp$.
\end{prop}

\section{\texorpdfstring{$\mc{G}$}{G}-objects in the category of crystalline local systems}

In this section we derive Tannakian analogues of the results of \cite{BhattScholzeCrystals}, \cite{GuoReinecke}, and \cite{DLMS}. Unless stated otherwise, we still use the notation given in Notation \ref{nota:Section-1}.

\subsection{The category of \texorpdfstring{$\mc{G}(\Z_p)$}{G(Zp)}-local systems}

In the statement of the main result of this section it will be helpful to have a clear notion of a $\mc{G}(\Z_p)$-local system (on a scheme or an adic space). 

\subsubsection{\texorpdfstring{$\mc{G}(\Z_p)$}{G(Zp)}-local systems on a scheme}\label{ss:G(Z_p)-local-system-scheme}

For a locally topologically Noetherian scheme $S$ (in the sense in \cite[Definition 6.6.9]{BhattScholzeProetale}), recall from \cite[Definition 4.1.1]{BhattScholzeProetale} that a morphism of schemes $f\colon S'\to S$ is \emph{weakly \'etale} if $f$ and $\Delta_f$ are flat. Moreover, we can associate to $S$ the \emph{pro\'etale site} $S_\proet$ which is the full subcategory of $S_{\mathrm{fpqc}}$ consisting of weakly \'etale morphisms $S'\to S$, with the induced topology.

We will be interested in the ringed site $(S_\proet,\underline{\Z_p}_S)$ , where 
\begin{equation*}
    \underline{\Z_p}_S\defeq R\varprojlim \underline{\Z/p^n}_S=\varprojlim \underline{\Z/p^n}_S,
\end{equation*}
where the second equality follows from \cite[Proposition 3.2.3, Proposition 3.1.10, and Proposition 4.2.8]{BhattScholzeProetale}. We often use the notation $\Z/p^\infty\defeq \Z_p$. More generally, for a topological space $T$, denote by $\underline{T}_S$ (or just $\underline{T}$ when $S$ is clear from context) the sheaf (see \cite[Lemma 4.2.12]{BhattScholzeProetale}) 
\begin{equation*}
    \underline{T}_S\colon S_\proet\to \cat{Set},\qquad U\mapsto \Hom_\text{cont.}(U,T).
\end{equation*}
If $T$ is totally disconnected, then $\Hom_\text{cont.}(U,T)$ is equal to $\text{Hom}_\text{cont.}(\pi_0(U),T)$, and so if $T=\varprojlim T_i$ with each $T_i$ finite then $\underline{T}=\varprojlim \underline{T_i}$ where each $\underline{T_i}$ is the constant sheaf.

\begin{lem}\label{lem:G-for-proet-scheme} With notation as in \emph{(\ref{eq: definition of mc G sub O})}, there is an identification $\mc{G}_{\underline{\Z/p^n}}\isomto \underline{\mc{G}(\Z/p^n)}$, compatible in $n$ and functorial in $\mc{G}$, for $n\in\bb{N}\cup\{\infty\}$.
\end{lem} 
\begin{proof} We handle only the case where $n$ is finite as the case for $n=\infty$ follows by passing to the limit. We provide for $U$ in $S_\proet$ an isomorphism $\mc{G}_{\underline{\Z/p^n}}(U)\isomto \underline{\mc{G}(\Z/p^n)}(U)$ bi-functorial in $U$ and $\mc{G}$ and compatible in $n$. Let $\{D_i\}$ be the set of discrete quotient spaces of $U$. Then, for any discrete space $X$ there is an identification between $\Hom_\text{cont.}(U,X)$ and $\varinjlim \Hom(D_i,X)$, bi-functorial in $U$ and $X$. As $\mc{G}$ is locally of finite presentation over $\Z_p$
\begin{equation*}
    \begin{aligned}\mc{G}_{\underline{\Z/p^n}}(U) &=\mc{G}(\Hom_\text{cont.}(U,\Z/p^n))\\ &=\varinjlim \mc{G}(\Hom(D_i,\Z/p^n))\\ &= \varinjlim \text{Hom}_{\cat{Alg}_{\Z_p}}(A_\mc{G},\text{Hom}(D_i,\Z/p^n)),\end{aligned}
\end{equation*}
(cf.\@ \stacks{01ZC}), where $A_\mc{G}\defeq \mc{O}_\mc{G}(\mc{G})$. On the other hand we see that 
\begin{equation*}
    \begin{aligned}
        \underline{\mc{G}(\Z_p/p^n\Z)}(U) &=  \text{Hom}_\text{cont.}(U,\mc{G}(\Z/p^n))\\ &= \varinjlim \text{Hom}(D_i,\mc{G}(\Z/p^n))\\ &= \varinjlim \Hom(D_i,\Hom_{\cat{Alg}_{\Z_p}}(A_\mc{G},\Z/p^n)).
    \end{aligned}
\end{equation*}
Thus, we are done via the natural isomorphism of groups 
\begin{equation*}
    \Hom(D_i,\Hom_{\cat{Alg}_{\Z_p}}(A_\mc{G},\Z/p^n))\to\text{Hom}_{\cat{Alg}_{\Z_p}}(A_\mc{G},\text{Hom}(D_i,\Z/p^n))
\end{equation*}
bi-functorial in $D_i$ and $\mc{G}$, given by currying.
\end{proof}

Denote by $\cat{Loc}_{\Z/p^n}(S)$ the category $\cat{Vect}(S_\proet,\underline{\Z/p^n})$ of \emph{$\Z/p^n$-local systems}. By \cite[Corollary 5.1.5 and Proposition 6.8.4]{BhattScholzeProetale}, this is equivalent to $\cat{Loc}_{\Z/p^n}(S_\et)$ (see \cite[Expos\'e VI]{SGA5}) as an exact $\Z_p$-linear $\otimes$-category. We refer to objects of $\GLoc_{\Z/p^n}(S)$ as \emph{$\mc{G}(\Z/p^n)$-local systems}.

It will also be convenient to consider the category $\cat{Loc}_{\Q_p}(S)$ of \emph{$\Q_p$-local systems} on $S$. This has the same objects as $\cat{Loc}_{\Z_p}(S)$, denoted by $\bb{L}$ of $\bb{L}[\nicefrac{1}{p}]$ when clarity is necessary, but with
\begin{equation}\label{eq:Q_p-local-systems}
    \cat{Hom}_{\cat{Loc}_{\Q_p}(S)}(\bb{L}[\nicefrac{1}{p}],\bb{L}'[\nicefrac{1}{p}])\defeq \Gamma(S_\proet,\mc{H}om(\bb{L},\bb{L}')[\nicefrac{1}{p}]),
\end{equation}
where $\mc{H}om(\bb{L},\bb{L}')$ is the internal Hom in $\cat{Loc}_{\Z_p}(S)$. The category $\cat{Loc}_{\Q_p}(S)$ admits a fully faithful embedding into $\cat{Vect}(S_\proet,\underline{\Q_p})$, and so inherits an exact $\Q_p$-linear $\otimes$-structure.

\begin{rem}\label{rem:no-lattice} The essential image of $\cat{Loc}_{\Q_p}(S)\to \cat{Vect}(S_\proet,\underline{\Q_p})$, consists of those $\underline{\Q_p}$-vector bundles which admit a $\underline{\Z_p}$-lattice. Equivalently, assuming $S$ is connected, this corresponds to the embedding $\cat{Rep}^\mr{cont.}_{\Q_p}(\pi_1^\et(S,\ov{s}))\to\cat{Rep}_{\Q_p}^\mr{cont.}(\pi_1^\proet(S,\ov{s}))$,  where $\pi_1^\proet(S,\ov{s})$ is the pro\'etale fundamental group as in \cite[\S7]{BhattScholzeProetale} and $\ov{s}$ is a geometric point. If $S$ is geometrically unibranch then $\pi_1^\proet(S,\ov{s})=\pi_1^\et(S,\ov{s})$ (see \cite[Lemma 7.4.10]{BhattScholzeProetale}), and so by compactness any representation $\pi_1^\proet(S,\ov{s})\to \GL_n(\Q_p)$ factorizes through a $\Z_p$-lattice and so $\cat{Loc}_{\Q_p}(S)\to \cat{Vect}(S_\proet,\underline{\Q_p})$ is an equivalence. But, this is not true in general (e.g.\@ if $S$ is the projective nodal curve, then $\pi_1^\proet(S,\ov{s})$ is isomorphic to $\Z$ with the discrete topology).
\end{rem}

\begin{prop}\label{prop:tannakian-formalism-proetale-scheme} For every $n\in\bb{N}\cup\{\infty\}$, the group $\mc{G}$ is reconstructible in $(S_\proet,\underline{\Z/p^n})$ and every object of $\GLoc_{\Z/p^n}(S)$ is locally trivial.
\end{prop}
\begin{proof}Assume first that $S$ is locally Noetherian. For the first statement, we deal only with the case $n=\infty$. Let $U$ in $S_\proet$ be arbitrary. One may identify the natural map $\mc{G}_{\underline{\Z_p}}(U)\to \underline{\Aut}(\omega_\triv)(U)$ with the map from $\Hom_\text{cont.}(\pi_0(U),\mc{G}(\Z_p))$ to the system $(g_\Lambda)$ of elements of $\Hom_\text{cont.}(\pi_0(U),\GL(\Lambda))$ for $\Lambda$ in $\cat{Rep}_{\Z_p}(\mc{G})$, which are compatible in $\Lambda$. The projection $(g_\Lambda)\mapsto g_{\Lambda_0}$ from $\underline{\Aut}(\omega_\triv)(U)$ to $\Hom_\text{cont.}(\pi_0(U),\GL(\Lambda_0))$ is injective as every object $\Lambda$ of $\cat{Rep}_{\Z_p}(\mc{G})$ is a subquotient of $\Lambda_0^{\otimes}$ (see \cite[Proposition 12]{dosSantos}). As the diagram 
    \begin{equation*}
        \xymatrixcolsep{1.3pc}\xymatrixrowsep{1.3pc}\xymatrix{\mc{G}_{\underline{\Z_p}}(U)\ar[r]^-a\ar[rd]_-c & \underline{\Aut}(\omega_\triv)(U)\ar[d]^-b\\ & \Hom_\text{cont.}(\pi_0(U),\GL(\Lambda_0))}
    \end{equation*}
    commutes, and $a$ and $b$ are injective, to show that $a$ is an isomorphism it suffices to show that $\mathrm{im}(b)\subseteq \mathrm{im}(c)$. But, by functoriality $\mr{im}(b)$ fixes $\mathds{T}$ and thus lies in $\mr{im}(c)$. 

    Let $\omega$ be an object of $\GLoc_{\Z_p}(S)$. For $n\in\mathbb{N}\cup\{\infty\}$ consider the sheaf $\underline{\Isom}(\omega_\triv,\omega_n)$, associating to a locally Noetherian $S$-scheme $f\colon T\to S$ the set of isomorphisms $\omega_\triv\to \omega_{n,T}$, where $\omega_{n,T}(\Lambda)\defeq f^{-1}(\omega(\Lambda)\otimes_{\underline{\Z_p}}\underline{\Z/p^n})$ is considered as an object of $\GLoc_{\Z/p^n}(T)$. To prove the second claim it suffices to show that $\underline{\Isom}(\omega_\triv,\omega_\infty)$ is representable by a weakly \'etale cover of $S$. Moreover, as $\underline{\Isom}(\omega_\triv,\omega_\infty)$ is the limit of $\underline{\Isom}(\omega_\triv,\omega_n)$ (cf.\@ \cite[Proposition 6.8.4]{BhattScholzeProetale}) it further suffices to show that $\underline{\Isom}(\omega_\triv,\omega_n)$ is represented by a finite 
    \'etale cover of $S$ for every $n$ in $\bb{N}$. To see this, observe that the natural map 
    \begin{equation*}
        \underline{\Isom}(\omega_\triv,\omega_n)\to \underline{\Isom}(\Lambda_0\otimes_{\Z_p} \underline{\Z/p^n},\omega_n(\Lambda_0))
    \end{equation*}
    is a closed embedding, cut out by the intersection of the following conditions (with terminology from \cite[Definition 10 and Proposition 12]{dosSantos}) on an $f$ in $\underline{\Isom}(\Lambda_0\otimes_{\Z_p} \underline{\Z/p^n},\omega_n(\Lambda_0))$: for every tuple $(a,b,W,U,q)$, where $a,b$ are multi-indices, $W$ is a special subrepresentation of $(\Lambda_0)^a_b$, and $q\colon W\to U$ is a surjection of representations, $f$ and $f^{-1}$ preserve $W$ and $\ker q$.  Thus, $\underline{\Isom}(\omega_\triv,\omega_n)\to S$ is finite. Moreover, by the topological invariance of the pro-\'etale topos (see \cite[Lemma 5.4.2]{BhattScholzeProetale}), pullback along $i\colon\Spec(A/I)\to\Spec(A)$, for a square-zero ideal $I$ of a Noetherian ring $A$, defines an equivalence 
    \begin{equation*}
        i^{-1}\colon \cat{Loc}_{\Z/p^n}(\Spec(A))\to \cat{Loc}_{\Z/p^n}(\Spec(A/I))
    \end{equation*}
    for $n\in\mathbb{N}\cup\{\infty\}$. From this we deduce that $\underline{\Isom}(\omega_\triv,\omega_n)\to S$ is formally \'etale in the sense of \stacks{02HG}, and thus finite \'etale by \stacks{02HM}.

    To show that $\underline{\Isom}(\omega_\triv,\omega_n)\to S$ is surjective, it suffices to show that the pullback to each geometric point of $S$ is a trivial torsor, and so we may assume that $S$ is the spectrum of an algebraically closed field. In this case, there is a bi-exact $\Z_p$-linear $\otimes$-equivalence between $\cat{Loc}_{\Z/p^n}(S)$ and $\cat{Vect}(\Z/p^n)$. But, by Theorem \ref{thm:broshi2}, if $\nu_n$ belongs to $\GVect(\Z/p^n)$ then $\underline{\Isom}(\omega_\triv,\nu_n)$ is a $\mc{G}$-torsor. But, $H^1(\Spec(\Z/p^n),\mc{G})$ is trivial (e.g.\@ by \cite[Corollary 17.98]{MilneGroups} and \cite[Theorem 2.1.6]{BCLoopTorsors}), and so the claim follows.

To remove the locally Noetherian hypothesis, we may proceed as follows. We may assume that $n$ is in $\mathbb{N}$ and $S$ is connected. Consider $W$ in $\cat{Rep}_{\Z_p}(\mc{G})$ with a saturated injection $W \hookrightarrow A_\mc{G}^*=\Hom_{\Z_p}(\mc{O}_\mc{G}(\mc{G}),\Z_p)$ whose image in $A_\mc{G}^*/p^n$ contains $\mc{G}(\Z/p^n)$. We may replace $S$ by a finite \'etale connected cover that trivializes $\omega_n(W)$. We claim then that $\omega_n(V)$ is trivial for all $V$. By the locally Noetherian case, for each $s$ in $S$, we have an isomorphism 
    \begin{equation*}
        \psi \colon \underline{\Isom}(\omega_\triv,\omega_n)_s \times^{\mc{G}(\Z/p^n)} W(\Z/p^n) \simeq \omega_n(W)_s.
    \end{equation*}
    Take a trivial $\mc{G}(\Z/p^n)$-torsor $\mc{T}$ over $S$ and an isomorphism $\mc{T} \times^{\mc{G}(\Z/p^n)} W(\Z/p^n) \simeq \omega_n(W)$ extending $\psi$. For an object $V$ of $\cat{Rep}_{\Z_p}(\mc{G})$ and  $v$ in $V(\Z_p)$, let $f_v \colon A_\mc{G}^* \to V$ be the morphism induced by the map $\mc{G} \to V$ sending  $g$ to $gv$. Let $\mathrm{ev}_1$ in $A_\mc{G}^*/p^n$ be the unit section. Then 
    \begin{equation*}
        \mc{T} \xrightarrow{\id \times \mathrm{ev}_1} \mc{T} \times^{\mc{G}(\Z/p^n)} W(\Z/p^n) \simeq \omega_n(W) \xrightarrow{\omega_n(f_v|_W)} \omega_n(V)
    \end{equation*} 
    gives a map $\mc{T} \times^{\mc{G}(\Z/p^n)} V(\Z/p^n) \to \omega_n(V)$. One checks this is an isomorphism over $s$, and so is an isomorphism as the source and target are finite \'etale over $S$.
\end{proof}

Given Lemma \ref{lem:G-for-proet-scheme}, we write $\cat{Tors}_{\mc{G}(\Z/p^n)}(S_\proet)$ instead of $\cat{Tors}_{\mc{G}_{\underline{\Z/p^n}}}(S_\proet)$, and call objects of this category \emph{$\mc{G}(\Z/p^n)$-torsors}. By Proposition \ref{prop:tannakian-formalism-proetale-scheme} and Proposition \ref{prop:torsors-twists-functors} we know that there is a natural equivalence of categories between $\GLoc_{\Z/p^n}(S)$ and $\cat{Tors}_{\mc{G}(\Z/p^n)}(S_\proet)$. Explicitly, this functor associates to a $\mc G(\Z/p^n)$-torsor $\mc A$ the $\mc{G}(\Z/p^n)$-local system $\omega_\mc{A}$ given by $\omega_\mc{A}(\Lambda)\defeq \mc A\wedge^\mc G\underline{\Lambda}$.

By a \emph{$G(\Q_p)$-local system} on $S$, we mean an exact $\Q_p$-linear $\otimes$-functor $\cat{Rep}_{\Q_p}(G)\to\cat{Loc}_{\Q_p}(S)$, the category of which we denote by $G\text{-}\cat{Loc}_{\Q_p}(S)$. The following trivial observation allows us to define a natural functor $(-)[\nicefrac{1}{p}]\colon \GLoc_{\Z_p}(S)\to G\text{-}\cat{Loc}_{\Q_p}(S)$. Let us denote by $\cat{Rep}_{\Z_p}(\mc{G})[\nicefrac{1}{p}]$ the category which has the same objects as $\cat{Rep}_{\Z_p}(\mc{G})$, denoted by either $\Lambda$ or $\Lambda[\nicefrac{1}{p}]$ for clarity, but where we set $\Hom_{\cat{Rep}_{\Z_p}(\mc{G})[\nicefrac{1}{p}]}(\Lambda[\nicefrac{1}{p}],\Lambda'[\nicefrac{1}{p}])$ to be $\cat{Hom}_{\cat{Rep}_{\Z_p}(\mc{G})}(\Lambda,\Lambda')[\nicefrac{1}{p}]$.

\begin{lem}\label{lem:rational-integral-rep-equiv} The functor
\begin{equation*}
    \cat{Rep}_{\Z_p}(\mc{G})[\nicefrac{1}{p}]\to \cat{Rep}_{\Q_p}(G),\qquad \Lambda\mapsto \Lambda[\nicefrac{1}{p}],
\end{equation*}
is an equivalence of categories. 
\end{lem}
\begin{proof}Let $\Lambda$ and $\Lambda'$ be objects of $\cat{Rep}_{\Z_p}(\mc{G})$, viewed as right comodules of $A_\mc{G}\defeq \mc{O}_G(\mc{G})$ and view $\Lambda[\nicefrac{1}{p}]$ and $\Lambda'[\nicefrac{1}{p}]$ as right comodules of $A_\mc{G}[\nicefrac{1}{p}]=\mc{O}_G(G)$. Then, it suffices to show that under the natural isomorphism $\Hom_{\Z_p}(\Lambda,\Lambda')[\nicefrac{1}{p}]\to \Hom_{\Q_p}(\Lambda[\nicefrac{1}{p}],\Lambda'[\nicefrac{1}{p}])$ that the subsets of comodule homomorphisms are matched. But, this is obvious as $A_\mc{G}$ is a flat $\Z_p$-module. To show that this functor is essentially surjective, let $V$ be an $A_\mc{G}[\nicefrac{1}{p}]$-comodule. An $A_\mc{G}$-comodule lattice is obtained using \cite[Lemma 3.1]{Broshi} applied to $V$, by taking an $A_\mc{G}$-comodule, locally free (and thus free) as a $\Z_p$-module, containing any given $\Z_p$-lattice of $V$.
\end{proof}

Finally, it is useful to have a more concrete description of torsors for a pro-finite group $H$ on $S_\proet$. We set the following notation. 

\begin{nota}\label{nota:H} Let $H=\varprojlim H_n$ be a pro-finite group, where each $H_n$ is a finite group.
\end{nota}
Consider a projective system of schemes $\{X_n\}$, each equipped with the structure of a principal homogeneous space over $S$ for $H_n$ (see \stacks{049A}). The quotient sheaf $X_n/K_n$, where $K_n\defeq\ker(H_n\to H_{n-1})$, is representable by \cite[Expos\'e V, \S7, Th\'eor\`eme 7.1]{SGA3-1}, and $X_n/K_n\to S$ is naturally a principal homogeneous space for $H_{n-1}$. We say $(X_n)$ is an \emph{$H$-covering} if each $X_n\to X_{n-1}$ is $K_n$-equivariant (when the target is given the trivial action) and the induced map $X_n/K_n\to X_{n-1}$ is an $H_{n-1}$-equivariant isomorphism. A morphism of $H$-coverings $(X_n)\to (Y_n)$ is a compatible family of $H_n$-equivariant morphisms $X_n\to Y_n$.

An $S$-scheme $X$ equipped with an action of $H$ is a \emph{principal homogeneous space} for $H$ if $h_X$ with the induced action of $\underline{H}$ is an $\underline{H}$-torsor. Morphisms of principal homogeneous spaces are $H$-equivariant morphisms of $S$-schemes. If $(X_n)$ is an $H$-covering then $X_\infty\defeq \varprojlim X_n$ is a principal homogeneous space for $H$. Conversely, if $X$ is a principal homogeneous space for $H$, the system $(X_n)$ with $X_n\defeq X/K_n$ (a scheme by Proposition \ref{lem:torsor-representable}) is an $H$-covering with $X_\infty= X$.

Applying \cite[Lemma 7.3.9]{BhattScholzeProetale} to $\mc{T}\times^{\underline{H}}\underline{H}_n$ for every $n$, every object $\mc{T}$ of $\cat{Tors}_{\underline{H}}(S_\proet)$ is representable. We summarize the above as follows.

\begin{prop}\label{prop:G(Zp)-covering-PHS-torsor-comparison-scheme-case}
    For locally topologically Noetherian $S$, the following are equivalences:
 \begin{equation*} 
	\left\{\begin{matrix}H\emph{-coverings}\\ \emph{of }S\end{matrix}\right\}\xrightarrow{(X_n)\mapsto X_\infty} \left\{\begin{matrix}\emph{Principal homogenous}\\ \emph{spaces for }H\emph{ on }S\end{matrix}\right\} \xrightarrow{X\mapsto h_X}\cat{Tors}_{\underline{H}}(S_\proet).
\end{equation*}
\end{prop}

\subsubsection{\texorpdfstring{$\mc{G}(\Z_p)$}{G(Zp)}-local systems on an adic space}\label{ss:G(Z_p)-local-systems-on-adic-spaces}

Let $X$ be a locally Noetherian adic space over $\Q_p$ (cf.\@ \cite[p.\@ 17]{ScholzepadicHT}), and let $X_\proet$ be the pro-\'etale site as in \cite[\S5.1]{BMSI}. As in \cite{ScholzepadicHT}, we call an object of $X_\proet$ \emph{affinoid perfectoid} if it can be represented as $(\Spa(R_i,R_i^+))$ with finite \'etale surjective transition maps, and the Huber pair $(R,R^+)=((\varinjlim R_i)^\wedge,(\varinjlim R_i^+)^\wedge)$ (endowed with the unique topology such that each $R_i\to R$ is adic) is perfectoid. In this case, $\Spa(R,R^+)\sim \varprojlim \Spa(R_i,R_i^+)$ (see \cite[\S2.4]{ScholzeWeinstein}). The affinoid perfectoid objects of $X_\proet$ form a basis of $X_\proet$ (cf.\@ \cite[Proposition 4.8]{ScholzepadicHT}). The site $X_\proet$ is not subcanonical (see \cite[Example 4.1.7]{ALY1P2}), but $h_Y$ and $h_Y^\#$ have the same value on affinoid perfectoid objects (see \cite[Proposition 4.1.8]{ALY1P2}). So, we abusively conflate the two.

We will again be interested in the ringed site $(X_\proet,\underline{\Z_p}_X)$, where 
\begin{equation*}
    \underline{\Z_p}_X=R\varprojlim \underline{\Z/p^n}_X=\varprojlim \underline{\Z/p^n}_X\in \cat{Shv}(X_\proet),
\end{equation*}
where $\underline{\Z/p^n}_X$ is the constant sheaf on $X_\proet$ associated to $\Z/p^n$, and the second equality follows now from  \cite[Proposition 8.2]{ScholzepadicHT}. Again, more generally, for a topological space $T$, we denote by $\underline{T}_X$ (or just $\underline{T}$ when $X$ is clear from context) the sheaf\footnote{That this is a sheaf can be deduced from the fact that if $\{Y_i\to Y\}$ is a cover, then $\coprod_i |Y_i|\to |Y|$ is a quotient map. By \cite[Lemma 2.5]{ScholzeECD}), to prove this it suffices to each $|Y_i|\to |Y|$ is generalizing, but this follows by combining \cite[Lemma 1.1.10]{HuberEC} and \cite[Lemma 2.11]{ScholzeECD}.}
\begin{equation*}
    \underline{T}_{X}\colon X_\proet\to\cat{Set},\qquad Y\mapsto \text{Hom}_\text{cont.}(Y,T).
\end{equation*}
If $T$ is totally disconnected, then $\Hom_\text{cont.}(Y,T)$ equals $\text{Hom}_\text{cont.}(\pi_0(Y),T)$, and if $T=\varprojlim T_i$ with each $T_i$ finite then $\underline{T}=\varprojlim \underline{T_i}$ where each $\underline{T_i}$ is the constant sheaf. 

For $n\in\mathbb{N}\cup\{\infty\}$, write $\cat{Loc}_{\Z/p^n}(X)$ for the category $\cat{Vect}(X_\proet,\underline{\Z/p^n})$ of \emph{$\Z/p^n$-local systems}. By \cite[Proposition 8.2]{ScholzepadicHT}, this category is equivalent to $\cat{Loc}_{\Z/p^n}(X_\et)$ as an exact $\Z_p$-linear $\otimes$-category. The objects of $\GLoc_{\Z/p^n}(X)$ are called \emph{$\mc{G}(\Z/p^n)$-local systems}. The following is proven, mutatis mutandis, as in Lemma \ref{lem:G-for-proet-scheme} and (the locally Noetherian case of) Proposition \ref{prop:tannakian-formalism-proetale-scheme} (or alternatively, one can use Remark \ref{rem: torsors on the diamond associated to a locally noetherian adic space} and Proposition \ref{prop:G-for-proet-diamonds}, which is proven independently).

\begin{prop}\label{prop:G-for-proet-adic}
    For $n\in\bb{N}\cup\{\infty\}$, there is an identification $ \mc{G}_{\underline{\Z/p^n}}\isomto \underline{\mc{G}(\Z/p^n)}$, the group $\mc{G}$ is reconstructible in $(X_\proet,\underline{\Z/p^n})$, and every object of $\GLoc_{\Z/p^n}(X)$ is locally trivial.
\end{prop}

It will again be convenient to consider the category $\cat{Loc}_{\Q_p}(X)$ of \emph{$\Q_p$-local systems} on $X$, defined again using a formula as in \eqref{eq:Q_p-local-systems}. The category $\cat{Loc}_{\Q_p}(X)$ admits a fully faithful embedding into $\cat{Vect}(X_\proet,\underline{\Q_p})$, and so inherits an exact $\Q_p$-linear $\otimes$-structure from this embedding

\begin{rem} As in Remark \ref{rem:no-lattice}, the embedding $\cat{Loc}_{\Q_p}(X)\to \cat{Vect}(X_\proet,\underline{\Q_p})$ is not generally essentially surjective. The same reasoning applies, except now the relevant non-compact group is the de Jong fundamental group $\pi_1^\mr{dJ}(X,\ov{x})$ (see \cite{deJongFG} and \cite[Corollary 4.4.2]{ALY1P2}). Here the situation is more extreme though, as even for $X=\bb{P}^{1,\an}_{\bb{C}_p}$ the de Jong fundamental group is non-compact (cf.\@ \cite[Proposition 7.4]{deJongFG}).
\end{rem}

Given Proposition \ref{prop:G-for-proet-adic}, we write $\cat{Tors}_{\mc{G}(\Z/p^n)}(X_\proet)$ instead of $\cat{Tors}_{\mc{G}_{\underline{\Z/p^n}}}(X_\proet)$, and call objects of this category \emph{$\mc{G}(\Z/p^n)$-torsors}. Moreover, by Proposition \ref{prop:torsors-twists-functors} we know that there is a natural equivalence of categories between $\GLoc_{\Z/p^n}(X)$ and $\cat{Tors}_{\mc{G}(\Z/p^n)}(X_\proet)$.  Explicitly, this functor associates to a $\mc G(\Z/p^n)$-torsor $\mc A$ the $\mc{G}(\Z/p^n)$-local sytem $\omega_\mc{A}$ given by $\omega_\mc{A}(\Lambda)\defeq \mc P\wedge^\mc G\Lambda$.

We again define the category of \emph{$G(\Q_p)$-local systems on $X$}, denoted $G\text{-}\cat{Loc}_{\Q_p}(X)$, to be the category of exact $\Q_p$-linear $\otimes$-functors $\cat{Rep}_{\Q_p}(G)\to\cat{Loc}_{\Q_p}(X)$. Again by Lemma \ref{lem:rational-integral-rep-equiv}, we obtain a natural functor $(-)[\nicefrac{1}{p}]\colon \GLoc_{\Z_p}(X)\to G\text{-}\cat{Loc}_{\Q_p}(X)$. 

Finally, as in \S\ref{ss:G(Z_p)-local-system-scheme}, it is convenient to have a more down-to-earth definition of torsor for a pro-finite group. We use Notation \ref{nota:H}. The definition of an $H$-covering $(Y_n)$ and principal homogeneous spaces $Y$ for $H$ is verbatim to the case of schemes, and the proof of the next result is obtained mutatis mutandis from that of Proposition \ref{prop:G(Zp)-covering-PHS-torsor-comparison-scheme-case} (cf.\@ \cite[Theorem 4.4.1]{ALY1P2}).

\begin{prop}\label{prop:G(Zp)-covering-PHS-torsor-comparison-adic-case}
    The following functors are equivalences functorial in $H$:
 \begin{equation*} 
	\left\{\begin{matrix}H\emph{-coverings}\\ \emph{of }X\end{matrix}\right\}\xrightarrow{(Y_n)\mapsto Y_\infty} \left\{\begin{matrix}\emph{Principal homogenous}\\ \emph{spaces for }H\emph{ on }X\end{matrix}\right\} \xrightarrow{Y\mapsto h_Y}\cat{Tors}_{\underline{H}}(X_\proet).
\end{equation*}
\end{prop}

\subsubsection{Comparison via analytification}\label{ss:comparison-via-analytification} Fix a non-archimedean extension $K$ of $\Q_p$, $S$ to be a finite type $K$-scheme, and $H$ as in Notation \ref{nota:H}. Consider the analytification $S^\mathrm{an}\defeq S\times_{\Spec(K)}\Spa(K)$ of $S$ (see \cite[Proposition 3.8]{HuberGen}). By \cite[Theorem 3.1]{LutkebohmertRE}, we have an equivalence
\begin{equation*}
    \left\{\begin{matrix}H\text{-coverings}\\ \text{of }S\end{matrix}\right\}\isomto \left\{\begin{matrix}H\text{-coverings}\\ \text{of }S^\an\end{matrix}\right\},\qquad (Y_n)\mapsto (Y_n^\an).
\end{equation*}
On the other hand, using the morphism of sites $S^\an_\et\to S_\et$ as in \cite[\S3.8]{HuberEC}, we have a functor 
\begin{equation*}
    (-)^\an\colon \cat{Loc}_{\Z_p}(S)\to \cat{Loc}_{\Z_p}(S^\an),
\end{equation*} 
which is a bi-exact $\Z_p$-linear $\otimes$-equivalence by \cite[Theorem 3.1]{LutkebohmertRE} and the density of classical points. This induces an equivalence of $\mc{G}(\Z_p)$-local systems agreeing with the above equivalence of $\mc{G}(\Z_p)$-coverings under the equivalences of Proposition \ref{prop:G(Zp)-covering-PHS-torsor-comparison-scheme-case} and Proposition \ref{prop:G(Zp)-covering-PHS-torsor-comparison-adic-case}. We denote the quasi-inverse of $(-)^\an$ by $(-)^\mr{alg}$.

\subsubsection{Comparison on affinoids}\label{ss:comparison-on-affinoids} Let $H$ be as in Notation \ref{nota:H}. For a strongly Noetherian Huber pair $(A,A^+)$, we have, by \cite[Example 1.6.6 ii)]{HuberEC}, an equivalence
\begin{equation*}
    \left\{\begin{matrix}H\text{-coverings}\\ \text{of }\Spec(A)\end{matrix}\right\}\to \left\{\begin{matrix}H\text{-coverings}\\ \text{of }\Spa(A,A^+)\end{matrix}\right\}.
\end{equation*}
If $(A,A^+)$ is topologically of finite type over a non-archimedean extension $K$ of $\Q_p$, this gives
\begin{equation*}
    \cat{Loc}_{\Z_p}(\Spec(A))\isomto \cat{Loc}_{\Z_p}(\Spa(A,A^+)),
\end{equation*}
which is a bi-exact $\Z_p$-linear $\otimes$-equivalence by the density of classical points. If $\Spec(A)$ is connected, then the choice of a geometric point $\ov{x}$ gives a futher bi-exact $\Z_p$-linear $\otimes$-equivalence 
\begin{equation*}
    \cat{Loc}_{\Z_p}(\Spec(A))\to\cat{Rep}_{\mathbb{Z}_p}^\mathrm{cont.}(\pi_1^\et(\Spec(A),\ov{x})),\qquad \bb{L}\mapsto \bb{L}_{\ov{x}},
\end{equation*}
(see \cite[Expos\'e VI, Proposition 1.2.5]{SGA5}). We often tacitly identify a $\Z_p$-local system on $\Spa(A,A^+)$ with such a representation.

\subsubsection{$\mc G(\Z_p)$-local systems on a diamond}
It will be convenient to interpret the notion of a $\mc G(\Z/p^n)$-local system in terms of the language of diamonds as in \cite[\S11]{ScholzeECD}. 
Let $Y$ be a diamond in the sense in \cite[Definition 11.1]{ScholzeECD} and $Y_\qproet$ denote the quasi-pro-\'etale site as in \cite[Definition 14.1 (ii)]{ScholzeECD}. 
Similarly as in \S\ref{ss:G(Z_p)-local-system-scheme}--\ref{ss:G(Z_p)-local-systems-on-adic-spaces}, for a topological space $T$, we consider the associated sheaf $\underline{T}_Y$ (as in \cite[Example 11.12]{ScholzeECD}). In particular, we have the ringed site $(Y_\qproet,\underline{\Z/p^n}_Y)$ and the sheaf of groups $\underline{\mc G(\Z/p^n)}_Y$ for $n\in\mathbb N\cup\{\infty\}$. 
Write $\cat{Loc}_{\Z/p^n}(Y)$ for the category $\cat{Vect}(Y_\qproet,\underline{\Z/p^n})$. 

\begin{rem}\label{rem: torsors on the diamond associated to a locally noetherian adic space}
    Let $X$ be a locally Noetherian adic space $X$ over $\Q_p$ and $X^\lozenge$ denote the diamond associated to $X$ (see \cite[Definition 15.5 and Lemma 15.6]{ScholzeECD}). 
    Then the morphism of sites $X^\lozenge_\qproet\to X_\proet$ induces a $\Z_p$-linear bi-exact $\otimes$-equivalence 
    \begin{equation*}
        \cat{Loc}_{\Z/p^n}(X)\isomto \cat{Loc}_{\Z/p^n}(X^\lozenge) 
    \end{equation*}
    and an equivalence of groupoids
    \begin{equation*}
        \cat{Tors}_{\underline{\mc G(\Z/p^n)}}(X)\isomto\cat{Tors}_{\underline{\mc G(\Z/p^n)}}(X^\lozenge). 
    \end{equation*}
    Indeed, when $n\in\mathbb N$, any of the above four categories is identified with the corresponding category for the \'etale site: 
    note first that the sheaf $\underline{\mc G(\Z/p^n)}$ on $X_\proet$ (resp.\ $X^\lozenge_\qproet$) is the pullback of the constant sheaf $\mc G(\Z/p^n)$ on $X_\et$ (resp.\@ $X^\lozenge_\et$); 
    then for the left (resp.\@ right) hand side, use the version of \cite[Corollary 3.17 (i)]{ScholzepadicHT} for sheaves of sets, which is shown by the same way as \cite[Proposition 8.5 (i)]{ScholzeECD}, (resp.\@ use \cite[Proposition 14.8]{ScholzeECD}) to see that the category of \'etale local systems/\'etale torsors is fully-faithfully embedded; 
    then use the version of \cite[Lemma 3.16]{ScholzepadicHT} for sheaves of sets, again shown as in \cite[Proposition 8.5 (ii)]{ScholzeECD}, (resp.\@ use \cite[Proposition 9.7, cf.\@ Definition 10.1]{ScholzeECD}) to see that any pro\'etale (resp.\@ quasi-pro\'etale) $\Z/p^n$-local system or $\underline{\mc G(\Z/p^n)}$-torsor is \'etale locally trivialized. 
    Then the equivalences follow from \cite[Lemma 15.6]{ScholzeECD}. By passage to limit we obtain the case of $n=\infty$, using the analogues of \cite[Proposition 6.8.4 (i)]{BhattScholzeProetale}, which holds as any $\underline{\mathcal{G}(\mathbb{Z}/p^n)}$-torsor on $X_\proet$ (resp.\@ $X^\lozenge_\qproet$) is trivialized by a finite \'etale cover. 
\end{rem}

\begin{prop}\label{prop:G-for-proet-diamonds}
    For $n\in\bb{N}\cup\{\infty\}$, there is an identification $ \mc{G}_{\underline{\Z/p^n}}\isomto \underline{\mc{G}(\Z/p^n)}$, the group $\mc{G}$ is reconstructible in $(Y_\qproet,\underline{\Z/p^n})$, and every object of $\GLoc_{\Z/p^n}(Y)$ is locally trivial.
\end{prop}
\begin{proof}
    The identification and the reconstructibility are proven mutatis mutandis as in Lemma \ref{lem:G-for-proet-scheme} and (the locally Noetherian case of) Proposition \ref{prop:tannakian-formalism-proetale-scheme}. 

    To prove the local triviality, we may assume that $Y$ is represented by a perfectoid space, in which case the assertion is proven in the proof of \cite[Proposition 22.6.1]{ScholzeBerkeley}. 
\end{proof}

\subsubsection{Comparison with \texorpdfstring{$v$}{v}-sheaves}\label{ss:comparison-for-v-sheaves} There is a unique functor 
\begin{equation*}
    (-)^\ad\colon \left\{\begin{matrix}p\text{-adic formal}\\\text{schemes}\end{matrix}\right\}\to \left\{\begin{matrix}\text{Pre-adic spaces}\\ \text{over }\Z_p\end{matrix}\right\},
\end{equation*}
sending open covers to open covers, and with $\Spf(A)^\ad=\Spa^Y(A,A)$ (see \cite[\S3.4]{ScholzeBerkeley}). For a $p$-adic formal scheme $\mf{Y}$, denote by $\mf{Y}^\lozenge$ the $v$-sheaf over $\Spd(\Z_p)$ associated to $\mf{Y}^\ad$ as in \cite[\S18.1]{ScholzeBerkeley}. Set $(\mf{Y}^\lozenge)_\eta\defeq \mf{Y}^\lozenge \times_{\Spd(\Z_p)}\Spd(\Q_p)$. For an affine open $\Spf(A)\subseteq \mf{Y}$, the pair $(A[\nicefrac{1}{p}],\wt{A})$ is a Huber pair over $(\Q_p,\Z_p)$, with $\wt{A}$ the integral closure of $A$ in $A[\nicefrac{1}{p}]$ and $A[\nicefrac{1}{p}]$ given the unique ring topology with the image of $A$ in $A[\nicefrac{1}{p}]$ open. The diamond $(\mf{Y}_\eta)^\lozenge$ associated to $\mf{Y}_\eta\defeq \colim \Spa^Y(A[\nicefrac{1}{p}],\wt{A})$ over $\Q_p$ (see \cite[\S10.1]{ScholzeBerkeley}) agrees with $(\mf{Y}^\lozenge)_\eta$. Denote the common object by $\mf{Y}_\eta^\lozenge$. 

Consider the quasi-pro-\'etale site $\mf{Y}_{\eta,\qproet}^\lozenge$ as in \cite[Definition 14.1]{ScholzeECD}. Consider
\begin{equation*}
    \underline{\Z_p}_{\mf{Y}}=R\lim \underline{\Z/p^n}_{\mf{Y}}=\lim \underline{\Z/p^n}_{\mf{Y}}
\end{equation*}
where $\underline{\Z/p^n}_\mf{Y}$ is the constant sheaf, and the second equality follows from \cite[Lemma 7.18]{ScholzeECD} and \cite[Proposition 3.1.10 and Proposition 3.2.3]{BhattScholzeProetale}. Define $\cat{Loc}_{\Z_p}(\mf{Y}_\eta)$ to be $\cat{Vect}(\mf{Y}_{\eta,\qproet}^\lozenge,\underline{\Z_p}_\mf{Y})$. 

If $\mf{Y}\to \Spf(\mc{O}_K)$ is locally of finite type, then $\mf{Y}_\eta$ is the rigid $K$-space associated to $\mf{Y}$ (see \cite[\S1.9]{HuberEC} or \cite[\S A.5]{FujiwaraKato}). Thus, in this case we have already defined an exact $\Z_p$-linear $\otimes$-category $\cat{Loc}_{\Z_p}(\mf{Y}_\eta)$, and so there is potential for ambiguity. But, with $H$ as in Notation \ref{nota:H}, defining $H$-coverings in the obvious way, one may again show that there is an equivalence of categories between such $H$-coverings and $\cat{Loc}(\mf{Y}_\eta)$. 
Note \cite[Lemma 15.6]{ScholzeECD} gives
\begin{equation*}
    \left\{\begin{matrix}H\text{-coverings}\\ \text{of }\mf{Y}_\eta\end{matrix}\right\}\isomto \left\{\begin{matrix}H\text{-coverings}\\ \text{of }\mf{Y}_\eta^\lozenge\end{matrix}\right\},\qquad (Y_n)\mapsto (Y_n^\lozenge).
\end{equation*}
Using this equivalence and \cite[Proposition 14.3]{ScholzeECD} (for the $\Z/p^n\Z$-local system for each $n\geqslant 1$) one obtains a bi-exact $\Z_p$-linear $\otimes$-equivalence \begin{equation*}
    \cat{Loc}_{\Z_p}(\mf{Y}_\eta)\isomto \cat{Vect}(\mf{Y}_{\eta,\qproet}^\lozenge,\underline{\Z_p}_\mf{Y}). 
\end{equation*} 
Thus, no abmiguity actually occurs if $\mf{Y}\to\Spf(\mc{O}_K)$ is locally of finite type.

If $\mf{Y}$ is quasi-syntomic, then there is a natural functor 
\begin{equation*}
    \cat{Loc}_{\Z_p}(\mf{Y}_\eta)\to \twolim_{\scriptscriptstyle \Spf(R)\in\mf{Y}_\qrsp}\cat{Loc}_{\Z_p}(\Spf(R)_\eta). 
\end{equation*}
This is a bi-exact $\Z_p$-linear $\otimes$-equivalence by the next lemma and \cite[Proposition 9.7]{ScholzeECD}, which we apply by using a limit argument to reduce to the case of finite coefficients.

\begin{lem}\label{lem:quasi-syn-cover-v-sheaf-cover} Let $\{\mf{Y}_i\to\mf{Y}\}$ be a faithfully flat cover of $p$-adic formal schemes. Then, the collection $\{\mf{Y}_{i,\eta}^\lozenge\to\mf{Y}^\lozenge_\eta\}$ is a cover of $v$-sheaves.
\end{lem}
\begin{proof} By \cite[Lemma 12.11]{ScholzeECD} it suffices to show that $\bigsqcup_i |\mf{Y}_{i,\eta}^\lozenge|\to |\mf{X}_\eta^\lozenge|$ is surjective and any quasi-compact open of the target is covered by a quasi-compact open of the source. By \cite[Lemma 15.6]{ScholzeECD} this is equivalent to proving this claim for $\bigsqcup_i |\mf{Y}_{i,\eta}|\to |\mf{X}_\eta|$. This reduces to showing that if $\Spf(B)\to\Spf(A)$ is faithfully flat then $\Spf(B)_\eta\to\Spf(A)_\eta$ is surjective. By \cite[Proposition 2.1.6]{ScholzeWeinstein}, we must show that for an affinoid field $(K,K^+)$ over $(\Q_p,\Z_p)$ and a morphism $\Spf(K^+)\to \Spf(A)$ there exists a surjection $\Spf(L^+)\to \Spf(K^+)$, with $(L,L^+)$ an affinoid field over $(\Q_p,\Z_p)$, so that $\Spf(L^+)\to \Spf(A)$ lifts to $\Spf(B)$. This follows from the discussion in \cite[Example (2), \S2.2.1]{CesnaviciusScholze}, as the argument there does not use that $K^+$ has rank at most $1$.
\end{proof}

\subsection{The \'etale realization functor}\label{ss: T et} Let $\mf{X}$ be a quasi-syntomic flat formal $\Z_p$-scheme, and $X=\mf{X}_\eta^\lozenge$. We discuss the equivalence between $\Z_p$-local systems on $X$ and prismatic Laurent $F$-crystals on $\mf{X}$ given in \cite[Corollary 3.7]{BhattScholzeCrystals}, explicating its bi-exactness.

Note that $\phi\colon \mc{O}_\Prism\to\mc{O}_\Prism$ induces a morphism $\phi\colon \mc{O}_\Prism[\nicefrac{1}{\mc{I}_\Prism}]^\wedge_p\to \mc{O}_\Prism[\nicefrac{1}{\mc{I}_\Prism}]^\wedge_p$.\footnote{For $(A,I)$, the map $\phi_A$ induces an endomorphism of $A[\nicefrac{1}{I}]/p^nA[\nicefrac{1}{I}]$ as $\phi_A(d)=d^p+p\delta(d)$ for any $d$ in $I$, which as $p$ is nilpotent in $A/p^n$, implies that the image of $I$ under $\phi_A\colon A/p^n\to A/p^n[\nicefrac{1}{I}]$ generates the unit ideal.\label{footnote:laurent-frobenius}} As in \cite[Definition 3.2]{BhattScholzeCrystals}, define the category $\cat{Vect}^\varphi(\mf{X}_\Prism,\mc{O}_\Prism[\nicefrac{1}{\mc{I}_\Prism}]^\wedge_p)$ of \emph{prismatic Laurent $F$-crystals on $\mf{X}$} to consist of pairs $(\mc{L},\varphi)$ with $\mc{L}$ an object of $\cat{Vect}(\mf{X}_\Prism,\mc{O}_\Prism[\nicefrac{1}{\mc{I}_\Prism}]^\wedge_p)$ and $ \varphi\colon \phi^\ast\mc{L}\isomto \mc{L}$ an isomorphism in $\cat{Vect}(\mf{X}_\Prism,\mc{O}_\Prism[\nicefrac{1}{\mc{I}_\Prism}]^\wedge_p)$, called the \emph{Frobenius}, and morphisms are morphisms in $\cat{Vect}(\mf{X}_\Prism,\mc{O}_\Prism[\nicefrac{1}{\mc{I}_\Prism}]^\wedge_p)$ commuting with the Frobenii.

We can endow $\cat{Vect}^\varphi(\mf{X}_\Prism,\mc{O}_\Prism[\nicefrac{1}{\mc{I}_\Prism}]^\wedge_p)$ with the structure of an exact $\Z_p$-linear $\otimes$-category, where it inherits its exact structure from $\cat{Vect}(\mf{X}_\Prism,\mc{O}_\Prism[\nicefrac{1}{\mc{I}_\Prism}]^\wedge_p)$. By \cite[Proposition 2.7]{BhattScholzeCrystals}, together with an argument as in the proof of Proposition \ref{prop:vector-bundle-limit-decomp}, with $\cat{Vect}^\varphi(A[\nicefrac{1}{I}]^\wedge_p)$ having the obvious meaning, taking global sections gives a bi-exact $\Z_p$-linear $\otimes$-equivalence
\begin{equation*}
\cat{Vect}^\varphi(\mf{X}_\Prism,\mc{O}_\Prism[\nicefrac{1}{\mc{I}_\Prism}]^\wedge_p) \isomto \twolim_{\scriptscriptstyle (A,I)\in\mf{X}_\Prism}\cat{Vect}^\varphi(A[\nicefrac{1}{I}]^\wedge_p),
\end{equation*}
where the right-hand side is given the term-by-term exact and $\Z_p$-linear $\otimes$-structure.

In \cite[\S3]{BhattScholzeCrystals}, Bhatt and Scholze define an \emph{\'etale realization functor}
\begin{equation*}
        T_{\mf{X},\et}\colon \cat{Vect}^\varphi(\mf{X}_\Prism,\mc{O}_\Prism[\nicefrac{1}{\mc{I}_\Prism}]^\wedge_p)\isomto \cat{Loc}_{\Z_p}(X).
    \end{equation*}
which is an equivalence by \cite[Corollary 3.8]{BhattScholzeCrystals} (cf.\@ \cite{MinWang} and \cite{Wu}). 
When $\mf{X}$ is clear from context, we shall drop $\mf{X}$ from the notation. We now wish to show the following.

\begin{prop}\label{prop:Laurent-crystal-equivalence} The functor $T_{\mf{X},\et}$ is a bi-exact $\Z_p$-linear $\otimes$-equivalence.
\end{prop}

Our proof essentially recounts the construction of Bhatt--Scholze, verifying exactness at each step. Let $R$ be a qrsp ring, and define 
\begin{equation*}
    \Prism_{R,\infty}\defeq (\varinjlim_\phi \Prism_R)^\wedge_{(p,I_R)},\quad I_{R,\infty}=I_R\Prism_{R,\infty}.
\end{equation*}
Then, $(\Prism_{R,\infty},I_{R,\infty})$ is a perfect prism and the $p$-adically complete $R$-algebra 
\begin{equation*}
    \Prism_{R,\infty}/I_{R,\infty}=(\varinjlim_\phi \Prism_R/I_R\varinjlim_\phi \Prism_R)^\wedge_p\efdeq R_\perfd 
\end{equation*}
is perfectoid and initial amongst maps from $R$ to a perfectoid ring (see \cite[Corollary 7.3]{BhattScholzePrisms}). Thus, the maps $\Spf(R_\perfd)^\lozenge\to\Spf(R)^\lozenge$ and $\Spf(R_\perfd)^\lozenge_\eta\to\Spf(R)^\lozenge_\eta$ are isomorphisms: indeed, for a perfectoid $(A,A^+)$, by using \cite[Proposition 2.1.6 (i)]{ScholzeWeinstein} the set  $\Spf(R)^\lozenge((A,A^+))$ is identified with the set $\Hom((R,R),(A,A^+))$ of maps of Huber pairs, which is nothing but the set $\Hom(R,A^+)$ of maps of $p$-adic rings; similarly for $\Spf(R_\perfd)^\lozenge$. 

We now briefly recall the construction of the \'etale realization functor $T_{\mf X,\et}$. When $\mf X=\Spf(S)$ with $S$ a perfectoid $p$-adic ring, the equivalence 
\begin{equation*}
    \cat{Loc}_{\Z_p}(X)\isomto \cat{Vect}^\varphi(\mf X_\Prism,\mc O_\Prism[\nicefrac{1}{\mc I}]^\wedge_p)=\cat{Vect}^\varphi(\Prism_S[\nicefrac{1}{I_S}]_p^\wedge)
\end{equation*}
is defined by passing to the tilt $X^\flat=\Spf(S^\flat)_\eta$ and applying the scalar extension along  $\Z_p\to W(\mc O_{X^\flat,\et})$ (see \cite[Example 3.5]{BhattScholzeCrystals} for the details). 
When $\mf X=\Spf(R)$ with $R$ a qrsp $p$-adic ring, the equivalence $T_{\mf X,\et}$ is defined via the isomorphism $\Spf(R_\mr{perfd})_\eta^\lozenge\to \Spf(R)^\lozenge_\eta$ and the scalar extension functor 
\begin{equation*}
    \cat{Vect}^\varphi(\Prism_R[\nicefrac{1}{I_R}]_p^\wedge)\to\cat{Vect}^\varphi(\Prism_{R,\infty}[\nicefrac{1}{I_{R,\infty}}]_p^\wedge), 
\end{equation*}
which is checked to be an equivalence by passing to mod $p$ and using \cite[Proposition 3.6]{BhattScholzeCrystals}. 
In general, for a quasi-syntomic $p$-adic formal scheme $\mf X$, the equivalence $T_{\mf X,\et}$ is defined via quasi-syntomic descent (cf.\ Lemma \ref{lem:quasi-syn-cover-v-sheaf-cover}).

\begin{lem}[{cf.\@ \cite[Proposition 3.6 and Corollary 3.7]{BhattScholzeCrystals}}]\label{lem:laurent-invariance-perfection-qrsp} 
Let $R$ be a $p$-torsion-free qrsp ring. Then, the base change functor
\begin{equation*}
\cat{Vect}^\varphi(\Prism_R[\nicefrac{1}{I_R}]^\wedge_p)\to \cat{Vect}^\varphi(\Prism_{R,\infty}[\nicefrac{1}{I_{R,\infty}}]^\wedge_p)
\end{equation*}
is a bi-exact $\Z_p$-linear $\otimes$-equivalence functorial in $R$.
\end{lem}
\begin{proof} We follow the notation of \cite[Proposition 3.6 and Corollary 3.7]{BhattScholzeCrystals}. The only thing to be verified is that this equivalence is bi-exact. In turn, by the method of proof in loc.\@ cit.\@ we reduce to the claim that the equivalence $\cat{Loc}_{\mathbb{F}_p}(S)\isomto \cat{Vect}^\varphi(S)$ for any $\F_p$-scheme $S$ from \cite[Proposition 3.4]{BhattScholzeCrystals} is bi-exact. But, this is clear as the map $\mathbb{F}_p\to \mc{O}_S$ is faithfully flat. 
\end{proof}

\begin{proof}[Proof of Proposition \ref{prop:Laurent-crystal-equivalence}] The only thing to verify is bi-exactness. As the proof in \cite[Corollary 3.7]{BhattScholzeCrystals} proceeds by descent to the case when $\mf{X}=\Spf(R)$ with $R$ qrsp and $p$-torsion-free, which preserves exactness (cf.\@  Lemma \ref{lem:quasi-syn-cover-v-sheaf-cover}), we also reduce to this case. Further, using Lemma \ref{lem:laurent-invariance-perfection-qrsp} and the isomorphism $\Spf(R_\perfd)^\lozenge_\eta\to\Spf(R)^\lozenge_\eta$ we may reduce to the case when $R$ is perfectoid.

By \cite[Lemma 3.21]{BMSI}, $(R[\nicefrac{1}{p}],R')$ is a Tate perfectoid algebra, where $R'$ is the integral closure of $R$ in $R[\nicefrac{1}{p}]$. Thus, there is a bi-exact $\Z_p$-linear $\otimes$-equivalence 
\begin{equation*}
    \cat{Loc}_{\Z_p}(\Spa(R[\nicefrac{1}{p}],R'))\isomto \cat{Loc}_{\Z_p}(\Spa(R[\nicefrac{1}{p}],R')^\flat)=\cat{Loc}_{\Z_p}(\Spa(R^\flat[\nicefrac{1}{\varpi^\flat}],(R')^\flat))
\end{equation*}
(see \cite[Theorem 6.3]{ScholzeECD}), where $\varpi$ is a pseudo-uniformizer of $R[\nicefrac{1}{p}]$ as in \cite[Lemma 6.2.2]{ScholzeBerkeley} which may be taken to lie in $R$. On the other hand, $\Prism_R=W(R^\flat)$ and $I_R=(\tilde{\xi})$, with $\tilde{\xi}=p+[\varpi^\flat]\alpha$ with $\alpha$ in $W(R^\flat)$ (see \cite[Lemma 6.2.10]{ScholzeBerkeley}). But, $W(R^\flat)[\nicefrac{1}{\tilde{\xi}}]^\wedge_p=W(R^\flat[\nicefrac{1}{\varpi^\flat}])$ as both are strict $p$-rings (in the sense of \cite[Definition 3.2.1]{KeLiRpHF}) with the same residue ring. Thus, we will be done if the $\Z_p$-linear $\otimes$-equivalence between the category of $\varphi$-modules for $W(R^\flat[\nicefrac{1}{\varpi^\flat}])$ and $\Z_p$-local systems on $\Spec(R^\flat[\nicefrac{1}{\varpi^\flat}])$ is bi-exact, which was already discussed in the proof of Lemma \ref{lem:laurent-invariance-perfection-qrsp}.
\end{proof}

\begin{example}\label{ex:small-etale-realization}
    Let $R$ be a base $\mc{O}_K$-algebra. As $\wt{\theta}\colon \Ainf(\check{R})\to \check{R}$ is $\Gamma_R$-equivariant, we see that $\Gamma_R$ acts on $(\Ainf(\check{R}),(\tilde{\xi}))$ as an object of $R_\Prism$. In this way, from a prismatic Laurent $F$-crystal $\mc{L}$ on $R$, we obtain a finite free $\Z_p$-module $\Lambda=\mc{L}(\Ainf(\check{R}),(\tilde{\xi}))^{\varphi=1}$ with a continuous action of $\Gamma_R$, which is the $\Gamma_R$-representation associated to $T_\et(\mathcal{L})$.
\end{example}

\subsection{Crystalline local systems and analytic prismatic \texorpdfstring{$F$}{F}-crystals} We discuss the equivalence from \cite{GuoReinecke} and \cite{DLMS} and its Tannakian consequences. Unless stated otherwise, we assume that $\mf{X}\to \Spf(\mc{O}_K)$ is smooth and write $X=\mf{X}_\eta$.

\subsubsection{Filtered \texorpdfstring{$F$-isocrystals}{F-isocrystals}}\label{ss:filtered-f-isocrystals} In this subsection we record our notation and conventions concerning the crystalline site, $F$-(iso)crystals, and filtered $F$-isocrystals. The reader is encouraged to skip this on first reading, referring back only as is necessary. 

\medskip

\paragraph{PD thickenings of formal schemes} Let $\mf{Z}\to\Spf(W)$ be an adic morphism. By a \emph{PD thickening} of formal $\mf{Z}$-schemes over $W$, we mean a pair $(i\colon \mf{U}\hookrightarrow \mf{T},\gamma)$ where $\mf{U}\to\mf{Z}$ is an adic morphism of formal $W$-schemes, $i\colon \mf{U}\to\mf{T}$ is a closed immersion of adic formal $W$-schemes, and if $\mc{I}\defeq\ker(\mc{O}_\mf{T}\to i_\ast\mc{O}_\mf{U})$ then $\gamma=(\gamma_n)\colon \mc{I}\to \mc{O}_\mf{T}$ is a sequence of morphisms of sheaves so that for every open $\mf{V}\subseteq \mf{T}$ the maps $\gamma_n\colon \mc{I}(\mf{V})\to \mc{O}_\mf{T}(\mf{V})$ form a PD structure compatible with the usual one on $(p)\subseteq W$. We will often drop $i$ and $\gamma$ from the notation when they are clear from context. 

A morphism $(f,g)
\colon (i'\colon \mf{U}'\hookrightarrow\mf{T}',\gamma')\to (i\colon \mf{U}\hookrightarrow \mf{T},\gamma)$ is a morphism $f\colon \mf{U}'\to \mf{U}$ of formal $\mf{Z}$-schemes, and $g\colon \mf{T}'\to \mf{T}$ a morphism of formal $W$-schemes, such that the diagram
\begin{equation}\label{eq:PD-diagram}
    \xymatrixcolsep{1.3pc}\xymatrixrowsep{1.3pc}\xymatrix{\mf{U}'\ar[r]^{i'}\ar[d]_f & \mf{T}'\ar[d]^{g}\\ \mf{U}\ar[r]_i & \mf{T},}
\end{equation}
commutes, and the induced map $g\colon \mc{I}\to g_\ast\mc{I}'$ satisfies $g\circ \gamma=g_\ast(\gamma')\circ g$. A morphism $(f,g)$ is \emph{Cartesian} if \eqref{eq:PD-diagram} is Cartesian, and a collection $\{(f_i,g_i)\colon (i_i\colon \mf{U}_i\hookrightarrow \mf{T}_i,\gamma_i)\to (i \colon\mf{U}\hookrightarrow \mf{T},\gamma)\}$ is a \emph{flat cover} if each $(f_i,g_i)$ is Cartesian and $\{g_i\colon \mf{T}_i\to\mf{T}\}$ is a cover in $\Spf(W)_\fl$. 

If $\mf{U}=\Spf(B)$ and $\mf{T}=\Spf(A)$, we will often write $(i\colon \mf{U}\hookrightarrow\mf{T},\gamma)$ as $(i\colon A\twoheadrightarrow B,\gamma)$, and write morphisms of such affine objects in the direction dictated by maps of rings. If $i\colon A\to A$ is the identity map, we further shorten $(i\colon A\to A,\gamma)$ to $A$. Moreover, for a PD thickening of affine formal $\mf{Z}$-schemes $(i\colon A\twoheadrightarrow B,\gamma)$, we have a filtration $\Fil^\bullet_\mr{PD}(A)$, or just $\Fil^\bullet_\mr{PD}$ when $A$ is clear from context, given by $\Fil^r_{\mr{PD}}(A)\defeq \ker(i)^{[r]}$. Here for an element $a$ of $A$ we sometimes abbreviate $\gamma_r(a)$ to $a^{[r]}$, and for an ideal $I\subseteq A$ by $I^{[r]}$ the ideal generated by $\gamma_{e_1}(i_1)\cdots\gamma_{e_k}(i_k)$ with $\sum e_j\geqslant r$ and $i_j$ in $I$ for all $j$.

\medskip

\paragraph{The big crystalline site of a formal scheme} The \emph{(big) crystalline site} $(\mf{Z}/W)_\crys$ is the site consisting of PD-thickenings of formal $\mf{Z}$-schemes over $W$, endowed with the topology whose covers are flat covers. If $p$ is locally nilpotent on $\mf{Z}$, this coincides with the crystalline site as in \stacks{07I5}, and in general the proof that $(\mf{Z}/W)_\crys$ is a site is proven in much the same way. The site $(\mf{Z}/W)_\crys$ naturally comes equipped with the sheaves $\mc{J}_{(\mf{Z}/W)_\crys}\subseteq \mc{O}_{(\mf{Z}/W)_\crys}$ where
\begin{equation*}
    \mc{O}_{(\mf{Z}/W)_\crys}(i\colon \mf{U}\hookrightarrow\mf{T},\gamma)\defeq \mc{O}_\mf{T}(\mf{T}),\qquad \mc{J}_{(\mf{Z}/W)_\crys}(i\colon \mf{U}\hookrightarrow\mf{T},\gamma)\defeq \ker\left(\mc{O}_\mf{T}\to i_\ast\mc{O}_\mf{U}\right), 
\end{equation*}
so $\mc{J}_{(\mf{Z}/W)_\crys}(i\colon A\twoheadrightarrow B,\gamma)=\Fil^1_\mr{PD}(A)$. We often shorten this notation to $\mc{O}_\crys$ and $\mc{J}_\crys$. When $\mf{Z}=\Spf(R)$ we shall shorten $(\mf{Z}/W)_\crys$ to $(R/W)_\crys$, and similarly for other notation below.

The following example will appear frequently in the sequel (compare with Example \ref{example:initial-prism-Acrys}).

\begin{example}\label{example:acrys} Let $S$ be a $p$-adically complete $W$-algebra with $S/p$ semi-perfect. Then, if $S^\flat$ denotes the perfection of $S/p$, there exists a universal $p$-adic pro-thickening $\theta\colon W(S^\flat)\to S$ over $W$ (see \cite[\S1.2]{FontainePeriodes}). Let $K=\ker(\theta)$, and let $\Acrys(S)$ denote the $p$-adic PD-envelope of the pair $(W(S^\flat),K)$. Then, $\theta$ extends to a $p$-adically continuous surjection $\theta\colon \Acrys(S)\to S$. The pair $(\theta\colon\Acrys(S)\to S,\gamma)$ is a final object of $(S/W)_\crys$
(see \cite[Th\'eor\`eme 2.2.1]{FontainePeriodes}).
\end{example}

There is a natural functor $u_\crys\colon (\mf{Z}/W)_\crys\to \mf{Z}_\mr{pr}^\mr{adic}$, where $\mf{Z}_\mr{pr}^\mr{adic}$ is the category of all adic formal $\mf{Z}$-schemes equipped with the pr-topology as in \cite[\S7.1]{LauDivided}, by $u_\crys(i\colon\mf{U}\hookrightarrow\mf{T},\gamma)\defeq \mf{U}$. This functor is cocontinuous (as one can lift pr-covers along surjective closed embeddings), and so induces a morphism of topoi
\begin{equation*}
    (u_{\crys,\ast},u_\crys^{-1})\colon \cat{Shv}((\mf{Z}/W)_\crys)\to \cat{Shv}(\mf{Z}_\mr{pr}^\mr{adic}).
\end{equation*}
If $\mf{Z}$ is quasi-syntomic, then qrsp objects are a basis for $\mf{Z}^\mr{adic}_\mr{pr}$ (cf.\@ the proof of \cite[Lemma 4.28]{BMS-THH}). So, one may prove the following much the same way as Proposition \ref{prop:cover-of-final-object-qrsp}.

\begin{prop}\label{prop:qsyn-crys-cover} Suppose that $\mf{Z}$ is quasi-syntomic. Then, $\{(\theta_i\colon \Acrys(S_i)\to S_i,\gamma_i)\}$ where $\{S_i\}$ runs over $\mf{Z}_\qrsp$ forms a basis of $\cat{Shv}((\mf{Z}/W)_\crys)$.
\end{prop}

The topos $\cat{Shv}((\mf{Z}/W)_\crys)$ is functorial. Namely, suppose that $f\colon \mf{Z}'\to\mf{Z}$ is an adic morphism of formal schemes lying over a morphism $W\to W$. There is then a morphism of topoi
\begin{equation*}
    (f_\ast,f^{-1})\colon \cat{Shv}((\mf{Z}'/W)_\crys)\to \cat{Shv}((\mf{Z}/W)_\crys),
\end{equation*}
and $f^{-1}$ has a very concrete description, with $f^{-1}\mc{O}_{(\mf{Z}/W)_\crys}=\mc{O}_{(\mf{Z}'/W)_\crys}$ (see \cite[1.1.10]{BBMDieuII}).

\medskip

\paragraph{The category of (iso)crystals} Let $\mf{Z}\to\Spf(W)$ be an adic morphism. Define the category of \emph{finitely presented (resp.\@ locally free) crystals} on $\mf{Z}$, denoted $\cat{Crys}(\mf{Z})$ (resp.\@ $\cat{Vect}(\mf{Z}_\crys)$) to be the category of finitely presented (resp.\@ locally free of finite rank) (in the sense of \stacks{03DL}) $\mc{O}_\crys$-modules.\footnote{The forgetful functor $(\mf{Z}/W(k))_\crys\to (\mf{Z}/\Z_p)_\crys$ induces an equivalence of ringed topoi (see \cite[1.1.13]{BBMDieuII}). It is for this reason that we can be slightly imprecise with our notation for crystals.} A finitely presented (resp.\@ locally free) crystal $\mc{E}$ on $\mf{Z}$ is equivalent to the following data: for every $(i\colon \mf{U}\hookrightarrow\mf{T},\gamma)$ a finitely presented (locally free of finite rank) $\mc{O}_\mf{T}$-module $\mc{E}_{(i\colon \mf{U}\hookrightarrow\mf{T},\gamma)}$ and for every morphism $(f,g)\colon (i'\colon \mf{U}'\hookrightarrow \mf{T}',\gamma')\to (i\colon \mf{U}\hookrightarrow\mf{T},\gamma)$ a morphism $g^{-1}\mc{E}_{(i\colon \mf{U}\hookrightarrow\mf{T},\gamma)}\to \mc{E}_{(i'\colon \mf{U}'\hookrightarrow \mf{T}',\gamma')}$ such that the induced morphism $g^\ast\mc{E}_{(i\colon \mf{U}\hookrightarrow\mf{T},\gamma)}\to \mc{E}_{(i'\colon \mf{U}'\hookrightarrow \mf{T}',\gamma')}$ is an isomorphism (see \stacks{07IT}). The association is as in \stacks{07IN}.

For each $n\in\bb{N}\cup\{\infty\}$ set $\mf{Z}_n\defeq (|\mf{Z}|,\mc{O}_\mf{Z}/p^{n+1})$ (where by definition $\mf{Z}_\infty=\mf{Z}$). For $m\leqslant n$ in $\bb{N}\cup\{\infty\}$, denote by $\iota_{m,n}\colon \mf{Z}_m\to \mf{Z}_{n}$ the natural closed immersion. Then, one may check that $(\iota_{m,n})_\ast \mc{O}_{(\mf{Z}_m/W)_\crys}=\mc{O}_{(\mf{Z}_n/W)_\crys}$, and by \cite[Lemma 2.1.4]{deJongCrystalline} there are pairs of quasi-inverse equivalences
\begin{equation}\label{eq:level-independence-crystals}
    \begin{aligned}((\iota_{m,n})_\ast,\iota_{m,n}^\ast)\colon \cat{Crys}(\mf{Z}_m) &\isomto \cat{Crys}(\mf{Z}_n),\\ ((\iota_{m,n})_\ast,\iota_{m,n}^\ast)\colon \cat{Vect}((\mf{Z}_m)_\crys)&\isomto \cat{Vect}((\mf{Z}_n)_\crys).\end{aligned}
\end{equation}
For a crystal $\mc{E}$ on $\mf{Z}$, set $\mc{E}_n\defeq \iota_{n,\infty}^\ast\mc{E}$. For an object $(i\colon \mf{U} \hookrightarrow \mf{T},\gamma)$ of $(\mf{Z}/W)_\crys$, we have
\begin{equation*}
    \mc{E}(i\colon \mf{U}\hookrightarrow\mf{T},\gamma)=\mc{E}_n(i_n\colon \mf{U}_n\hookrightarrow\mf{T},\gamma)
\end{equation*}
as $\mc{O}_\mf{T}(\mf{T})$-modules. Here, $i_n$ is the composition of the canonical closed embedding $\mf{U}_n\hookrightarrow \mf{U}$ with $i$ which, as $\gamma$ is compatible with the PD structure on $W$, admits a unique extension (also denoted $\gamma$) to $\ker(\mc{O}_\mf{T}\to (i_n)_\ast\mc{O}_{\mf{U}_n})=(p^n,\ker(\mc{O}_\mf{T}\to i_\ast\mc{O}_{\mf{U}}))$.

\begin{rem} Note that while $\mc{J}_{(\mf{Z}/W)_\crys}$ is locally quasi-coherent (see \stacks{07IS}), it does not satisfy the crystal condition. Nor is it true that $(\iota_{m,n})_\ast\mc{J}_{(\mf{Z}_m/W)_\crys}=\mc{J}_{(\mf{Z}_n/W)_\crys}$. 
\end{rem}

Finally, the category $\cat{Isoc}(\mf{Z})$ of \emph{isocrystals} on $\mf{Z}$ has the same objects as $\cat{Crys}(\mf{Z})$, denoted by $\mc{E}$ or the formal symbol $\mc{E}[\nicefrac{1}{p}]$ when clarity is needed, but with the following morphisms
\begin{equation*}
    \Hom(\mc{E}[\nicefrac{1}{p}],\mc{E}'[\nicefrac{1}{p}])\defeq\Gamma((\mf{Z}/W)_\crys,\mc{H}om(\mc{E},\mc{E}')\otimes_{\underline{\Z_p}}\underline{\Q_p}).
\end{equation*}
For any $m\leqslant n$ in $\bb{N}\cup\{\infty\}$ we again have an equivalence
\begin{equation*}
    ((\iota_{m,n})_\ast,\iota_{m,n}^\ast)\colon \cat{Isoc}(\mf{Z}_m)\isomto \cat{Isoc}(\mf{Z}_n).
\end{equation*}
We again denote by $\mc{E}_n$ (or $\mc{E}_n[\nicefrac{1}{p}]$), the pullback of $\mc{E}$ (or $\mc{E}[\nicefrac{1}{p}]$) to $\mf{Z}_n$.

The categories $\cat{Crys}(\mf{Z})$, $\cat{Vect}(\mf{Z}_\crys)$, and $\cat{Isoc}(\mf{Z})$ carry natural exact $\Z_p$-linear $\otimes$-structures, and for $m\leqslant n$ in $\bb{N}\cup\{\infty\}$ the equivalences $(\iota_{m,n})_\ast$ and $\iota_{m,n}^\ast$ are bi-exact $\Z_p$-linear $\otimes$-equivalences.

\medskip

\paragraph{$F$-(iso)crystals}\label{page:F-crystal-convention} Let $\mf{Z}\to\Spf(W)$ be an adic morphism. The absolute Frobenius morphism $F_{\mf{Z}_0}\colon \mf{Z}_0\to \mf{Z}_0$ lies over the Frobenius $\phi\colon W\to W$, and so induces a morphism of ringed topoi 
\begin{equation*}
    (\phi_\ast,\phi^\ast)\colon (\cat{Shv}((\mf{Z}_0/W)_\crys),\mc{O}_{(\mf{Z}_0/W)_\crys})\to (\cat{Shv}((\mf{Z}_0/W)_\crys),\mc{O}_{(\mf{Z}_0/W)_\crys}).
\end{equation*}
The category $\cat{Crys}^\varphi(\mf{Z})$ (resp.\@ $\cat{Isoc}^\varphi(\mf{Z})$) of \emph{$F$-(iso)crystals} on $\mf{Z}$ has as objects pairs $(\mc{E},\varphi_\mc{E})$ (resp.\@ $(\mc{E}[\nicefrac{1}{p}],\varphi_{\mc{E}[\nicefrac{1}{p}]})$), where $\mc{E}$ (resp.\@ $\mc{E}[\nicefrac{1}{p}]$) is an (iso)crystal on $\mf{Z}$ and $\varphi_\mc{E}$ (resp.\@ $\varphi_{\mc{E}[\nicefrac{1}{p}]}$) is a \emph{Frobenius isomorphism} $(\phi^\ast\mc{E}_0)[\nicefrac{1}{p}]\isomto \mc{E}_0[\nicefrac{1}{p}]$ of isocrystals, with morphisms those morphisms of (iso)crystals commuting with the Frobenii. Denote by $\cat{Vect}^\varphi(\mf{Z}_\crys)$ the full subcategory of $\cat{Crys}^\varphi(\mf{Z})$ of $F$-crystal whose underlying crystal is a vector bundle. Each of the categories $\cat{Crys}^\varphi(\mf{Z})$, $\cat{Vect}^\varphi(\mf{Z}_\crys)$, and $\cat{Isoc}^\varphi(\mf{Z})$ carry natural exact $\Z_p$-linear $\otimes$-structures. For every $m\leqslant n$ in $\bb{N}\cup\{\infty\}$ the morphism $(\iota_{m,n})_\ast$ induces quasi-inverse pairs $((\iota_{m,n})_\ast,\iota_{m,n}^\ast)$ 
\begin{equation*}
    \cat{Crys}^\varphi(\mf{Z}_m)\isomto \cat{Crys}^\varphi(\mf{Z}_n),\,\,  \cat{Vect}^\varphi((\mf{Z}_m)_\crys)\isomto \cat{Vect}^\varphi((\mf{Z}_n)_\crys),\,\,  \cat{Isoc}^\varphi(\mf{Z}_m)\isomto \cat{Isoc}^\varphi(\mf{Z}_n).
\end{equation*}
These are each bi-exact $\Z_p$-linear $\otimes$-equivalences.

\begin{rem}\label{rem:pullback-crystal-compat}Suppose that $(i\colon \mf{U}\hookrightarrow\mf{T},\gamma)$ is an object of $(\mf{Z}/W)_\crys$, and $\phi_\mf{T}\colon\mf{T}\to\mf{T}$ is a Frobenius lift compatible with the Frobenius map $\phi$ on $W$. Fix a crystal $\mc{E}$ on $\mf{Z}$. From the morphism $(F_{\mf{U}_0},\phi_\mf{T})\colon (i_0\colon \mf{U}'_0\hookrightarrow\mf{T}',\gamma)\to (i_0\colon \mf{U}_0 \hookrightarrow \mf{T},\gamma)$ in $(\mf{Z}_0/W)_\crys$, and the crystal property, there is an identification
\begin{equation*}
    \phi^\ast\mc{E}_0(i_0\colon \mf{U}_0 \hookrightarrow \mf{T},\gamma)=\mc{E}(i_0\colon \mf{U}_0' \hookrightarrow \mf{T}',\gamma)\simeq \phi_\mf{T}^\ast\mc{E}(i_0\colon\mf{U}_0\hookrightarrow\mf{T},\gamma)=\phi_\mf{T}^\ast\mc{E}(i\colon\mf{U}\hookrightarrow\mf{T},\gamma).
\end{equation*}
By $\mf{T}'$ (resp.\@ $\mf{U}'$) we denote $\mf{T}$ (resp.\@ $\mf{U}$) but with $W$-structure map (resp.\@ $\mf{Z}$-structure map) twisted by $\phi$ (resp.\@ $F_{\mf{Z}_0}$), so the first equality holds by definition.
\end{rem}

\medskip

\paragraph{Base formal schemes and modules with connection} Let us now assume that $\mf{Z}\to\Spf(W)$ is a base formal $W$-scheme. Define $\cat{MIC}^\mathrm{tqn}(\mf{Z})$ to be the category of pairs $(\mc{V},\nabla)$, where $\mc{V}$ is a coherent $\mc{O}_\mf{Z}$-module, and $\nabla\colon \mc{V}\to \mc{V}\otimes_{\mc{O}_\mf{Z}}\Omega^1_{\mf{Z}/W}$ is an integrable topologically quasi-nilpotent connection (see \cite[Remark 2.2.4]{deJongCrystalline}). Let $\cat{Vect}^\nabla(\mf{Z})$ denote the full subcategory of $\cat{MIC}^\mr{tqn}(\mf{Z})$ of those pairs $(\mc{V},\nabla)$ where $\mc{V}$ is a vector bundle. By \cite[Corollary 2.2.3]{deJongCrystalline}, there are equivalences
\begin{equation*}
    \cat{Crys}(\mf{Z})\isomto \cat{MIC}^\mr{tqn}(\mf{Z}),\qquad \cat{Vect}(\mf{Z}_\crys)\isomto\cat{Vect}^\nabla(\mf{Z}),\qquad \mc{E}\mapsto (\mc{E}_\mf{Z},\nabla_\mc{E}).
\end{equation*}
Here, for an $\mc{O}_\crys$-module object $\mc{F}$, we denote by $\mc{F}_\mf{Z}$ the $\mc{O}_\mf{Z}$-module given by associating to a Zariski open $\mf{U}\subseteq \mf{Z}$ the value $\mc{F}(\id\colon \mf{U}\hookrightarrow \mf{U},\gamma)$. These functors are bi-exact $\Z_p$-linear $\otimes$-equivalences, functorial in $\mf{Z}$.

\medskip

\paragraph{Filtered $F$-isocrystals} For a rigid $K$-space $Y$, denote by $\cat{Vect}^\nabla(Y)$ the category of pairs $(V,\nabla_V)$ where $V$ is a vector bundle on $Y$, and $\nabla_V\colon V\to V\otimes_{\mc{O}_Y}\Omega^1_{Y/K}$ is an integrable connection. Denote by $\cat{VectF}^\nabla(Y)$ the category of triples $(V,\nabla_V,\Fil^\bullet_V)$ where $(V,\nabla_V)$ is an object of $\cat{Vect}^\nabla(Y)$ and $\Fil^\bullet_V$ is a locally split filtration satisfying Griffiths transversality: for all $i\geqslant 0$ the containment $\nabla_V(\Fil^i_V)\subseteq \Fil^{i-1}_V\otimes_{\mc{O}_Y}\Omega^1_{Y/K}$ holds. The category $\cat{Vect}^\nabla(Y)$ has an obvious exact $\Z_p$-linear $\otimes$-structure, and we endow $\cat{VectF}^\nabla(Y)$ with an exact $\Z_p$-linear $\otimes$-structure where
\begin{equation*}
    \Fil^k_{V_1\otimes V_2}=\sum_{i+j=k}\Fil^i_{V_1}\otimes \Fil^j_{V_2},
\end{equation*}
and an exact structure where a sequence is exact if for all $i$ the sequence of vector bundles on $Y$ given by the $i^\text{th}$-graded pieces is exact.

Suppose now that $\mf{Z}\to \Spf(\mc{O}_K)$ is smooth with rigid generic fiber $Z$. In  \cite[Remark 2.8.1 and Theorem 2.15]{Ogus-F-Isocrystals-II}, Ogus constructs an exact $\Z_p$-linear $\otimes$-functor
\begin{equation*}
    \cat{Isoc}^\varphi(\mf{Z}_k)\to \cat{Vect}^\nabla(Z), \qquad (\mc{E},\varphi_{\mc{E}})\mapsto (E,\nabla_E).
\end{equation*}
This functor possesses the following property. For any open $\mf{V}\subseteq \mf{Z}$, and any smooth model $\mf{W}$ of $\mf{V}$ over $W$, one has $(E,\nabla_E)|_{\mf{V}_\eta}$ is isomorphic to $(\mc{F}_{\mf{W}},\nabla_{\mc{F}})\otimes K$, where $\mc{F}=(\iota_{0,\infty})_\ast(\mc{E}|_{\mf{W}_0})$, which is well-defined (i.e.\@, doesn't depend on $\mc{E}$ within its isogeny class).

A \emph{filtered $F$-isocrystal} on $\mf{Z}$ is a triple $(\mc{E},\varphi_\mc{E},\mathrm{Fil}^\bullet_E)$ where $(\mc{E},\varphi_{\mc{E}})$ is an object of $\Isoc^\varphi(\mf{Z}_k)$ and $\Fil^\bullet_E$ is a locally split filtration on $E$ satisfying Griffiths transversality. With the obvious notion of morphism, we denote by $\cat{IsocF}^\varphi(\mf{Z})$ the category of filtered $F$-isocrystals on $\mf{Z}$, which is seen to be identified with the fiber product $\cat{Isoc}^\varphi(\mf{Z}_k)\times_{\cat{Vect}^\nabla(Z)}\cat{VectF}^\nabla(Z)$. We endow $\cat{IsocF}^\varphi(\mf{Z})$ with the exact $\Z_p$-linear $\otimes$-structure inherited from this decomposition and those structures on $\cat{Isoc}^\varphi(\mf{Z}_k)$ and $\cat{VectF}^\nabla(Z)$.

\subsubsection{de Rham local systems}\label{ss:de-Rham-local-systems} Let $Y$ be a smooth rigid $K$-space. Recall the following rings for an affinoid perfectoid space $S=\Spa(R,R^+)\sim \varprojlim \Spa(R_i,R_i^+)$ over $C$ in $Y^\aff_\proet$:
\begin{itemize}
    \item $\Ainf(S)\defeq \Ainf(R^+)$,
    \item $\Bdr^+(S)\defeq \Bdr^+(R^+)\defeq \Ainf(S)[\nicefrac{1}{p}]^\wedge_{\xi_0}$,
    \item $\Bdr(S)\defeq \Bdr(R^+)\defeq \Bdr^+(R^+)[\nicefrac{1}{\xi_0}]$,
    \item $\mc{O}\Bdr^+(S)\defeq \mc{O}\Bdr^+(R^+)\defeq \varinjlim ((R_i^+\wh{\otimes}_{W}\Ainf(R^+))[\nicefrac{1}{p}]^\wedge_{\ker(\theta)})$, given the $\ker(\theta)$-adic filtration, where $\theta\colon (R_i^+\wh{\otimes}_{W(k)}\Ainf(S))[\nicefrac{1}{p}]\to R$ is the base extension of $\theta\colon \Ainf(S)\to R^+$,
    \item $\mc{O}\Bdr(S)\defeq \mc{O}\Bdr(R^+)\defeq \mc{O}\Bdr^+(R^+)[\nicefrac{1}{t}]$ with filtration
    \begin{equation*}
        \Fil^i\mc{O}\Bdr(S)=\sum_{j\in\Z}t^{-j}\Fil^{i+j}\mc{O}\Bdr^+(S)
    \end{equation*}
\end{itemize}
which extend to sheaves on $Y^\aff_\proet$. By \cite[Corollary 6.13]{ScholzepadicHT}, $\mc{O}\Bdr^+$ carries a $\Bdr^+$-horizontal integrable connection satisfying Griffiths transversality, extending to $\mc{O}\Bdr$ by base change. 

An object $\bb{L}$ of $\cat{Loc}_{\Z_p}(Y)$ is called \emph{de Rham} if there exists an object $(V,\nabla_V,\Fil^\bullet_V)$ of $\cat{VectF}^\nabla(Y)$, such that there exists an isomorphism of sheaves of modules over $\Bdr^+$:
\begin{equation*}
    c_\mr{Sch}\colon \bb{L}\otimes_{\underline{\Z}_p}\Bdr^+\isomto \Fil^0\left(V\otimes_{\mc{O}_Y}\mc{O}\Bdr\right)^{\nabla=0},
\end{equation*}
where $\nabla\defeq \nabla_V\otimes \nabla_{\mc{O}_{\Bdr}}$ and we endow $V\otimes_{\mc{O}_Y}\mc{O}\Bdr$ with the tensor product filtration. By \cite[Theorem 7.6]{ScholzepadicHT}, the object $(V,\nabla_V,\Fil^\bullet_V)$ is functorially associated to $\bb{L}$, and we denote it by $D_\dR(\bb{L})$ (which we sometimes conflate with the underlying vector bundle). The category of \emph{de Rham $\Q_p$ local systems} is the essential image of the natural functor $\cat{Loc}_{\Z_p}^\mr{dR}(X)\to \cat{Loc}_{\Q_p}(X)$ and we set $D_\dR(\bb{L}[\nicefrac{1}{p}])\defeq D_\dR(\bb{L})$ (which is independent of the choice of $\bb{L}$).

Denote by $\cat{Loc}^\dR_{\Z_p}(Y)$ the full subcategory of de Rham objects of $\cat{Loc}_{\Z_p}(Y)$, which is seen to be stable under tensor products and duals. The functor $D_\dR\colon \cat{Loc}^\dR_{\Z_p}(Y)\to \cat{VectF}^\nabla(Y)$ is an exact $\Z_p$-linear $\otimes$-functor. If $Y=\Spa(K)$, these notations agree with those of Fontaine.

Fix an object $\omega$ of $\mc{G}\text{-}\cat{Loc}_{\Z_p}^\dR(Y)$ and $y\colon \Spa(K')\to Y$, for $K'$ a finite extension of $K$ in $\ov{K}$. Let $\cat{VectF}(K')$ denote the category of finite-dimensional filtered $K'$-vector spaces. Then, we have an exact $\Z_p$-linear $\otimes$-functor 
\begin{equation*}
   (\omega_{\dR,y},\Fil^\bullet_\dR)\colon \cat{Rep}_{\Z_p}(\mc{G})\to \cat{VectF}(K'), \quad \Lambda\mapsto (D_\dR(\omega(\Lambda)),\mathrm{Fil}^\bullet_{D_\dR(\omega(\Lambda))})_y.
\end{equation*}
Fix a conjugacy class $\bm{\mu}$ of cocharacters of $G_{\ov{K}}$. We say $\omega$ has \emph{cocharacter $\bm{\mu}$} if for all such $y$ and all (equiv.\@ for one) $\mu$ in $\bm{\mu}$ the following condition holds. There is an isomorphism $(\omega_{\dR,y})_{\ov{K}}\simeq \omega_\mr{triv}$ such that $(\Fil^\bullet_\dR)_{\ov{K}}$ is carried to $\Fil^\bullet_\mu$ where for a representation $\rho\colon \mc{G}\to \GL(\Lambda)$ one has $\Fil^r_\mu= \bigoplus_{i\geqslant r}\Lambda_{\ov{K}}[i]$, where $\Lambda_{\ov{K}}[i]$ is the $i$-weight space for the cocharacter $\rho\circ\mu$.  The subcategory of $\GLoc_{\Z_p}^{\dR}(Y)$ of those $\omega$ which have cocharacter $\bm{\mu}$ is denoted by $\GLoc^{\dR}_{\Z_p,\bm{\mu}} (Y)$.

Finally, for a smooth $K$-scheme $Y$, denote by $\cat{Loc}^\dR_{\Z_p}(Y)$ (resp.\@ $\GLoc^{\dR}_{\Z_p,\bm{\mu}}(Y)$) the full subcategory of $\cat{Loc}_{\Z_p}(Y)$ (resp.\@ $\GLoc_{\Z_p}(Y)$) consisting of those $\bb{L}$ (resp.\@ $\omega$) such that $\bb{L}^\an$ (resp.\@ $\omega^\an$) belongs to $\cat{Loc}_{\Z_p}^\dR(Y^\an)$ (resp.\@ $\GLoc_{\Z_p,\bm{\mu}}^{\dR}(Y^\an)$).

\subsubsection{Crystalline local systems} As in \cite[\S2A]{TanTong}, recall the following rings for an affinoid perfectoid $C$-space $S=\Spa(R,R^+)\sim \varprojlim \Spa(R_i,R_i^+)$ in $X_\proet$:
\begin{itemize}
    \item $\Acrys(S)\defeq \Acrys(R^+)$, filtered by $\Fil^\bullet_{\mr{PD}}$,
    \item $\Bcrys^+(S)\defeq \Bcrys^+(R^+)\defeq \Acrys(S)[\nicefrac{1}{p}]$ filtered by $\Fil^\bullet_\mr{PD}[\nicefrac{1}{p}]$,
    \item $\Bcrys(S)\defeq \Bcrys(R^+)\defeq \Bcrys^+(S)[\nicefrac{1}{t}]$, with filtration given by 
    \begin{equation*}
        \Fil^i\Bcrys(S)=\sum_{j\in\Z}t^{-j}\Fil^{i+j}\Bcrys^+(R,R^+).
    \end{equation*}
\end{itemize}

These rings, and their filtrations, extend to sheaves on $X_\proet$. The Frobenius on $\Acrys$ extends uniquely to $\Bcrys^+$ as $p$ is Frobenius invariant, and further extends uniquely to $\Bcrys$ with $\phi(\nicefrac{1}{t})=\nicefrac{1}{pt}$ (see \cite[\S2C]{TanTong}). In particular, for any object $\bb{L}[\nicefrac{1}{p}]$ of $\cat{Loc}_{\Q_p}(X)$ the sheaf $\Bcrys\otimes_{\underline{\Q_p}}\bb{L}[\nicefrac{1}{p}]$ carries a natural Frobenius and filtration.

\medskip

\paragraph{The Faltings formulation} For an object $(\mc{E},\varphi,\Fil^\bullet_E)$ of $\cat{IsocF}^\varphi(\mf{X})$ and an affinoid perfectoid space $S=\Spa(R,R^+)$ over $C$ in $X_\proet$, which uniquely extends to a map $\Spf(R^+)\to \mf{X}$ (e.g.\@ those factorizing through $\Spf(A)_\eta$ for an affine open $\Spf(A)\subseteq \mf{X}$), we have that $\Acrys(S)$ is a pro-infinitesimal PD thickening of $R^{+}$, so we have the associated $\Bcrys(R,R^+)$-module
\begin{equation*}
    \Bcrys(\mc{E})(R,R^+)\defeq \mc{E}(\Acrys(R,R^+)\twoheadrightarrow R^+)[\nicefrac{1}{p},\nicefrac{1}{t}].
\end{equation*} 
As in \cite[\S2.3]{GuoReinecke}, this extends to a sheaf on $X_\proet$ and inherits a Frobenius morphism $\varphi\colon \phi^\ast\Bcrys(\mc{E})\to \Bcrys(\mc{E})$ and a filtration $\Fil^\bullet_{\Bcrys(\mc{E})}$ induced from $(\mc{E},\varphi,\Fil^\bullet_E)$. 

As in \cite[p.\@ 67]{Faltings89}, call an object $\bb{L}[\nicefrac{1}{p}]$ of $\cat{Loc}_{\Q_p}(X)$ (resp.\@ object $\bb{L}$ of $\cat{Loc}_{\Z_p}(X)$) \emph{crystalline relative to} $\mf{X}$ if there exists an object $(\mc{E},\varphi_\mc{E},\Fil^\bullet_E)$ of $\cat{IsocF}^\varphi(\mf{X})$ and an isomorphism of $\Bcrys$-modules
\begin{equation*}
    c_\mr{Fal}\colon \Bcrys\otimes_{\underline{\Q_p}}\bb{L}[\nicefrac{1}{p}]\isomto \Bcrys(\mc{E}),
\end{equation*}
compatible with Frobenius and filtration (resp.\@ $\bb{L}$ is crystalline). Denote by $\cat{Loc}_{\Q_p}^\crys(X)$ (resp.\@ $\cat{Loc}_{\Z_p}^\crys(X)$), suppressing $\mf{X}$ from the notation, the full subcategory of crystalline $\Q_p$-local (resp.\@ $\Z_p$-local) systems. By the arguments in \cite[Corollary 2.35]{GuoReinecke}, the filtered $F$-isocrystal $\mc{E}$ is functorial in $\bb{L}[\nicefrac{1}{p}]$. We say that $(\mc{E},\varphi,\Fil^\bullet_E)$ and $\bb{L}[\nicefrac{1}{p}]$ are \emph{associated}, and write $D_\crys(\bb{L})$ for $\mc{E}$. We futher define $D_\crys(\bb{L})\defeq D_\crys(\bb{L}[\nicefrac{1}{p}])$ for a crystalline $\Z_p$-local system $\bb{L}$. As explained in \cite[Proposition 3.22]{TanTong} and \cite[Corollary 2.37]{GuoReinecke}, any crystalline local system $\bb{L}$ on $X$ is de Rham, and if $D_\crys(\bb{L})=(\mc{E},\varphi_\mc{E},\Fil^\bullet_E)$ then $D_\dR(\bb{L})$ is equal to $(E,\nabla_E,\Fil^\bullet_E)$ (cf.\@ \S\ref{ss:filtered-f-isocrystals}).

Let $S=\Spa(R,R^+)$ be an affinoid perfectoid $C$-space $S=\Spa(R,R^+)$ in $X_\proet$, which uniquely extends to a map $\Spf(R^+)\to \mf{X}$. We consider the isomorphism 
\be
\theta^+_\crys\colon D_\crys(\mbb L)(\Acrys(S)\to R^+)\otimes_{\Acrys(S)}\mc O\Bcrys^+(S)\isomto (D_\crys(\mbb L)\otimes_{\mc O_X}\mc O\Bcrys^+)(S)
\ee
from the paragraph right after \cite[Proposition 3.21]{TanTong}. Put $\theta_\crys\defeq \theta_\crys^+[\nicefrac{1}{t}]$. 
Let $\theta^+_\dR$ denote the scalar extension of $\theta^+_\crys$ along $\mc O\Bcrys^+(S)\to \mc O\Bdr^+(S)$, and 
\begin{equation}\label{eq: Faltings c and Scholze c}
    \theta^{+,\nabla}_\dR\colon D_\crys(\mbb L)(\Acrys(S)\to R^+)\otimes_{\Acrys(S)}\Bdr^+(S)\isomto (D_\crys(\mbb L)\otimes_{\mc O_X}\mc O\Bdr^+)(S)^{\nabla=0}
\end{equation}
be the map induced on the spaces of horizontal sections. Put $\theta^\nabla_\dR\defeq \theta^{+,\nabla}_\dR[\nicefrac{1}{t}]$. Since, for a crystalline local system $\mbb L$ on $X$, the isomorphism $c_\mr{Sch}$ is given as the scalar extension of the composition $\theta_\crys\circ c_\mr{Fal}$ as explained in the proof of \cite[Corollary 2.37]{GuoReinecke}, we obtain the following lemma, which is used in the proof of Theorem \ref{thm:prismatic-F-crystal-shtuka-relationship}. 
\begin{lem}\label{lem: Faltings c and Scholze c}
    Assume that $\mbb L|_{S}$ is constant. 
    Then the diagram 
    \begin{equation*}
    \xymatrixcolsep{3pc}\xymatrix{
    (\mbb L\otimes_{\underline{\Z_p}}\Bdr)(S)\ar[r]^-{c_\mr{Fal}\otimes 1}\ar[rd]_-{c_\mr{Sch}\otimes 1}
    & D_\crys(\mbb{L})(\Acrys(S)\to R^+)\otimes_{\Acrys(S)}\Bdr(S)
    \ar[d]^{\theta_\dR^\nabla}
    \\& (D_\dR(\mbb{L})\otimes_{\mc O_X}\mc O\Bdr)(S)^{\nabla=0}
    }
    \end{equation*}
    commutes.
\end{lem}

For a conjugacy class $\bm{\mu}$ of cocharacters of $G_{\ov{K}}$ denote by $\GLoc_{\Z_p,\bm{\mu}}^{\crys}(X)$ those $\mc{G}(\Z_p)$-local systems which lie in both $\GLoc_{\Z_p}^\crys(X)$ and $\GLoc_{\Z_p,\bm{\mu}}^{\dR}(X)$.

\medskip

\paragraph{The Brinon formulation}Fix a base $\mc{O}_K$-algebra $A=A_0\otimes_W \mc{O}_K$, with formal framing $t\colon T_d\to A_0$. Consider the following rings:
\begin{itemize}
    \item $\mc{O}\Acrys(\check{A})$ is the $p$-adically completed PD envelope of the map $\theta \colon A_0\otimes_W \Ainf(\check{A})\to \check{A}$, filtered by $\Fil^\bullet_\mr{PD}$,
    \item $\mc{O}\Bcrys(\check{A})\defeq \mc{O}\Acrys(\check{A})[\nicefrac{1}{p},\nicefrac{1}{t}]$ filtered by 
    \begin{equation*}
    \Fil^r\mc{O}\Bcrys(\check{A})=\sum_{n\geqslant -r}t^{-n}\Fil^{n+r}\mc{O}\Acrys(\check{A}).
    \end{equation*}
\end{itemize}
The tensor product Frobenius $\phi$ on $A_0\otimes_W \Acrys(\check{A})$ uniquely extends to $\mc{O}\Acrys(\check{A})$ and we can extend the Frobenius to $\mc{O}\Bcrys(\check{A})$ with $\phi(\nicefrac{1}{t})=\nicefrac{1}{pt}$. There is a natural connection
\begin{equation*}
    \nabla \colon \mc{O}\Acrys(\check{A})\to \mc{O}\Acrys(\check{A})\otimes_{A_0}\wh{\Omega}^1_{A_0/W},
\end{equation*}
with $\Acrys(\check{A})$ and $\phi$ horizontal (see \cite[Proposition 4.3]{KimBK}), which extends to $\Bcrys(\check{A})$ by base change. There is a natural action of $\Gamma_A$ on $\mc{O}\Bcrys(\check{A})$ and the tautological morphism $A_0[\nicefrac{1}{p}]\to \mc{O}\Bcrys(\check{A})^{\Gamma_A}$ is an isomorphism (see \cite[Proposition 6.2.9]{Brinon}). 

For a finite-dimensional continuous $\Q_p$-representation $\rho\colon \Gamma_A\to\GL_{\Q_p}(V)$, write
\begin{equation*}
    D_\crys(V)\defeq (\mc{O}\Bcrys(\check{A})\otimes_{\Q_p}V)^{\Gamma_A},
\end{equation*}
which is an $A_0[\nicefrac{1}{p}]$-module. There is a natural morphism of $\mc{O}\Bcrys(\check{A})$-modules
\begin{equation*}
    \alpha_\crys\colon \mc{O}\Bcrys(\check{A})\otimes_{A_0[\nicefrac{1}{p}]}D_\crys(V)\to \mc{O}\Bcrys(\check{A})\otimes_{\Q_p}V
\end{equation*}
and following \cite{Brinon} we say $\rho$ (or $V$) is \emph{crystalline} if it is an isomorphism. This notion is independent of the choice of $A_0$ (see \cite[Proposition 8.3.5]{Brinon}). A finite-rank free $\Z_p$-representation $\Lambda$ of $\Gamma_A$ is \emph{crystalline} if $\Lambda[\nicefrac{1}{p}]$ is, in which case we write $D_\crys(\Lambda)\defeq D_\crys(\Lambda[\nicefrac{1}{p}])$. 

Denote by $\cat{Rep}^\crys_{\Q_p}(\Gamma_A)$ (resp.\@ $\cat{Rep}_{\Z_p}^\crys(\Gamma_A)$) the category of crystalline $\Q_p$-representations (resp.\@ $\Z_p$-representations) endowed with the evident structure of an exact $\Z_p$-linear $\otimes$-category. If $V$ is an object of $\cat{Rep}_{\Q_p}(\Gamma_A)$ then $D_\crys(V)$ has the structure of an object of $\cat{IsocF}^\varphi(A)$, and \cite[Th\'eor\`eme 8.5.1]{Brinon} defines a $\Z_p$-linear $\otimes$-functor
\begin{equation*}
    D_\crys\colon \cat{Rep}^\crys_{\Q_p}(\Gamma_A)\to \cat{IsocF}^\varphi(A),
\end{equation*}
which is a bi-exact equivalence onto its image, functorial in $A$. 

Let $\Sigma$ be either $\Z_p$ or $\Q_p$. If $\{\Spf(A)\}$ is a small open cover of $\mf{X}$, then an object $\bb{L}$ of $\cat{Loc}_{\Sigma}(X)$ is crystalline if and only if the $\Sigma$-representation $V_A$ associated to $\bb{L}|_{\Spf(A)_\eta}$ is crystalline for all $A$ (see \cite[Proposition A.10]{DLMS}). In this case we have that $D_\crys(\bb{L})|_{\Spf(A)_\eta}$ agrees with $D_\crys(V_A)$ for all $A$. Thus, 
\begin{equation*}
    D_\crys\colon \cat{Loc}^\crys_{\Q_p}(X)\to \cat{IsocF}^\varphi(\mf{X}),
\end{equation*}
is a $\Z_p$-linear $\otimes$-functor which is a bi-exact equivalent onto its image. Also, $\cat{Loc}^\crys_{\Sigma}(X)$ is closed under duals, tensor products, direct sums, and subquotients (see \cite[Th\'eor\`eme 8.4.2]{Brinon}).

\begin{prop}\label{prop:crystalline-can-be-tested-on-faithful-rep} For an object $\omega$ of $\GLoc_{\Sigma}(X)$, $\omega$ factorizes through $\cat{Loc}^\crys_\Sigma(X)$ if and only if $\omega(V_0)$ is crystalline for some faithful $\Sigma$-representation $V_0$.
\end{prop}
\begin{proof} It suffices to prove the if condition. Moreover, as $\omega$ is crystalline if and only if $\omega[\nicefrac{1}{p}]$ is, we may suppose that $\Sigma=\Q_p$. By \cite[Proposition 3.1 (a)]{DeligneHodge} every object $V$ of $\cat{Rep}_{\Q_p}(G)$ occurs as a subquotient of $\bigoplus_i (V_0)^{\otimes m_i}\oplus (V_0^\vee)^{\otimes n_i}$ for some finite list of integers $m_i$ and $n_i$. As $\cat{Loc}_{\Q_p}^\crys(X)$ is closed under duals, tensor products, direct sums, and subquotients,  the claim follows. 
\end{proof}

\subsubsection{Crystalline \texorpdfstring{$\mc{G}(\Z_p)$}{G(Zp)}-local systems}
Consider the \emph{\'etale realization functor}
\begin{equation*}
    T_\et\colon \cat{Vect}^{\an,\varphi}(\mf{X}_\Prism)\to \cat{Loc}_{\Z_p}(X),
\end{equation*}
defined to be the composition
\begin{equation*}
\cat{Vect}^{\an,\varphi}(\mf{X}_\Prism)\to \cat{Vect}^\varphi(\mf{X}_\Prism,\mc{O}_\Prism[\nicefrac{1}{\mc{I}_\Prism}]^\wedge_p)\xrightarrow{\,T_\et\,}\cat{Loc}_{\Z_p}(X),
\end{equation*}
where the first functor is obtained by patching together the pullbacks  $\Spec(A[\nicefrac{1}{I}]^\wedge_p)\to U(A,I)$. 

\begin{thm}[{\cite{GuoReinecke} (cf.\@ \cite{DLMS})}]\label{thm:GR-DLMS} The functor $T_\et$ induces an equivalence
\begin{equation*}
    T_\et\colon \cat{Vect}^{\an,\varphi}(\mf{X}_\Prism)\isomto \cat{Loc}^\crys_{\Z_p}(X).
\end{equation*}
\end{thm}

We devote the rest of this subsection to proving the following claim. 

\begin{prop}\label{prop:GR-DLMS-exact} The functor 
\begin{equation*}
    T_\et\colon \cat{Vect}^{\an,\varphi}(\mf{X}_\Prism)\to \cat{Loc}^\crys_{\Z_p}(X)
\end{equation*}
is a bi-exact $\Z_p$-linear $\otimes$-equivalence. In particular, $T_\et$ induces an equivalence of categories
\begin{equation*}
    \GVect^{\an,\varphi}(\mf{X}_\Prism)\to\GLoc^\crys_{\Z_p}(X).
\end{equation*}
\end{prop}

As the base extension functor $\cat{Vect}^{\an,\varphi}(\mf{X}_\Prism)\to \cat{Vect}^\varphi(\mf{X}_\Prism,\mc{O}_\Prism[\nicefrac{1}{\mc{I}_\Prism}]^\wedge_p)$ is exact, as an exact sequence of vector bundles is universally exact (in the sense of \stacks{058I}), exactness of $T_\et$ follows from Proposition \ref {prop:Laurent-crystal-equivalence}. So, we have reduced ourselves to showing that $T_\et^{-1}$ is exact. But, combining Proposition \ref{prop:cover-of-final-object-qrsp} and Proposition \ref{prop:exactness-preserved-cover-analytic-prismatic-crystals} we are reduced to showing the following.

\begin{prop}\label{prop:exactness on BK prism}
If $R$ is a small $\mc{O}_K$-algebra, then the following composition is exact: 
\begin{equation*}
    \cat{Rep}_{\Z_p}^\crys(\Gamma_{R})\isomto \cat{Loc}^\crys_{\Z_p}(\Spa(R[\nicefrac{1}{p}]))\xrightarrow{T_\et^{-1}}\cat{Vect}^{\an,\varphi}(R_\Prism)\xrightarrow{\mathrm{eval.}}\cat{Vect}^\an(\Ainf(\wt{R}),(\tilde{\xi})).
\end{equation*}
\end{prop}

Define
\begin{equation*}
    \begin{aligned}
        \mathrm{B}_{[\nicefrac{1}{p},\infty]} &\defeq \Ainf(\widetilde{R})[\nicefrac{[p^\flat]^p}{p}]^\wedge_p[\nicefrac{1}{p}]\\ 
        \mathrm{B}_{[0,\nicefrac{1}{p}]} &\defeq \Ainf(\widetilde{R})[\nicefrac{p}{[p^\flat]^p}]^\wedge_{[p^\flat]^p}[\nicefrac{1}{[p^\flat]^p}]\\ 
        \mathrm{B}_{[\nicefrac{1}{p},\nicefrac{1}{p}]} &= \mathrm{B}_{\nicefrac{1}{p}}\defeq \Ainf(\widetilde{R})[\nicefrac{[p^\flat]^p}{p},\nicefrac{p}{[p^\flat]^p}]^\wedge[\nicefrac{1}{p[p^\flat]^p}],
    \end{aligned}
\end{equation*}
where the final completion is either the $p$-adic or $[p^\flat]^p$-adic completion.\footnote{As in \cite[\S12.2]{ScholzeBerkeley}, denote by $\mc{Y}$ the analytic locus $\Spa(\Ainf(\widetilde{R}))^\an$. This is an analytic adic space equipped with a surjective continuous map $\kappa\colon |\mc{Y}|\to [0,\infty]$. For $[a,b]\subseteq [0,\infty]$ with rational (possibly infinite) endpoints, set $\mathcal{Y}_{[a,b]}\defeq \kappa^{-1}([a,b])^\circ$, which is a rational open subset of $\Spa(\Ainf(\widetilde{R}))$. When comparing references it is useful to observe that for each $[a,b]$ in $\{[0,\nicefrac{1}{p}],[\nicefrac{1}{p},\infty],[\nicefrac{1}{p},\nicefrac{1}{p}]\}$ there is a natural isomorphism $\mathrm{B}_{[a,b]}\to \mc{O}_{\mc{Y}}(\mc{Y}_{[a,b]})$.\label{footnote:analytic-algebraic}} We denote by $\wt{\varphi}$ the following composition
\begin{equation*}
    \Acrys(\wt{R})[\nicefrac{1}{p}]=\Ainf(\wt{R})[\tfrac{\xi^n}{n!}]^\wedge_p[\nicefrac{1}{p}]\xrightarrow{\phi}\Ainf(\wt{R})[\tfrac{\tilde \xi^n}{n!}]^\wedge_p[\nicefrac{1}{p}]\hookrightarrow \Ainf(\wt{R})[\nicefrac{[p^\flat]^p]}{p}]^\wedge_p[\nicefrac{1}{p}]=B_{[\nicefrac{1}{p},\infty]},
\end{equation*}
which we use to view $\mathrm{B}_{[\nicefrac{1}{p},\infty]}$ as an $\Acrys(\widetilde{R})[\nicefrac{1}{p}]$-algebra. 

There is a natural pullback functor
\begin{equation*}
    \cat{Vect}\left(\Spec(\Ainf(\widetilde{R}))-V(p,[p^\flat]^p)\right)\to \cat{Vect}(\mathrm{B}_{[0,\nicefrac{1}{p}]})\times_{\cat{Vect}(\mathrm{B}_{\nicefrac{1}{p}})}\cat{Vect}(\mathrm{B}_{[\nicefrac{1}{p},\infty]}),
\end{equation*}
as the natural maps $\Spec(\mathrm{B}_{[a,b]})\to \Spec(\Ainf(\widetilde{R}))$ for $[a,b]\in\{[0,\nicefrac{1}{p}],[\nicefrac{1}{p},\infty]\}$ factorize through $\Spec(\Ainf(\widetilde{R}))-V(p,[p^\flat]^p)$, and are equalized when composed with $\Spec(\mathrm{B}_{\nicefrac{1}{p}})\to\Spec(\mathrm{B}_{[a,b]})$. For a vector bundle $M$ on $\Spec(\Ainf(\wt{R}))-V(p,[p^\flat]^p)$ we denote by $M_{[a,b]}$, for $[a,b]\in\{[0,\nicefrac{1}{p}],[\nicefrac{1}{p},\infty]\}$, the induced vector bundle on $\mr{B}_{[a,b]}$.

\begin{lem}[{cf.\@ \cite[Theorem 3.8]{KedlayaRing-Theoretic}}]\label{lem:Kedlaya gluing} The functor
\begin{equation*}
    \cat{Vect}\left(\Spec(\Ainf(\widetilde{R}))-V(p,[p^\flat]^p)\right)\to \cat{Vect}(\mathrm{B}_{[0,\nicefrac{1}{p}]})\times_{\cat{Vect}(\mathrm{B}_{\nicefrac{1}{p}})}\cat{Vect}(\mathrm{B}_{[\nicefrac{1}{p},\infty]})
\end{equation*}
is a bi-exact $\Z_p$-linear $\otimes$-equivalence.
\end{lem}
\begin{proof} By the proof of \cite[Theorem 3.8]{KedlayaRing-Theoretic}, it remains to prove this functor is bi-exact. The exactness follows as an exact sequence of vector bundles is universally exact. To prove that the quasi-inverse is exact, from acyclicity of vector bundles on sheafy affinoid adic spaces and \cite[Proposition 3.2]{KedlayaRing-Theoretic} as a whole, we deduce exactness of the quasi-inverse to the functor in \cite[Proposition 3.2 (a)]{KedlayaRing-Theoretic}, which we apply to (defaulting to the notation in \cite[Definition 3.5]{KedlayaRing-Theoretic}) $A_1\to B_{1}\oplus B_2'$ and $A_2\to B_1'\oplus B_2$ to get the desired exactness.
\end{proof}

\begin{rem}
The proof of this lemma implies exactness of Kedlaya's equivalence between vector bundles on $\Spec(\Ainf(\widetilde{R}))- V\left(p,[p^\flat]^p\right)$ and those on $\Spa(\Ainf(\widetilde{R}))_\an$ (cf.\@ Footnote \ref{footnote:analytic-algebraic}), which seems well-known. Though the adic space perspective could clarify the following proof of Lemma \ref{prop:exactness on BK prism}, we have chosen to avoid it for the sake of brevity.
\end{rem}

We now recall some structural properties of the functor $T_\et^{-1}$.
\begin{lem}\label{lem:triviality on Y}
For an object $\Lambda$ of $\cat{Rep}_{\Z_p}^\crys(\Gamma_R)$, set $\mc{M}\defeq T_\et^{-1}(\Lambda)$ and $M\defeq \mc{M}(\Ainf(\widetilde{R}),(\tilde\xi))$. 
\begin{enumerate}
    \item There is a natural isomorphism of $\mathrm{B}_{[0,\nicefrac{1}{p}]}$-modules
    \be
    M_{[0,\nicefrac{1}{p}]}\isomto \Lambda\otimes_{\Z_p}\mathrm{B}_{[0,\nicefrac{1}{p}]}.
    \ee 
    \item Let $\mc E=D_\crys(\Lambda)$. Then there is a natural isomorphism
    \begin{equation}\label{eq:GR-second-isom}
   M_{[\nicefrac{1}{p},\infty]}[\nicefrac{1}{\tilde{\xi}}]\isomto\mc E(\Acrys(\widetilde{R})\twoheadrightarrow \wt{R})[\nicefrac{1}{p}]\otimes_{\Acrys(\widetilde{R})[\nicefrac{1}{p}],\wt{\varphi}}\mathrm{B}_{[\nicefrac{1}{p},\infty]}[\nicefrac{1}{\tilde{\xi}}].
    \end{equation}
\end{enumerate}
\end{lem}

\begin{proof}
In the proof of \cite[Theorem 4.8]{GuoReinecke}, the prismatic crystal $\mc M$ is described as the vector bundle $\mc M_{\widetilde{R}}$ on $\Spec(\Ainf(\widetilde{R})) - V\left(p,[p^\flat]^p\right)$ constructed in \cite[Theorem 4.15]{GuoReinecke}, together with a descent datum on $\widetilde{R}\hat{\otimes}_R \widetilde{R}$. In particular, $M$ is naturally isomorphic to $\mc M_{\widetilde{R}}$. In the proof of \cite[Theorem 4.15]{GuoReinecke}, the module $\mc M_{\widetilde{R}}$ is constructed by gluing the $\mathrm{B}_{[\nicefrac{1}{p},\infty]}$-module $\mc M_{3,\widetilde{R}}$ constructed in \cite[Proposition 4.11]{GuoReinecke} and the $\mathrm{B}_{[0,\nicefrac{1}{p}]}$-module $\Lambda\otimes_{\Z_p}\mathrm{B}_{[0,\nicefrac{1}{p}]}$ via the equivalence in Lemma \ref{lem:Kedlaya gluing}. In particular, assertion (1) follows by definition. To prove assertion (2), note that $\mc M_{3,\widetilde{R}}$ is obtained by the Beauville--Laszlo gluing of the $\mathrm{B}_{[\nicefrac{1}{p},\infty]}[\nicefrac{1}{\tilde{\xi}}]$-module
\begin{equation*}
    \mc{E}(\Acrys(\widetilde{R})\twoheadrightarrow \wt{R})[\nicefrac{1}{p}]\otimes_{\Acrys(\widetilde{R})[\nicefrac{1}{p}],\wt{\varphi}}\mathrm{B}_{[\nicefrac{1}{p},\infty]}[\nicefrac{1}{\tilde{\xi}}]
\end{equation*}
and a certain $(\mathrm{B}_{[\nicefrac{1}{p},\infty]})_{\tilde{\xi}}^\wedge$-module (specifically the submodule $\wt{\mr{Fil}}^0_{E}$ of the $(\mathrm{B}_{[\nicefrac{1}{p},\infty]})_{\tilde{\xi}}^\wedge[\nicefrac{1}{\tilde \xi}]$-module $\mc E(\Acrys(\widetilde{R})\twoheadrightarrow \wt{R})[\nicefrac{1}{p}]\otimes_{\Acrys(\widetilde{R})[\nicefrac{1}{p}],\tilde\varphi}(\mathrm{B}_{[\nicefrac{1}{p},\infty]})_{\tilde{\xi}}^\wedge[\nicefrac{1}{\tilde\xi}]$
). So, (2) again follows by definition.
\end{proof}

\begin{proof}[Proof of Proposition \ref{prop:exactness on BK prism}] By Lemma \ref{lem:Kedlaya gluing} we are reduced to the assertions that the functors
\begin{equation*}
\begin{aligned} &\cat{Rep}^\crys_{\Z_p}(\Gamma_R)\to \cat{Vect}(\mathrm{B}_{[0,\nicefrac{1}{p}]}),\qquad &\Lambda \mapsto T_\et^{-1}(\Lambda)(\Ainf(\widetilde{R}))_{[0,\nicefrac{1}{p}]},\\
&\cat{Rep}^\crys_{\Z_p}(\Gamma_R)\to \cat{Vect}(\mathrm{B}_{[\nicefrac{1}{p},\infty]}),\qquad &\Lambda \mapsto T_\et^{-1}(\Lambda)(\Ainf(\widetilde{R}))_{[\nicefrac{1}{p},\infty]},
\end{aligned}
\end{equation*}
are exact. The first functor being exact follows immediately from the first assertion of Lemma \ref{lem:triviality on Y} and the fact that $\Z_p\to \mathrm{B}_{[0,\nicefrac{1}{p}]}$ is flat. As Beauville--Laszlo gluing is exact, to prove that the second functor is exact, it suffices to show that the functors 
\begin{equation}\label{eq:refined-functor-1}
\cat{Rep}^\crys_{\Z_p}(\Gamma_R)\to \cat{Vect}(\mathrm{B}_{[\nicefrac{1}{p},\infty]}[\nicefrac{1}{\tilde{\xi}}]),\qquad \Lambda \mapsto T_\et^{-1}(\Lambda)(\Ainf(\widetilde{R}))_{[\nicefrac{1}{p},\infty]}[\nicefrac{1}{\tilde{\xi}}],
\end{equation}
and
\begin{equation}\label{eq:refined-functor-2}
\cat{Rep}^\crys_{\Z_p}(\Gamma_R)\to \cat{Vect}((\mathrm{B}_{[\nicefrac{1}{p},\infty]})^\wedge_{\tilde{\xi}}),\qquad \Lambda \mapsto \left[T_\et^{-1}(\Lambda)(\Ainf(\widetilde{R}))_{[\nicefrac{1}{p},\infty]}\right]_{\tilde{\xi}}^\wedge,
\end{equation}
are exact. The functor in \eqref{eq:refined-functor-1} being exact follows from the second assertion of Lemma \ref{lem:triviality on Y} as $D_\crys$ is exact. To see that the functor in \eqref{eq:refined-functor-2} is exact, we observe that there is an identification $(\mathrm{B}_{[0,\nicefrac{1}{p}]})_{\tilde{\xi}}^\wedge=(\mathrm{B}_{[\nicefrac{1}{p},\infty]})_{\tilde{\xi}}^\wedge$, and thus we are again reduced to the first assertion of Lemma \ref{lem:triviality on Y}.
\end{proof}

\subsection{\texorpdfstring{$\mc{G}$}{G}-objects in the category of prismatically good reduction local systems}\label{ss:G-objects-strongly-crystalline} We now wish to extend some of the results of the last subsection to the case of prismatic $F$-crystals. For the remainder of this subsection, we assume that $\mf{X}\to\Spf(\mc{O}_K)$ is smooth and $\mc{G}$ is reductive.

\begin{defn}
    The category $\cat{Loc}^{\strcrys}_{\Z_p}(X)$ of \emph{prismatically good reduction} $\Z_p$-local systems on $X$ (relative to $\mf{X}$) is the full exact $\Z_p$-linear $\otimes$-subcategory of $\cat{Loc}_{\Z_p}^\crys(X)$ consisting of those $\bb{L}$ with $T_\et^{-1}(\bb{L})$ a prismatic $F$-crystal on $\mf{X}$.
\end{defn}

 If $\mf{X}=\Spf(\mc{O}_K)$, then every crystalline $\Z_p$-local system has prismatically good reduction (cf.\@ \cite[Proposition 3.8]{GuoReinecke}), but for higher-dimensional $\mf{X}$ this ceases to be the case (cf.\@ \cite[Example 3.36]{DLMS}). There is a $\Z_p$-linear $\otimes$-equivalence
\begin{equation}\label{eq:T-equiv-strcrys-vect}
    T_\et\colon \cat{Vect}^\varphi(\mf{X}_\Prism)\to\cat{Loc}^\strcrys_{\Z_p}(X).
\end{equation}
This is exact as $\cat{Vect}^\varphi(\mf{X}_\Prism)\to \cat{Vect}^{\an,\varphi}(\mf{X}_\Prism)$ is, and so induces a functor 
\begin{equation*}
    T_\et\colon \GVect^\varphi(\mf{X}_\Prism)\to \GLoc^\strcrys_{\Z_p}(X).
\end{equation*} 
That said, the quasi-inverse to the functor in \eqref{eq:T-equiv-strcrys-vect} is not exact even for $\mf{X}=\Spf(\mc{O}_K$) as its evaluation at the Breuil--Kisin prism is the functor $\mf{M}$ from \cite{KisinFCrystal} (see \cite[Remark 7.11]{BhattScholzeCrystals}), which is known to not be exact (e.g.\@ see \cite[Example 4.1.4]{LiuDifferentUnif}).

But, despite the functor in \eqref{eq:T-equiv-strcrys-vect} not being bi-exact, we still have the following.

\begin{thm}\label{thm:equiv-G-Vect-and-G-strcrys} The functor 
\begin{equation*}
T_\et\colon \GVect^\varphi(\mf{X}_\Prism)\to\GLoc^\strcrys_{\Z_p}(X)
\end{equation*}
is an equivalence.
\end{thm}
\begin{proof} From Proposition \ref{prop:GR-DLMS-exact}, it is clear that this functor is fully faithful. To show that it is essentially surjective, fix $\omega$ in $\GLoc^\crys_{\Z_p}(X)$ such that $\omega(\Lambda_0)$ has prismatically good reduction. Write $(\mc{E}_0,\varphi_{\mc{E}_0})$ for the associated object of $\cat{Vect}^\varphi(\mf{X}_\Prism)$ and let $(\mc{V}_0,\varphi_{\mc{V}_0})$ denotes its image in $\cat{Vect}^{\an,\varphi}(\mf{X}_\Prism)$. As the functor $\cat{Vect}^\varphi(\mf{X}_\Prism)\to \cat{Vect}^{\an,\varphi}(\mf{X}_\Prism)$ is fully faithful, we see that the tensors
\begin{equation*}
    T_\et^{-1}(\omega(\mathds{T}_0))\subseteq \Hom\left((\mathcal{O}_\Prism^\an,\phi),(\mc{V}_0,\varphi_{\mc{V}_0})^\otimes\right)
\end{equation*}
obtained from Proposition \ref{prop:GR-DLMS-exact} uniquely lift to a set of tensors $\mathds{T}_\Prism\subseteq (\mc{E}_0,\varphi_{\mc{E}_0})^\otimes$. Set
\begin{equation*}
    \mc{Q}_\omega\defeq \underline{\Isom}\left((\Lambda_0\otimes_{\Z_p}\mc{O}_\Prism,\mathds{T}_0\otimes 1),(\mc{E}_0,\mathds{T}_\Prism)\right),
\end{equation*}
with the Frobenius structure inherited from $(\mc{E}_0,\mathds{T}_\Prism)$. This is a pseudo-torsor for $\mc{G}_\Prism$ on $\mf{X}_\Prism$.

\begin{prop}\label{prop:main-torsor-claim}
    The pseudo-torsor $\mc{Q}_\omega$ is a torsor.
\end{prop}
\begin{proof} For any small affine open subset $\Spf(R)$ of $\mf{X}$, set 
\begin{equation*}
    (M_R,\mathds{T}_R):=(\mc{E}_0,\mathds{T}_\Prism)(\mf{S}_R,(E)),\quad \mc{Q}_{\omega,R}\defeq\underline{\Isom}\left((\Lambda_0\otimes_{\Z_p}\mf{S}_R,\mathds{T}_0\otimes 1),(M_R,\mathds{T}_R)\right),
\end{equation*}
considered as a pseudo-torsor for $\mc{G}$ on $\Spec(\mf{S}_R)_\et$. By Corollary \ref{prop:small-open-cover-topos-cover}, it suffices to show that $\mc{Q}_{\omega,R}$ is a torsor for all such $R$. But, the restriction of $\mc{Q}_{\omega,R}$ to $U(\mf{S}_R,(E))$ is identified with
\begin{equation*}
    \underline{\Isom}\left((\Lambda_0\otimes_{\Z_p}\mc{O}_{U(\mf{S}_R,E)},\mathds{T}_0\otimes 1),(\mc{V}_0,T_\et^{-1}(\omega(\mathds{T}_0)))|_{(\mf{S}_R,(E))}\right). 
\end{equation*}
Proposition \ref{prop:GR-DLMS-exact} implies that $\mc{Q}_{\omega,R}|_{U(\mf{S}_R,E)}$ is a torsor. As the height of $(p,E)\subseteq \mf{S}_R$ is $2$ and $M_R$ is a vector bundle, $\mc{Q}_{\omega,R}$ is a torsor by Proposition \ref{prop:reflexive-pseu-tors-is-tors-criterion} or Remark \ref{rem:alternative-to-rflx-pseudo-torsors}. 
\end{proof}
Let $\nu$ be the object of $\GVect^\varphi(\mf{X}_\Prism)$ associated to $\mc{Q}_\omega$ by Proposition \ref{prop:varphi-equivariant-GVect-and-tors-identification}. We claim that $T_\et\circ \nu$
is isomorphic to $\omega$. But, by Proposition \ref{prop:varphi-equivariant-GVect-and-tors-identification} it suffices to observe that, by setup, both $T_\et\circ \nu$ and $\omega$ have the value $(\omega(\Lambda_0),\omega(\mathds{T}_0))$ when evaluated on $(\Lambda_0,\mathds{T}_0)$.
\end{proof}

As a byproduct of the above proof and Theorem \ref{thm:broshi1} (which implies every faithful representation can be upgraded to a tensor package) we obtain an analogue of Proposition \ref{prop:crystalline-can-be-tested-on-faithful-rep}.

\begin{cor}\label{cor:strcrys-faithful-condition} Let $\Lambda$ be a faithful representation of $\mc{G}$. Then, an object  $\omega$ of $\GLoc_{\Z_p}(X)$ belongs to $\GLoc^\strcrys_{\Z_p}(\mf{X})$ if and only if $\omega(\Lambda)$ is an object of $\cat{Loc}^\strcrys_{\Z_p}(X)$.
\end{cor}

Proposition \ref{prop:varphi-equivariant-GVect-and-tors-identification} and Theorem \ref{thm:equiv-G-Vect-and-G-strcrys} yield an equivalence $\cat{Tors}_\mc{G}^\varphi(\mf{X}_\Prism)\isomto \GLoc_{\Z_p}^\strcrys(X)$ which we also denote by $T_\et$ (or $T_{\mf{X},\et}$), which is compatible in $\mf{X}$ and $\mc{G}$ in the obvious way.

\subsection{Complementary results about base \texorpdfstring{$\mc{O}_K$}{OK}-algebras}\label{ss:complementary-results}
While the proof of Theorem \ref{thm:equiv-G-Vect-and-G-strcrys} was built on the work of \cite{GuoReinecke}, Proposition \ref{prop:main-torsor-claim} works more generally, using \cite[Proposition 1.3.4]{KisIntShab} and the results of \cite{DLMS}. This could be potentially useful in other contexts (e.g., it provides an alternative method to prove some results in \cite{IKY2}). 

Let $R$ be a (formally framed) base $\mc{O}_K$-algebra. In \cite[\S4.4]{DLMS},\footnote{While in loc.\@ cit.\@ the authors only construct this functor for non-negative Hodge--Tate weights,  this definition can be easily extended using Breuil--Kisin twists.} there is constructed a $\Z_p$-linear $\otimes$-functor
\begin{equation*}
    \mf{M}\colon \cat{Rep}_{\Z_p}^\crys(\Gamma_R)\to \cat{Vect}^{\an,\varphi}(\mf{S}_R,(E)). 
\end{equation*}
Let $j$ denote the inclusion $U(\mf{S}_R,(E))\hookrightarrow \Spec(\mf{S}_R)$. We say that a representation $\Lambda$ in $\cat{Rep}_{\Z_p}^\crys(\Gamma_R)$ has \emph{prismatically good reduction} if $j_\ast \mf{M}(\Lambda)$ is a vector bundle on $\Spec(\mf{S}_R)$. 

Let us suppose that $\Lambda_0$ carries the structure of an object of $\cat{Rep}_{\Z_p}^\crys(\Gamma_R)$. Denote $(\mf{M}(\Lambda_0),\mf{M}(T_0))$ by $(M^\an,\mathds{T}^\an)$, and denote the global sections of $j_\ast(\mf{M}(\Lambda_0),\mf{M}(T_0))$ by $(M,\mathds{T})$.

\begin{prop}\label{prop:base-ring-torsor}  Consider the following sheaf on $U(\mf{S}_R,(E))_\fpqc$:
\begin{equation*}
    \mc{Q}^\an=\underline{\Isom}\left((\Lambda_0\otimes_{\Z_p}\mc{O}_{U(\mf{S}_R,(E))},\mathds{T}_0\otimes 1),(M^\an,\mathds{T}^\an)\right).
\end{equation*}
Set $\mc{Q}\defeq j_\ast\mc{Q}^\an$. Then $\mc{Q}$ is a reflexive pseudo-torsor and if $\mc{G}$ is reductive, then
\begin{equation*}
\mc{Q}=\underline{\Isom}\left((\Lambda_0\otimes_{\Z_p}\mf{S}_R,\mathds{T}_0\otimes 1),(M,\mathds{T})\right),
\end{equation*}
and is a $\mc{G}$-torsor on $\Spec(\mf{S}_R)_\mr{fppf}$ if and only if $\Lambda_0$ has prismatically good reduction.
\end{prop}
\begin{proof} Given Proposition \ref{prop:pseudo-torsors-large-opens-equiv} and Proposition \ref{prop:reflexive-pseu-tors-is-tors-criterion}, the latter claims will follow if we can show that $\mc{Q}$ is a reflexive pseudo-torsor. As $\mc{Q}^\an$ is clearly affine and finite type over $U(\mf{S}_R,E)$, we further deduce from Proposition \ref{prop:rflx-omnibus} that it suffices to prove $\mc{Q}$ is a torsor after pulled back to all codimension $1$ points of $\Spec(\mf{S}_R)$. 

Let $\mc O_L$ (resp.\@ $\mc{O}_{L_0}$) denote the $p$-adic completion of the localization $R_{(p)}$ (resp.\@ $(R_0)_{(p)}$). Note that $\mathcal{O}_L=\mathcal{O}_{L_0}\otimes_W \mathcal{O}_K$ is a base $\mc{O}_K$-algebra, and write $(\mf{S}_L,(E))$ for its Breuil--Kisin prism. Furthermore, let $\mc O_\mc E$ denote the $p$-adic completion of $\mf S_R[u^{-1}]$. As the morphism $\Spec (\mf S_L)\sqcup \Spec(\mc O_\mc E)\to \Spec (\mf S_R)$ is flat (see \stacks{00MB}) and its image contains all codimension $1$ points, it suffices to show that $\mc{Q}$ is a $\mc{G}$-torsor on pullback to $\mf S_L$ and $\mc O_\mc E$.

To prove the first claim, note that the perfection $\varinjlim_{\phi}\mc O_{L_0}$ is a discrete valuation ring that is faithfully flat over $\mc O_{L_0}$. In particular, its $p$-adic completion $\mc O_{L'_0}$ is faithfully flat over $\mc O_{L_0}$. The ring $\mc{O}_{L'}\defeq \mc{O}_{L'_0}\otimes_W \mc{O}_K$ is a complete discrete valuation ring with perfect residue field, and, with notation as above, $\mf S_{L'}$ is faithfully flat over $\mf S_{L}$. Thus, it suffices to show $\mc{Q}_{\mf{S}_{L'}}$ is a $\mc{G}$-torsor. Recall though that $M\otimes_{\mf S_R}\mf S_{L'}$ is canonically identified with the (classical) Breuil--Kisin module associated to $\rho$ restricted to the absolute Galois group of $L'$ as in \cite{KisinFCrystal} (see \cite[Lemma 4.18 and Proposition 4.26]{DLMS}). Thus, from \cite[Proposition 1.3.4]{KisIntShab} we see that $\mc{Q}(\mf{S}_{L'})$ is non-empty. 

For the second claim, observe that (cf.\@ \cite[Proposition 4.26]{DLMS}) there is an isomorphism 
\begin{equation*}
M\otimes_{\mf{S}_R}\wh{\mc{O}}^\mathrm{ur}_\mc{E}\simeq \Lambda_0\otimes_{\Z_p}\wh{\mc{O}}^\mathrm{ur}_{\mc{E}},
\end{equation*}
in $\cat{Mod}^\varphi(\wh{\mc{O}}^\mathrm{ur}_\mc{E})$, functorial in $\Lambda_0$, where $\wh{\mc{O}}_\mc{E}^\mathrm{ur}$ is the $p$-adic completion of a colimit of finite \'etale extensions $\mc{O}^\mathrm{ur}_\mc{E}$ of $\mc{O}_\mc{E}$ with compatible extension of $\phi$ (see \cite[\S2.3]{DLMS}). From this functorial isomorphism, we see that $\mc{Q}(\wh{\mc{O}}_\mc{E}^\mathrm{ur})$ is non-empty, and so $\mc{Q}|_{\mc{O}_\mc{E}}$ is a $\mc{G}$-torsor.
\end{proof}

\section{Shtukas and analytic prismatic \texorpdfstring{$F$}{F}-crystals}\label{ss:shtukas-and-prismatic-F-crystals} In this final section we recall the notion of a $\mc{G}$-shtuka on a (formal) scheme and discuss the relationship to (analytic) prismatic $\mc{G}$-torsors with $F$-structure. Any undefined notation and conventions concerning shtukas, their accompanying adic spaces (e.g.\@ the spaces $\mc{Y}_I(S)$), and $v$-sheaves are as in \cite[\S2]{PappasRapoportI}.

\begin{nota}\label{nota:shtuka} Fix the following notation:
\begin{itemize}
    \item $k$ to be a perfect extension of $\bb{F}_p$,
    \item $W\defeq W(k)$ and $E_0\defeq \Frac(W)$, 
    \item $E$ is a totally ramified finite extension of $E_0$ with uniformizer $\pi$,
    \item $\ov{E}$ is an algebraic closure of $E$ which induces an algebraic closure $\ov{k}$ of $k$,
    \item $C$ denote the $p$-adic completion of $\ov{E}$,
    \item $\breve{W}\defeq W(\ov{k})$,
    \item $\mc{G}$ is a parahoric group $\Z_p$-scheme with generic fiber denoted $G$, 
    \item $\bm{\mu}$ is a conjugacy class of cocharacters of $G_{\ov{E}}$ with field of definition $E$ over $E_0$.
\end{itemize}
\end{nota}

\begin{rem}
    The assumption that $\mc G$ is parahoric is needed only when bounding the relative position of two lattices by $\mathbb \mu$. Otherwise, everything in this section works under the assumption that $\mc G$ is a smooth affine group $\Z_p$-scheme with connected fibers. 
\end{rem}

\subsection{\texorpdfstring{$\mc{G}$}{G}-shtukas}\label{sss: shtukas} In this subsection we recall the definition of a $\mc{G}$-shtuka over a (formal)-scheme. We fix notation as in Notation \ref{nota:shtuka}.

\subsubsection{\texorpdfstring{$\mc{G}$}{G}-shtukas over \texorpdfstring{$v$}{v}-sheaves}\label{sss:shtukas-over-v-sheaves} Denote by $\cat{Perf}_k$ the category of affinoid perfectoid spaces $S=\Spa(R,R^+)$ with $R^+$ a $k$-algebra, endowed with the $v$-topology. For any $v$-sheaf $\mc{F}$ on $\cat{Perf}_k$ we let $\cat{Perf}_{\mc{F}}$ denote the slice site over $\mc{F}$. In other words, an object of $\cat{Perf}_{\mc{F}}$ is a pair $(S,\alpha)$ where $S$ is an object of $\cat{Perf}_k$ and $\alpha\colon S\to \mc{F}$ is a morphism, a morphism $(S_1,\alpha_1)\to (S_2,\alpha_2)$ is a map $f\colon S_1\to S$ such that $\alpha_2\circ f=\alpha_1$, and such a map is a $v$-cover if $f$ is. 

Recall that if $\ms{X}$ is a pre-adic space over $\mc{O}_E$, then $\ms{X}^\lozenge$ is the $v$-sheaf on $\cat{Perf}_k$ associating to $S$ the set of pairs $(S^\sharp,f)$ where $S^\sharp$ is an untilt of $S$ and $f\colon S^\sharp\to \ms{X}$ is a morphism of adic spaces. This construction is functorial in $\ms{X}$. We shorten the slice site $\cat{Perf}_{\ms{X}^\lozenge}$ to $\cat{Perf}_{\ms{X}}$ and further to $\cat{Perf}_A$ if $\ms{X}=\Spa(A)$. We observe that objects of $\cat{Perf}_\ms{X}$ can be interpreted as pairs $(S^\sharp,f)$ where $S^\sharp$ is a perfectoid space over $\mc{O}_E$, and $f\colon S^\sharp\to\ms{X}$ is an $\mc O_E$-morphism, and we often use this interpretation without comment.

Let us denote by 
\begin{equation*}
   \cat{Sht}\colon \cat{Perf}_{\mc{O}_E}\to \cat{Grpd},\quad  \GSht\colon \cat{Perf}_{\mc{O}_E}\to \cat{Grpd},\quad \GShtmu\colon \cat{Perf}_{\mc{O}_E}\to \cat{Grpd}
\end{equation*}
the $v$-stacks associating to $(S,S^\sharp,\beta)$ the groupoid of shtukas over $S$ with a single leg at $S^\sharp$, the groupoid of $\mc{G}$-shtukas over $S$ with a single leg at $S^\sharp$, and the groupoid of $\mc{G}$-shtukas over $S$ with a single leg at $S^\sharp$ bounded by $\bm{\mu}$, respectively, as in \cite[Definition 2.4.3]{PappasRapoportI}.

\begin{defn}[{\cite[Definition 2.3.1]{PappasRapoportI}}]\label{defn:shtuka-over-X} For a $v$-sheaf $\mc{F}$ on $\cat{Perf}_{\mc{O}_E}$ and $\mathscr{S}$ an element of $\{\cat{Sht},\GSht,\GShtmu\}$, we define $\ms{S}(\mc{F})$ to be the groupoid of maps of $v$-stacks $\mc{F}\to \ms{S}$.
\end{defn}

Equivalently, each $\ms{S}$ as in Definition \ref{defn:shtuka-over-X} defines a natural stack $p_\mc{F}\colon \ms{S}\to \cat{Perf}_{\mc{F}}$ and the groupoid $\ms{S}(\mc{F})$ may be identified with the groupoid of Cartesian sections of $p_\mc{F}$ (cf.\@ \stacks{07IV}). In other words, an object of $\ms{S}(\mc{F})$ as a functorial rule associating to each $(S,\alpha)$ in $\cat{Perf}_\mc{F}$ an object of $\ms{S}(S,\alpha')$ where $\alpha'$ is the composition $S\xrightarrow{\alpha}\mc{F}\xrightarrow{\text{structure}}\Spd(\mc{O}_E)$. We treat these interpretations as interchangeable below.

\begin{rem} Let $\ms{S}$ be a $v$-stack on $\cat{Perf}_{\mc{O}_E}$. Then, one can naturally form the left Kan extension $\ms{S}\colon \cat{PSh}(\cat{Perf}_{\mc{O}_E})\to \cat{Grpd}$, since $\cat{Grpd}$ is $2$-complete. Concretely, 
\begin{equation*}
    \ms{S}(\mc{F})\defeq \twolim_{S\to \mc{F}}\ms{S}(S).
\end{equation*}
The above notions of shtuka-like objects on $v$-sheaves are a special case of this construction, as $\ms{S}$ commutes with $2$-limits, and $\displaystyle \mc{F}=\lim_{S\to \mc{F}}S$ (e.g., see \cite[Expos\'e I, 3.4.0]{SGA4-1}). This is related to the interpretation as Cartesian sections via the Grothendieck construction.
\end{rem}

We call an object of $\cat{Sht}(\mc{F})$, $\GSht(\mc{F})$, $\GShtmu(\mc{F})$ a \emph{shtuka over $\mc{F}$}, a $\mc{G}$-shtuka over $\mc{F}$, and a $\mc{G}$-shtuka over $\mc{F}$ bounded by $\mu$, respectively. These groupoids are evidently $2$-functorial in $\mc{F}$, and in fact form stacks on $\cat{Sh}(\cat{Perf}_{\mc{O}_E})$.

We often write an object of any of these shtuka-like groupoids over $\mc{F}$ as $(\ms{P},\varphi_{\ms{P}})$. For an object $(S,\alpha)$ of $\cat{Perf}_{\mc{O}_E}$, and an element of $\mc{F}(S,\alpha)$ (i.e., a morphism $(S,\alpha)\to \mc{F}$ or, equivalently, an element of $\cat{Perf}_\mc{F}$), we may pull back $(\ms{P},\varphi_{\ms{P}})$ to an object over $(S,\alpha)$. We denote this pullback by evaluation, i.e., by $(\ms{P},\varphi_{\ms{P}})(S,\alpha)$. 

Finally, let us observe that $\cat{Sht}(\mc{F})$ is actually an exact $\Z_p$-linear $\otimes$-category. Indeed, for each map $(S,\alpha)\to \mc{F}$ we get a map $(S,\alpha')$ in $\cat{Perf}_{\mc{O}_E}$ which corresponds to a triple $(S,S^\sharp,\beta)$. The category $\Sht(S,S^\sharp,\beta)$ is an exact $\Z_p$-linear $\otimes$-category in the usual way as vector bundles on $S\dot{\times}\Spa(\Z_p)$ with meromorphic Frobenius away from a fixed locus (i.e., $S^\sharp\subseteq S\dot{\times}\Spa(\Z_p)$). As the morphisms $\Sht(S,S^\sharp,\beta)\to \Sht(S_1 ,S_1^\sharp,\beta_1)$ are exact $\Z_p$-linear $\otimes$-functors for a morphism $(S_1,S_1^\sharp,\beta_1)\to (S,S^\sharp,\beta)$ (exact because exact sequences of vector bundles are universally exact), we see that we may endow $\cat{Sht}(\mc{F})$ with the structure of an exact $\Z_p$-linear $\otimes$-category by passing to the $2$-limit. It's clear that this structure is functorial in $\mc{F}$.

\begin{rem}\label{rem: G-Sht fits into the Tannakian formalism}
    It is evident (cf.\@ \cite[Appendix to Lecture 19]{ScholzeBerkeley}) that, as the notation suggests, $\GSht(\mc{F})$ is the category of $\mc{G}$-objects in $\Sht(\mc{F})$. In particular, there is a natural identification
\begin{equation*}
    \GL_n\text{-}\cat{Sht}(\mc{F})\isomto \cat{Sht}_n(\mc{F})\subseteq \cat{Sht}(\mc{F}),
\end{equation*}
where the target is the full subgroupoid of shtukas over $\mc{F}$ of rank $n$. We shall make these identifications freely in the sequel.
\end{rem}

\subsubsection{Some \texorpdfstring{$v$}{v}-sheaves associated to (formal) schemes}\label{sss:some-v-sheaves} We will be mostly interested in studying shtuka-like objects over $v$-sheaves which are obtained from (formal) schemes. In particular, we recall the following notational definitions.
\begin{itemize}
    \item If $\ms{X}$ is a locally of finite type $E$-scheme we shorten $(\ms{X}^\an)^\lozenge$ to $\ms{X}^\lozenge$ and shorten $\cat{Perf}_{\ms{X}^\an}$ to $\cat{Perf}_\ms{X}$.
    \item If $\ms{X}$ is a locally formally of finite type formal $\mc{O}_E$-scheme we shorten $(\ms{X}^\ad)^\lozenge$ to $\ms{X}^\lozenge$ and shorten $\cat{Perf}_{\ms{X}^\mr{ad}}$ to $\cat{Perf}_\ms{X}$.
    \item If $\ms{X}$ is a separated locally of finite type $\mc{O}_E$-scheme then we define
    \begin{equation}\label{eq:v-sheaf-gluing}
        \ms{X}^{\lozenge/}=\wh{\ms{X}}^\lozenge\sqcup_{(\wh{\ms{X}}^\lozenge)_E}(\ms{X}_E)^\lozenge. 
    \end{equation}
\end{itemize}
In \eqref{eq:v-sheaf-gluing} the morphism $(\wh{\ms{X}}^\lozenge)_E\to (\ms{X}_E)^\lozenge$ is the open embedding obtained via the composition
\begin{equation*}
    (\wh{\ms{X}}^\lozenge)_E=(\wh{\ms{X}}_E)^\lozenge\xrightarrow{j_\ms{X}^\lozenge} (\ms{X}_E)^\lozenge,
\end{equation*}
where $j_\ms{X}\colon \wh{\ms{X}}_E\to \ms{X}_E^\mr{an}$ is the open embedding from \cite[\S1.9]{HuberEC}. 



\subsubsection{Shtukas over (formal) schemes}
With the definitions from \S\ref{sss:some-v-sheaves}, we may import the definitions from \S\ref{sss:shtukas-over-v-sheaves}.

\begin{defn}\label{defn:Shtoversch}
Let $\ms{S}$ be an element of $\{\cat{Sht},\GSht,\GShtmu\}$:
\begin{itemize}
    \item if $\ms{X}$ is a locally of finite type $E$-scheme or a locally of finite type formal $\mc{O}_E$-scheme, we define $\ms{S}(\ms{X})\defeq \ms{S}(\ms{X}^\lozenge)$,
    \item if $\ms{X}$ is a separated finite type $\mc{O}_E$-scheme, we define $\ms{S}(\ms{X})\defeq \ms{S}(\ms{X}^{\lozenge/})$.
\end{itemize}
\end{defn}

To help contextualize this second definition, we make the following further definition. 

\begin{defn} The category $\cat{Tri}(\mc{O}_E)$ of \emph{gluing triples} over $\mc{O}_E$ has 
\begin{itemize}[leftmargin=.25in]
    \item objects of the form $(X,\mf{X},j)$ with $X$ a separated locally of finite type $E$-scheme, $\mf{X}$ a separated locally of finite type flat formal $\mc{O}_E$-scheme, and $j\colon \mf{X}_\eta\to X^\an$ an open embedding,
    \item morphisms $(f,g)\colon (X_1,\mf{X}_1,j_1)\to (X_2,\mf{X}_2,j_2)$ with $f\colon X_1\to X_2$ a morphism of $E$-schemes and $g\colon \mf{X}_1\to\mf{X}_2$ a morphism of formal $\mc{O}_E$-schemes, such that $f^\an\circ j_1=j_2\circ g_\eta$.
\end{itemize}
\end{defn}

If $\ms{X}$ is a separated locally of finite type flat $\mc{O}_E$-scheme, one may functorially associate a gluing triple $(\ms{X}_E,\widehat{\ms{X}},j_{\ms{X}})$. This in particular gives a functor 
\begin{equation}\label{eq:triple-embedding}
    \mathsf{t}\colon \left\{\begin{matrix}\text{Locally of finite type}\\ \text{separated flat }\mathcal{O}_E\text{-schemes}\end{matrix}\right\}\to\cat{Tri}(\mc{O}_E),\qquad \ms{X}\mapsto \mathsf{t}(\ms{X})=(\ms{X}_E,\widehat{\ms{X}},j_\ms{X}).
\end{equation}
We use $j_\ms{X}$ to consider $\wh{\ms{X}}_\eta$ as an open adic subspace of $\ms{X}_E^\an$ without comment in the sequel.

\begin{prop}\label{prop:triple-fully-faithful} The functor $\mathsf{t}$ is fully faithful.
\end{prop}
\begin{proof} By a standard devissage, we reduce ourselves to the following. Suppose that $A$ is a finite type flat $\mc{O}_E$-algebra. Then, the preimage of $\wh{A}$ under $A_E\to \wh{A}_E$ is $A$. Indeed, writing an element of $A_E$ as $\tfrac{a}{\pi^k}$ with $a$ in $A$, by assumption $a$ maps into $\pi^k\wh{A}$ and thus $a$ lies in the preimage of $\pi^k\wh{A}$ which is the kernel of $A\to \wh{A}\to \wh{A}/\pi^k\wh{A}=A/\pi^k A$ and so $\pi^k A$. Thus $\tfrac{a}{\pi^k}$ is in $A$ as desired. 
\end{proof}

\begin{rem}\label{rem:alg-space-gluing} See \cite{AchingerYoucis} for a much stronger version of Proposition \ref{prop:triple-fully-faithful}.
\end{rem}

Now, fix a gluing triple $(X,\mf{X},j)$ over $\mc{O}_E$. For $\ms{S}$ any element of $\{\cat{Sht},\GSht,\GShtmu\}$ we denote by $\ms{S}(X,\mf{X},j)$ the category of triples $(s_X,s_{\mf{X}},\gamma)$ where $s_X$ and $s_\mf{X}$ are objects of $\ms{S}(X)$ and $\ms{S}(\mf{X})$ respectively, and $\gamma\colon j^\ast(s_X)\isomto s_{\mf{X}}|_{\mf{X}_\eta}$ is an isomorphism in $\ms{S}(\mf{X}_\eta)$, with the obvious notion of morphisms. Then, for a separated finite type $\mc{O}_E$-scheme $\ms{X}$ there are natural equivalences
\begin{equation}\label{eq:triples-shtuka}
    \mathsf{t}\colon \ms{S}(\ms{X})\to\ms{S}(\mathsf{t}(\ms{X}))
\end{equation}
functorial in $\ms{X}$. Therefore, we shall often not differentiate between $\ms{S}(\ms{X})$ and $\ms{S}(\mathsf{t}(\ms{X}))$ below.

\subsection{\texorpdfstring{$\mc{G}$}{G}-shtukas and local systems}\label{ss:shtukas-and-prismatic-torsors} We now recall and elaborate on the relationship between $\mc{G}$-shtukas and $\mc{G}$-objects in de Rham local systems as in \cite{PappasRapoportI}.

\subsubsection{Shtukas and \texorpdfstring{$\Bdr^+$-pairs}{Bdr-pairs}}\label{sss:BdR-pairs} We begin by elaborating on the relationship between shtukas on $\Bdr^+$-pairs. Throughout the following we fix $X$ to be a locally Noetherian adic space over $E$. 

\begin{thm}[{\cite[\S2.5.1--\S2.5.2]{PappasRapoportI}}]\label{thm:PR-BdR-pairs} There is an equivalence of categories
\begin{equation}\label{eq:DRT-equiv}
    \Phi_X\colon \GSht(X)\isomto \left\{(\bb{P},H)\colon \begin{aligned}(1) & \,\,\bb{P}\emph{ is an object of }\cat{Tors}_{\mc{G}(\Z_p)}(X),\\ (2) & \,\,\emph{a }\underline{\mc{G}(\Z_p)}\emph{-equivariant map }H\colon\bb{P}^\lozenge\to \mr{Gr}_{G,\mr{Spd}(E)}.\end{aligned}\right\}. 
\end{equation}
\end{thm}
Let us elaborate on the notation on the right-hand side of \eqref{eq:DRT-equiv}:
\begin{itemize} 
\item $\bb{P}^\lozenge=\varprojlim (\bb{P}/K_n)^\lozenge$ is the diamond associated to $\bb{P}$ as in \S\ref{ss:comparison-for-v-sheaves}, 
\item $\mr{Gr}_{G,\mr{Spd}(E)}$ is the $\B_\dR^+$-Grassmannian associated to $G$ as in \cite[\S III.3]{FarguesScholze}, 
\item and $H$ is a morphism of $v$-sheaves over $\mr{Spd}(E)$. 
\end{itemize}
The morphisms in the target category in \eqref{eq:DRT-equiv} are the obvious ones. 

\begin{rem}
    Although in \cite{PappasRapoportI}, the target category is defined with $\cat{Tors}_{\mc G(\Z_p)}(X^\lozenge)$ instead of $\cat{Tors}_{\mc G(\Z_p)}(X)$, these groupoids are equivalent (see Remark \ref{rem: torsors on the diamond associated to a locally noetherian adic space}). 
\end{rem}
\begin{rem}\label{rem: the lattice category fits into the Tannakian formalism}
    The target category of the functor $\Phi_X$ fits into the Tannakian framework in the following sense. Let $S^\sharp$ be an affinoid perfectoid space over $E$ and $\Bdr^+\hyphen\cat{Lat}_\mc G(S^\sharp)$ denote the target category. We let $\Bdr^+\hyphen\cat{Lat}^\mr{vect}(S^\sharp)$ denote the groupoid of triples $(\mbb L,M,\psi\colon \mbb L\otimes_{\underline{\Z_p}}\Bdr(S^\sharp)\isomto M\otimes_{\Bdr^+(S^\sharp)}\Bdr(S^\sharp))$ where $\mbb L$ is a finite free $\Z_p$-local system on $S^\sharp$ (or equivalently on the tilt $S$), $M$ is a finite projective $\Bdr^+(S^\sharp)$-module, and $\mbb L\otimes_{\underline{\Z_p}}\Bdr(S^\sharp)$ is the finite projective $\Bdr(S^\sharp)$-module constructed as in the beginning of \cite[\S22.3]{ScholzeBerkeley}. Then with the notation in Definition \ref{defn: G-C the groupoid of exact tensor functors}, there is an equivalence 
    \begin{equation*}
        \Bdr^+\hyphen\cat{Lat}_\mc G(S^\sharp)\isomto \mc G\hyphen\Bdr^+\hyphen\cat{Lat}^\mr{vect}(S^\sharp).
    \end{equation*}
    In fact, $\Bdr^+\hyphen\cat{Lat}_\mc G(S^\sharp)$ is identified with the groupoid of triples $(\mbb P,\mc Q,\psi\colon \mbb P\times^{\underline{\mc{G}(\Z_p)}}G_{\Bdr(S^\sharp)}\isomto\mc Q\otimes_{\Bdr^+(S^\sharp)}\Bdr(S^\sharp))$ where $\mbb P$ is an object of $\cat{Tors}_{\underline{\mc G(\Z_p)}}(S^\sharp)=\cat{Tors}_{\underline{\mc G(\Z_p)}}(S)$, $\mc Q$ is a $G$-torsor on $\Spec(\Bdr^+(S^\sharp))$, and $\mbb P\times^{\underline{\mc{G}(\Z_p)}}G_{\Bdr(S^\sharp)}$ denotes the $G$-torsor on $\Bdr(S^\sharp)$ constructed via \cite[Theorem 22.5.2]{ScholzeBerkeley} (cf.\ \cite[\S III.3]{FarguesScholze}). 
\end{rem}

To explain the definition of the functor $\Phi_X$, begin by fixing an object of $S=\Spa(R,R^+)$ of $\cat{Perf}_k$ as well as an untilt $S^\sharp=\Spa(R^\sharp,R^{\sharp+})$ over $E$ with $(\xi)\defeq \ker(\theta\colon W(R^+)\to R^{\sharp+})$. Also, for a closed subset $Z$ of a topological space $Y$ and a sheaf $\mc{F}$ on $Y$, we write $\Gamma^\dagger(Z,\mathcal{F})$ as shorthand for $\varinjlim_{Z\subseteq U}\mc{F}(U)$.

As in \cite[Definitions 4.2.2 and 5.1.1]{KeLiRpHF}, consider the \emph{integral Robba ring (over $S$)} 
\begin{equation*}
    \wt{\mc R}^\mr{int}_S\defeq\varinjlim_{r\to 0^+}\mc{O}(\mc{Y}_{[0,r]}(S))=\Gamma^\dagger(S,\mc{O}_{S\dot{\times}\Z_p}). 
\end{equation*} 
For this last object, we are considering $S$ as a closed Cartier divisor of $S\dot{\times}\Z_p$ in the usual way (see \cite[Proposition II.1.4]{FarguesScholze}). We observe that $\wt{\mc{R}}^\mr{int}_S$ carries a natural Frobenius morphism, compatible with that on $S\dot{\times}\Z_p$. As $S^\sharp$ does not intersect $S$ we see that we have a functor 
\begin{equation*}
    \cat{Sht}_n(S^\sharp)\to \cat{Mod}^\varphi(\wt{\mc{R}}^\mr{int}_S),\qquad (\mc{P},\varphi_\mc{P})\mapsto (\Gamma^\dagger(S,\mc{P}),\varphi_\mc{P}),
\end{equation*}
where the target is the category of \'etale $\varphi$-modules over $\wt{\mc{R}}^\mr{int}_S$ (see \cite[Definition 6.1.1]{KeLiRpHF}).

Denote by $\cat{Sht}_{n,\mr{free}}(S^\sharp)$ the full subgroupoid of $\cat{Sht}_n(S^\sharp)$ consisting of shtukas $(\mc{P},\varphi_\mc{P})$ with the property that the associated object of $\cat{Mod}^\varphi(\wt{\mc{R}}^\mr{int}_S)$ is free (which is called `trivial' in loc.\@ cit.\@). On the other hand, let ${\B}_\dR^+\text{-}\cat{Pair}_n(S^\sharp)$ be the groupoid of pairs  $(T,\Xi)$, where $T$ a finite free $\Z_p$-module of rank $n$ and, $\Xi$ is a ${\B}_\dR^+(S^\sharp)$-lattice of $T\otimes_{\Z_p}{\B}_\dR(S^\sharp)$, with the obvious notion of (iso)morphisms. 

There is a natural morphism of groupoids
\begin{equation*}
    \Phi_{S^\sharp}\colon \cat{Sht}_\mr{n,free}(S^\sharp)\to {\B}_\dR^+\text{-}\cat{Pair}_n(S^\sharp), \qquad (\mc{P},\varphi_\mc{P})\mapsto (T_\mc{P},\Xi_\mc{P}). 
\end{equation*}
Here $T_\mc{P}\defeq \Gamma^\dagger(S,\mc{P})^{\varphi=1}$, which is a free $\Z_p$-module of rank $n$. Let $\mc P_{\Bdr^+(S^\sharp)}$ denote $\Gamma(\mc Y_{[0,1]},\mc P)^\wedge_{\xi}$, and set $\mc P_{\Bdr(S^\sharp)}\defeq \mc P_{\Bdr^+(S^\sharp)}[\nicefrac{1}{\xi}]$ . Then there is a natural isomorphism 
\begin{equation}\label{eq:phi-module-isom}
    T_\mc{P}\otimes_{\Z_p}{\B}^+_\dR(S^\sharp)\isomto \mc{P}_{\Bdr^+(S^\sharp)}
    \,\,(=\Gamma^\dagger(S^\sharp,\mc{P})^\wedge_{\xi}),
\end{equation}
obtained by extending an element of $T_\mc{P}$, which is a priori an element of $\Gamma(\mc Y_{[0,r]}(S),\mc P)$ for some $r>0$, to $\Gamma(\mc Y_{[0,1]},\mc P)$ via $\varphi_\mc P$ (cf.\@ \cite[Corollary 12.4.1]{ScholzeBerkeley}). Set $\Xi_\mc{P}$ to be the ${\B}_\dR^+(S^\sharp)$-lattice in $T_\mc{P}\otimes_{\Z_p}{\B}_\dR(S^\sharp)$ corresponding under \eqref{eq:phi-module-isom} to $\varphi_\mc P((\phi^*\mc P)_{\Bdr^+(S^\sharp)})$ where
\begin{equation*}
\varphi_\mc P\colon (\phi^*\mc P)_{\Bdr(S^\sharp)}\to \mc P_{\Bdr(S^\sharp)},
\end{equation*} 
is the map induced by the Frobenius structure of $\mc{P}$.

If an object $(\mc{P},\varphi_\mc{P})$ of $\cat{Sht}_{n,\mr{free}}(S^\sharp)$ is obtained by restriction from an object $(M,\varphi_M)$ of $\cat{Vect}^\varphi(W(R^+),(\xi))$ then we can describe $\Phi_{S^\sharp}(\mc{P},\varphi_\mc{P})$ more concretely. Namely, we have the equality $T_\mc{P}=(M\otimes_{W(R^+)}\wt{\mc{R}}_S^\mr{int})^{\varphi=1}$, and $\Xi_\mc{P}$ corresponds under \eqref{eq:phi-module-isom} to the image of the $\B_\dR^+(S^\sharp)$-lattice $\phi^\ast M\otimes_{W(R^+)}{\B}_\dR^+(S^\sharp)$ under
\begin{equation*}
    \varphi_M\otimes 1\colon \phi^\ast M\otimes_{W(R^+)}{\B}_\dR(S^\sharp)\to M\otimes_{W(R^+)}{\B}_\dR(S^\sharp).
\end{equation*}
In any case, we have the following result, which can be proven exactly in the same way as \cite[Proposition 12.4.6]{ScholzeBerkeley}, as mentioned in the proof of \cite[Proposition 2.5.1]{PappasRapoportI}. 
\begin{lem}[{cf.\@ \cite[Proposition 12.4.6]{ScholzeBerkeley} and  \cite[Proposition 2.5.1]{PappasRapoportI} }]\label{lem: Fargues equivalence for shtukas} The functor 
    \begin{equation*}
    \Phi_{S^\sharp}\colon \cat{Sht}_{n,\mr{free}}(S^\sharp)\to {\B}_\dR^+\text{-}\cat{Pair}_n(S^\sharp), \qquad (\mc{P},\varphi_\mc{P})\mapsto (T_\mc{P},\Xi_\mc{P}) 
\end{equation*}
is an equivalence. 
\end{lem}

When $\mc{G}=\GL_{n,\Z_p}$, the equivalence $\Phi_X$ is then obtained from the observations that 
\begin{itemize} 
\item[(a)] $\Phi_{S^\sharp}$ defines a functor on $\cat{Perf}_X$,
\item[(b)] the source and target are the global sections of the stackification of $\cat{Sht}_{n,\mr{free}}$ and $\B_\dR^+\hyphen\cat{Pair}_n$ for the pro-\'etale topology, respectively, where here we consider these objects as prestacks on $\cat{Perf}_{X}$.\footnote{The fact that every object of $\cat{Sht}_n(S^\sharp)$ is pro-\'etale locally in $\cat{Sht}_{n,\text{free}}(S^\sharp)$ follows from \cite[Theorem 8.5.3]{KeLiRpHF} as any local system can be trivialized pro-\'etale locally.} 
\end{itemize}
The case for general $\mc{G}$ is then obtained by applying the Tannakian formalism (see Remarks \ref{rem: G-Sht fits into the Tannakian formalism} and \ref{rem: the lattice category fits into the Tannakian formalism}). 

Let $\Gr_{G,\bm{\mu},\Spd(E)}$ and $\Gr_{G,\leqslant\bm{\mu},\Spd(E)}$ be as in \cite[Definition 19.2.2]{ScholzeBerkeley}. Observe that if $\bm{\mu}$ is minuscule, then $\Gr_{G,\bm{\mu},\Spd(E)}=\Gr_{G,\leqslant\bm{\mu},\Spd(E)}$. The following lemma is just an unraveling of the definitions.

\begin{lem}\label{lem:boundedness-factorization-equiv} If $(\ms{P},\varphi_\ms{P})$ is an object of $\GSht(X)$, and $\Phi_X(\ms{P},\varphi_\ms{P})=(\bb{P},H)$, then $(\ms{P},\varphi_\ms{P})$ is an object of $\GSht_{\bm{\mu}^{-1}}(X)$ if and only if $H$ factorizes through $\Gr_{G,\leqslant \bm{\mu},\Spd(E)}$.
\end{lem}
\begin{proof} Let us begin by observing that as $\Gr_{G,\leqslant\bm{\mu},\Spd(E)}$ is a closed subdiamond of $\Gr_{G,\Spd(E)}$ (see \cite[Proposition 19.2.3]{ScholzeBerkeley}). Thus, $H$ factorizes through $\Gr_{G,\leqslant\bm{\mu},\Spd(E)}$ if and only if it does so at the level of points. Moreover, as $\Gr_{G,\leqslant\bm{\mu},\Spd(E)}$ is closed in $\Gr_{G,\Spd(E)}$ it is partially proper (see \cite[Lemma 19.1.4]{ScholzeBerkeley}), and so we further see that $H$ factorizes through $\Gr_{G,\leqslant \bm{\mu},\Spd(E)}$ if and only if it does so for points of the form $S^\sharp=\Spa(C^\sharp,C^{\sharp\circ})$ with $C$ an algebraically closed perfectoid field. Tracing through the definitions, this means that for trivializations $G_{\B_\dR^+(C^{\sharp+})}\isomto \ms P_{\B_\dR^+(C^{\sharp+})}$ and $G_{\B_\dR^+(C^{\sharp+})}\isomto \phi^*\ms P_{\B_\dR^+(C^{\sharp+})}$, that the relative position of $\ms{P}$ and $\varphi_\ms{P}(\phi^\ast\ms{P})$ defines an element of $\Gr_{G,\leqslant\bm{\mu},\Spd(E)}(C^\sharp,C^{\sharp\circ})$. On the other hand, by definition, $(\ms{P},\varphi_{\ms{P}})$ lies in $\GSht_{\bm{\mu}^{-1}}(X)$ if and only if for all such $S^\sharp$ and all such trivializations, the relative position of $\varphi_\ms{P}(\phi^\ast\mc{P})$ and $\ms{P}$  defines an element of $\Gr_{G,\leqslant\bm{\mu}^{-1},\Spd(E)}(C^\sharp,C^{\sharp\circ})$. These are clearly equivalent. 
\end{proof}

\subsubsection{Shtukas associated to de Rham local systems}\label{sss:de-Rham-shtukas} Let $X$ be a smooth rigid $E$-space. In \cite[\S2.6.1--\S2.6.2]{PappasRapoportI} there is constructed a functor 
\begin{equation*}
    U_\sht\colon \GLoc_{\Z_p}^{\dR}(X)\to \GSht(X),
\end{equation*}
which we now recall. 

Thanks to Theorem \ref{thm:PR-BdR-pairs}, it suffices to construct a functor
\begin{equation*}
    \Phi_X\circ U_\sht\colon \GLoc_{\Z_p}^\dR(X)\to \left\{(\bb{P},H)\colon \begin{aligned}(1) & \,\,\bb{P}\text{ is an object of }\cat{Tors}_{\mc{G}(\Z_p)}(X),\\ (2) & \,\,\text{a }\underline{\mc{G}(\Z_p)}\text{-equivariant map }H\colon\bb{P}^\lozenge\to \mr{Gr}_{G,\mr{Spd}(E)}.\end{aligned}\right\}. 
\end{equation*}
Further, by employing the Tannakian formalism (see Remark \ref{rem: the lattice category fits into the Tannakian formalism}), we may further assume that $\mc{G}=\GL_{n,\Z_p}$. In this case, an object of $\GLoc_{\Z_p}^\dR(X)$ is nothing but a rank $n$ de Rham local system $\bb{L}$ on $X$. Then one defines $\Phi_X(U_\sht(\bb{L}))$ to be $(\bb{P}_\bb{L},H_\bb{L})$. Here, $\bb{P}_\bb{L}=\underline{\Isom}(\underline{\Z}_p^n,\bb{L})$ is the $\underline{\GL_n(\Z_p)}$-torsor as in \S\ref{ss:G(Z_p)-local-systems-on-adic-spaces}. Recall, as in \S\ref{ss:de-Rham-local-systems}, there is a canonical $\B_\dR^+$-equivariant isomorphism
\begin{equation*}
    c_\mr{Sch}\colon \bb{L}\otimes_{\underline{\Z}_p}\B_\dR^+\isomto \Fil^0(D_\dR(\bb{L})\otimes_{\mc{O}_X}\mc{O}\B_\dR)^{\nabla=0}.
\end{equation*}
This induces a canonical $\B_\dR$-equivariant isomorphism 
\begin{equation*}
    c_\mr{Sch}\otimes 1\colon \bb{L}\otimes_{\underline{\Z}_p}\B_\dR\isomto (D_\dR(\bb{L})\otimes_{\mc{O}_X}\mc{O}\B_\dR)^{\nabla=0}.
\end{equation*}
So, from an isomorphism $a\colon \underline{\Z}_p^n\isomto \bb{L}_{S^\sharp}$, where $S^\sharp$ is an untilt over $X$ of some $S$ in $\cat{Perf}_k$, we obtain an isomorphism of $\Bdr(S^\sharp)$-modules
\begin{equation*}
    (c_\mr{Sch}\otimes 1)\circ a\colon \underline{\Z_p}^n(S^\sharp)\otimes_{\underline{\Z_p}(S^\sharp)}\B_\dR(S^\sharp)\isomto (D_\dR(\bb{L})(S^\sharp)\otimes_{S^\sharp}\mc{O}\B_\dR(S^\sharp))^{\nabla=0}.
\end{equation*}
We then set 
\begin{equation*}
    H_\bb{L}(a)\defeq ((c_\mr{Sch}\otimes 1)\circ a)^{-1}((D_\dR(\bb{L})(S^\sharp)\otimes_{S^\sharp}\mc{O}\B_\dR^+(S^\sharp))^{\nabla=0}),
\end{equation*}
a $\Bdr^+(S^\sharp)$-lattice in $\underline{\Z_p}^n(S^\sharp)\otimes_{\underline{\Z_p}(S^\sharp)}\Bdr(S^\sharp)$. 

\begin{prop}\label{prop:mu-equiv-factorization} Let $\omega$ be an object of $\GLoc_{\Z_p}^\dR(X)$ and write $\Phi_X(U_\sht(\omega))=(\bb{P},H)$. Then, we have that $\omega$ belongs to $\GLoc^{\dR}_{\Z_p,\bm{\mu}}$ 
if and only if $H$ factorizes through $\Gr_{G,\bm{\mu},\Spd(E)}$.
\end{prop}
\begin{proof}
    Observe that $\Gr_{G,\bm{\mu},\Spd(E)}$ is an open subdiamond of a closed subdiamond of $\Gr_{G,\Spd(E)}$ (see \cite[Proposition 19.2.3]{ScholzeBerkeley}). So, if $\Phi_X(U_\sht(\omega))=(\bb{P},H)$, then $H$ factorizes through $\Gr_{G,\bm{\mu},\Spd(E)}$ if and only if $|H|$ factorizes through $|\Gr_{G,\bm{\mu},\Spd(E)}|$ (see \cite[Proposition 12.15]{ScholzeECD}). But, as $|\bb{P}^\lozenge|=|\bb{P}|$ (see \cite[Lemma 15.6]{ScholzeECD}), and the projection map $|\bb{P}|\to |X|$ is open (see \cite[Lemma 3.10 (iv)]{ScholzepadicHT}), so that the preimage of a dense set is dense, we see that $H$ factorizes through $\Gr_{G,\bm{\mu},\Spd(E)}$ if and only if it does so after restriction to each classical point $x$. Thus, we may assume that $X=\Spa(E)$. By the Tannakian formalism, we are reduced to the case when $\mc{G}=\GL_{n,\Z_p}$, and we identify $\omega$ with its value $\bb{L}$ at the tautological representation. Set $T$ to be $\mbb L(C,\mc{O}_C)$ with its natural Galois action. By definition, $H$ factorizes through $\Gr_{\GL_n,\bm{\mu},\Spd(E)}(C,\mc{O}_C)$ if and only if we can find trivializations $\B_\dR^+(C)^n\isomto T\otimes_{\Z_p}\B_\dR^+(C)$ and 
    \begin{equation*}
        (\B_\dR^+(C))^n\isomto D_\dR(T)
        \otimes_{E}\B_\dR^+(C)
    \end{equation*} 
    such that through the filtered isomorphism 
    \begin{equation*}
        c_\mr{Sch}\colon D_\dR(T)
        \otimes_{E}\Bdr(C)\isomto T\otimes_{\Z_p}\Bdr(C)
    \end{equation*} 
    the induced automorphism of $\B_\dR(C)^n$ given by the multiplication of an element of the double coset $\GL_n(\B_\dR^+(C))\bm{\mu}(\xi)\GL_n(\B_\dR^+(C))$. Choose an element $\mu$ of $\bm{\mu}$, and write $\mu(\xi)=(\xi^{r_1},\ldots,\xi^{r_n})$. Then this condition holds if and only if there exists a $\Bdr^+(C)$-basis $(e_\nu)_{\nu=1}^n$ of $D_\dR(T)\otimes_{E}\Bdr^+(C)$ such that the filtration $\Fil^r$ on $D_\dR(T)\otimes_{E}\Bdr(C)$ is given by $\Fil^r=\sum_{\nu=1}^n\xi^{r-r_\nu}\Bdr^+(C)\cdot e_\nu$. This is seen to be equivalent $\mbb L$ belonging to $\cat{Loc}_{\Z_p,\bm{\mu}}^{\dR}(\Spa(E))$. 
\end{proof}


\subsection{Analytic prismatic \texorpdfstring{$F$-crystals}{F-crystals} and shtukas} We finally come to the shtuka realization functor from analytic prismatic $F$-crystals on a smooth formal $\mc{O}_E$-scheme and the comparison isomorphism to the functor from \S\ref{sss:de-Rham-shtukas}, which will allow us to extend this realization functor to the scheme-theoretic setting.

We continue to fix notation as in Notation \ref{nota:shtuka}. 

\subsubsection{Prismatic \texorpdfstring{$\mc{G}$-torsors}{G-torsors} with \texorpdfstring{$F$-structure}{F-structure} bounded by \texorpdfstring{$\mu$}{mu}} We would first like to define a notion of prismatic $\mc{G}$-torsors with $F$-structure bounded by a cocharacter $\mu$ that will match the category $\GShtmu$ under the shtuka realization functor constructed below.

To this end, temporarily fix the following data:
\begin{itemize}
    \item $k$ is a perfect field of characteristic $p$,
    \item $W\defeq W(k)$,
    \item $\mf{Y}$ is a quasi-syntomic $p$-adic formal $W$-scheme,
    \item $\mu\colon \bb{G}_{m,W}\to \mc{G}_{W}$ is a minuscule cocharacter.
\end{itemize} 
Let $(A,I)$ be an object of $\mf{Y}_\Prism$. Define $\cat{Tors}^{\varphi,\mu}_{\mc{G}}(A,I)$ to be the full subcategory of $\cat{Tors}_{\mc{G}}^\varphi(A,I)$ consisting of $(\mc{A},\varphi_\mc{A})$ such that there exists a $(p,I)$-adic faithfully flat cover $A\to A'$ such that $IA'$ is principal, and there exists a trivialization $\mc{A}\isomto \mc{G}$ after restriction to the slice site over $(A',IA')$ such that under this trivialization $\varphi_\mc{A}$ corresponds to left multiplication by an element of $\mc{G}(A')\mu(d)\mc{G}(A')$ for some (equiv.\@ for any) generator $d$ of $IA'$. Set
\begin{equation*}
    \cat{Tors}_{\mc{G}}^{\varphi,\mu}(\mf{Y}_\Prism)\defeq \twolim_{(A,I)\in\mf{Y}_\Prism}\cat{Tors}_{\mc{G}}^{\varphi,\mu}(A,I).
\end{equation*}
We give the objects of the category the following name.

\begin{defn}\label{defn:bounded-by-mu}
    An object of $\cat{Tors}_{\mc{G}}^{\varphi,\mu}(\mf{Y}_\Prism)$ is called a \emph{prismatic $\mc{G}$-torsor with $F$-structure bounded by $\mu$} on $\mf{Y}$.
\end{defn}

We now return to the notation as in Notation \ref{nota:shtuka}, but assume that $\mc{O}_E$ is absolutely unramified. Suppose further that $\mf{X}$ is a smooth formal $\mc{O}_E$-scheme, with generic fiber $X$, and let $\mu\colon \bb{G}_{m,\mc{O}_E}\to\mc{G}_{\mc{O}_E}$ be a minuscule cocharacter with $\mu_{\ov{E}}$ an element of our previously fixed conjugacy class $\bm{\mu}$. Let us define $\cat{Tors}^\varphi_{\mc{G},\bm{\mu}}(\mf{X}_\Prism)$ to be the full subcategory of $\cat{Tors}^\varphi_\mc{G}(\mf{X}_\Prism)$ consisting of those $(\mc{A},\varphi_\mc{A})$ such that $T_\et(\mc{A},\varphi_\mc{A})$ corresponds to an element of $\GLoc^{\crys}_{\Z_p,\bm{\mu}}(X)$. To understand the relationship between $\cat{Tors}^{\varphi}_{\mc{G},\bm{\mu}}(\mf{X}_\Prism)$ and $\cat{Tors}_{\mc{G}}^{\varphi,\mu^{-1}}(\mf{X}_\Prism)$, we make the following observation.

\begin{prop}\label{prop:matching-of-bounded-by-mu-conditions}
     Let $\{\Spf(R_i)\}$ be an open  cover $\mf{X}$ with each $R_i$ small. Then, an object $(\mc{A},\varphi_\mc{A})$ of $\cat{Tors}^\varphi_\mc{G}(\mf{X}_\Prism)$ lies in $\cat{Tors}^\varphi_{\mc{G},\bm{\mu}}(\mf{X}_\Prism)$ if and only if the following condition holds: for each $i$ there exists an \'etale cover $(\wt{R}_i[\nicefrac{1}{p}],\wt{R}_i)\to (S^\sharp,S^{\sharp+})$ in $\cat{Perf}_E$ such that $\varphi_\mc{A}$ on $\mc{A}(\Ainf(S^{\sharp+}),\xi_{S^{\sharp+}})$ lies in $G(\B_\dR^+(S^\sharp))\mu^{-1}(\xi_{S^{\sharp+}})G(\B_\dR^+(S^\sharp))$. 
\end{prop}
\begin{proof} Observe that by Lemma \ref{lem:Bhatt-flatness} and Lemma \ref{lem:quasi-syn-cover-v-sheaf-cover}, $\{\Spa(\wt{R}_i[\nicefrac{1}{p}],\wt{R}_i)\to X^\lozenge\}$ (see \S\ref{sss: shtukas} for notation) is a $v$-cover, where $(\wt{R}_i[\nicefrac{1}{p}],\wt{R}_i)$ is a perfectoid Huber pair over $E$ (cf.\@ \cite[Lemma 3.21]{BMSI}). Note that $H\colon \bb{P}\to\Gr_{G,\Spd(E)}$, where $(\bb{P},H)=\Phi_X(T_\mr{sht}(\mc{A},\varphi_\mc{A}))$, factorizing through $\Gr_{G,\bm{\mu},\Spd(E)}$ can be checked on a $v$-cover. The claim is then clear since as $\bm{\mu}$ is minuscule, one has that $\Gr_{G,\bm{\mu}^{-1},\Spd(E)}$ is the \'etale sheafification of
\begin{equation*}
    (S^\sharp,S^{\sharp+})\mapsto G(\Bdr^+(S^\sharp))\mu^{-1}(\xi_{S^{\sharp+}})G(\Bdr^+(S^\sharp))\subseteq G(\Bdr(S^\sharp))/G(\Bdr^+(S^\sharp)),
\end{equation*}
a presheaf on $\cat{Perf}_E$.
\end{proof}

\begin{cor}\label{cor:containment-two-torsors} There is a containment $\cat{Tors}_{\mc{G}}^{\varphi,\mu^{-1}}(\mf{X}_\Prism)\subseteq \cat{Tors}_{\mc{G},\bm{\mu}}^\varphi(\mf{X}_\Prism)$ and so, in particular, for $(\mc{E},\varphi_\mc{E})$ in $\cat{Tors}_{\mc{G}}^{\varphi,\mu^{-1}}(\mf{X}_\Prism)$ one has that $T_\et(\mc{E},\varphi_\mc{E})$ corresponds to an element of $\GLoc^{\crys}_{\Z_p,\bm{\mu}}(X)$.
\end{cor}

\begin{rem} The assumption that $\mathcal{O}_E$ is absolutely unramified is not conceptually necessary, but without it one would be required to work with $\mc{O}_E$-prisms as in \cite[\S2]{Ito1}. 
\end{rem}

\subsubsection{The shtuka realization functor}
Let $\mf{X}$ be a smooth formal $\mc{O}_E$-scheme. As in \cite[Definition 4.6]{GleasonSpecialization}, we call a morphism $f\colon \Spa(R^\sharp,R^{\sharp+})\to \mf{X}^\mr{ad}$, for $(\Spa(R^\sharp,R^{+\sharp}),\alpha)$ in $\cat{Perf}_{\mc{O}_E}$, \emph{formalizing} if $f=g^\ad\circ i_R$ where $g$ is some morphism $\Spf(R^{\sharp+})\to \mf{X}$, and $i_R\colon \Spa(R^{\sharp},R^{\sharp ^+})\to \Spa(R^{\sharp+})$ is the canonical map. By \cite[Proposition 4.17]{GleasonSpecialization}, such a $g$ is unique if it exists. We call $g$ the \emph{formalization} of $f$. Set $\cat{Perf}^+_\mf{X}$ denotes the full subsite of $\cat{Perf}_\mf{X}$ consisting of formalizing morphisms $f\colon \Spa(R^\sharp,R^{\sharp+})\to \mf{X}^\mr{ad}$. Evidently $\cat{Perf}^+_{\mf{X}}$ a basis of $\cat{Perf}_\mf{X}$ with the analytic (i.e.\@, topological open cover) topology, as can be seen from taking an affine open cover $\mf{X}$.  

For an object $(S,S^\sharp,f)$ of $\cat{Perf}^+_\mf{X}$, with $S^\sharp=\Spa(R^\sharp,R^{\sharp+})$ and $g$ a formalization of $f$, observe that the triple $(\Ainf(R^{\sharp+}),(\xi),g)$ defines an object of $\mf{X}_\Prism$. So, $\mc{A}(\Ainf(R^{\sharp+}),({\xi}),g)$ defines an object of $\cat{Tors}^{\an,\varphi}_\mc{G}(\Ainf(R^{\sharp+}),(\xi))$. Pulling this back along the map of locally ringed spaces $\mc{Y}_{[0,\infty)}(S)\to\Spec(W(R^+))-V(p,\xi)$, gives a $\mc{G}$-shtuka $\mc{A}_\mathrm{sht}(S^\sharp,f)$ over $S$ with one leg at $S^\sharp$ (cf.\@ \cite[\S3.1]{Daniels}). It is clear that this construction defines a sheaf on $\cat{Perf}^+_\mf{X}$ for the analytic topology, and thus extends uniquely to give a section of $p_\mf{X}$, and thus an object of $\GSht(\mf{X})$.

\begin{construction} Let $\mf{X}$ be a smooth formal $\mc{O}_E$-scheme. The functor
\begin{equation*}
    T_\sht\colon \cat{Tors}_\mc{G}^{\an,\varphi}(\mf{X}_\Prism)\to \GSht(\mf{X}),\qquad (\mc{A},\varphi_\mc{A})\mapsto T_\sht(\mc{A},\varphi_{\mc{A}})=(\mc{A}_\sht,\varphi_{\mc{A}_\sht})
\end{equation*}
is called the \emph{shtuka realization functor}, and restricts to a functor
\begin{equation*}
     T_\sht\colon \cat{Tors}_{\mc{G},\bm{\mu}}^{\varphi}(\mf{X}_\Prism)\to \GSht_{\bm{\mu}}(\mf{X})
\end{equation*}
\end{construction}

\subsubsection{The comparison isomorphism} We now show that the shtuka realization functor $T_\mr{sht}$ intertwines the \'etale realization functor from \S\ref{ss: T et} and the functor $U_\mr{sht}$ from \S\ref{sss:de-Rham-shtukas}.

Let $\mf{X}$ be a smooth formal $\mc{O}_E$-scheme and let $X$ be its generic fiber. Recall (see \cite[Corollary 2.37]{GuoReinecke} and \cite[Propositions 3.21 and 3.22]{TanTong}) that there is a containment $\cat{Loc}_{\Z_p}^\crys(X)\subseteq \cat{Loc}_{\Z_p}^\dR(X)$. Thus, associated to an object $(\mc{A},\varphi_\mc{A})$ of $\cat{Tors}^{\an,\varphi}_\mc{G}(\mf{X}_\Prism)$, one may build two ostensibly different shtukas on $X$: the shtukas $T_\sht(\mc{A},\varphi_\mc{A})_\eta$ and $U_\sht(T_\et(\mc{A},\varphi_\mc{A}))$. The following comparison result says that these are canonically the same.

\begin{thm}\label{thm:prismatic-F-crystal-shtuka-relationship} For an object $(\mc{A},\varphi_\mc{A})$ of $\cat{Tors}_{\mc{G}}^{\an,\varphi}(\mf{X}_\Prism)$, there is a natural identification 
\begin{equation*}
    \varrho_{(\mc{A},\varphi_\mc{A})}\colon T_\sht(\mc{A},\varphi_{\mc{A}})_\eta\isomto U_\sht(T_\et(\mc{A},\varphi_\mc{A})),
\end{equation*}
functorial in $(\mc{A},\varphi_\mc{A})$.
\end{thm}
\begin{proof} By the Tannakian formalism (see Remark \ref{rem: G-Sht fits into the Tannakian formalism}), we are reduced to the case when $\mc G=\GL_{n,\Z_p}$, in which case we identify $(\mc{A},\varphi_\mc{A})$ with an object $(\mc{E},\varphi_\mc{E})$ of $\cat{Vect}^{\an,\varphi}(\mf{X}_\Prism)$ and write $\bb{L}=T_\et(\mc{E},\varphi_\mc{E})$. 
Consider an object $S=\Spa(R,R^+)$ of $\cat{Perf}_k$ and let $(S^\sharp,f)$ be an element of $\cat{Perf}^+_\mf{X}$ where $S^\sharp\to \Spa(\mc{O}_E)$ factorizes over $\Spa(E)$. Write $S^\sharp=\Spa(R^\sharp,R^{\sharp+})$.  We must show that the shtuka given by $(\mc E,\varphi_\mc{E})(\Ainf(R^{\sharp+}),({\xi}))|_{\mc{Y}_{[0,\infty)}(S)}$ and $U_\sht(\bb{L})(S^\sharp,f)$ are isomorphic functorially in $(S^\sharp,f)$.

By \cite[Lemma 4.10]{GuoReinecke}, it suffices to work only with $(S^\sharp,f)$ where $S^\sharp$ belongs to the category $X^w_\qrsp$ from \cite[Definition 4.9]{GuoReinecke}. 
In particular, we may assume that both of these shtukas over $S$ lie in $\cat{Sht}_{n,\mr{free}}(S^\sharp)$ and so by Lemma \ref{lem: Fargues equivalence for shtukas} it suffices to show that they have the same value under $\Phi_{S^{\sharp}}$. 
For both $(\mc E,\varphi_\mc{E})(\Ainf(R^{\sharp+}),({\xi}))|_{\mc{Y}_{[0,\infty)}(S)}$ and $U_\sht(\bb{L})(S^\sharp,f)$, the underlying free $\Z_p$-module is (by definition) $T=\bb{L}(S^\sharp)$. 
We denote by $(T,\Xi_{\mr{GR}})$ the $\B_\dR^+$-pair corresponding to the former, and by $(T,\Xi_{\dR})$ the one attached to the latter.

The former lattice $\Xi_\mr{GR}$ is described as follows. By the description of $\Phi_{S^\sharp}$ (before Lemma \ref{lem: Fargues equivalence for shtukas}), we have $\Xi_\mr{GR}=\varphi_\mc E((\phi^*\mc E)_{\Bdr^+(S^\sharp)})$. Observe that we have a commutative diagram of isomorphisms as in the proof of \cite[Theorem 4.8]{GuoReinecke} (see \cite[Remark 4.12]{GuoReinecke})
\begin{equation*}
\xymatrixrowsep{3pc}\xymatrixcolsep{5pc}\xymatrix{\phi^*\mc E(\Ainf(S),(\xi))_{\Bdr(S^\sharp)}\ar[d]_{(\text{cf. } \eqref{eq:GR-second-isom})} \ar[r]^{\varphi_\mc{E}} &\mc E(\Ainf(S),(\xi))_{\Bdr(S^\sharp)}\\ D_\crys(\mbb{L})(\Acrys(R^{\sharp+})\twoheadrightarrow R^{\sharp+})\otimes_{\Acrys(R^{\sharp+})}\Bdr(S^{\sharp})\ar[r]_-{c_\mr{Fal}^{-1}\otimes1} & T\otimes_{\Z_p}\Bdr(S^{\sharp}).\ar[u]_{\eqref{eq:phi-module-isom}}} 
\end{equation*}
Thus, we have 
 \begin{equation*}
     \Xi_\mr{GR}=(c_\mr{Fal}\otimes 1)^{-1}\left(D_\crys(\mbb{L})(\Acrys(R^{\sharp+})\twoheadrightarrow R^{\sharp+})\otimes_{\Acrys(R^{\sharp+})}\Bdr^+(S^{\sharp})\right).
 \end{equation*} 
 On the other hand, recall that the lattice $\Xi_\dR$ is given by 
 \begin{equation*}
     \Xi_\dR= (c_\mr{Sch}\otimes 1)^{-1}(D_\dR(\mbb L)(S^\sharp)\otimes_{S^\sharp}\mc O\Bdr^+(S^\sharp))^{\nabla=0}).
 \end{equation*} 
 Thus, considering the isomorphism 
\be
\theta_\dR^{+,\nabla}\colon D_\crys(\mbb L)(\Acrys(R^{\sharp+})\twoheadrightarrow R^{\sharp+})\otimes_{\Acrys(R^{\sharp+})}\Bdr^+(S^{\sharp})\isomto (D_\dR(\mbb L)(S^\sharp)\otimes_{S^\sharp}\mc O\Bdr^+(S^\sharp))^{\nabla=0}
\ee
from \eqref{eq: Faltings c and Scholze c}, the assertion follows from Lemma \ref{lem: Faltings c and Scholze c}. 
\end{proof}

 \subsubsection{The category of analytic prismatic \texorpdfstring{$F$-crystals}{F-crystals} over \texorpdfstring{$\mc{O}_E$-schemes}{OE-schemes} and shtuka realization} We now define the category of analytic prismatic $F$-crystals over $\mc{O}_E$-schemes, and use the comparison isomorphism in Theorem \ref{thm:prismatic-F-crystal-shtuka-relationship} to obtain a shtuka realization functor also on such objects.

Suppose that $\ms{X}$ is a separated locally of finite type flat $\mc{O}_E$-scheme with $\ms X_E$ smooth. 

\begin{defn}\label{defn:prismatic-model} An \emph{analytic prismatic $F$-crystal} on $\ms{X}$ is a triple $(\mc{F},(\mc{V},\varphi_\mc{V}),\iota)$ where
\begin{itemize}
    \item $\mc{F}$ is an object of $\cat{Loc}^{\dR}_{\Z_p}(\ms{X}_E^\an)$, 
    \item $(\mc{V},\varphi_\mc{V})$ is an object of $\cat{Vect}^{\mr{an},\varphi}(\wh{\ms{X}}_\Prism)$, 
\item and $\iota\colon T_\et(\mc{V},\varphi_\mc{V}) \isomto \mc{F}|_{\wh{\ms{X}}_\eta}$ is an isomorphism.
\end{itemize}
Call $(\mc{F},(\mc{V},\varphi_\mc{V}),\iota)$ a \emph{prismatic $F$-crystal} if $(\mc{V},\varphi_\mc{V})$ is one. 
\end{defn} 

A morphism of analytic prismatic $F$-crystals over $\ms{X}$
\begin{equation*}
    (f,g)\colon (\mc{F}',(\mc{V}',\varphi_{\mc{V}'}),\iota')\to (\mc{F},(\mc{V},\varphi_{\mc{V}}),\iota)
\end{equation*} 
consists of a morphism $f\colon \mc{F}'\to \mc{F}$ as well as a morphism $g\colon (\mc{V}',\varphi_{\mc{V}'})\to(\mc{V},\varphi_\mc{V})$ satisfying $\iota\circ f|_{\wh{\ms{X}}_\eta}=T_\et(g)\circ \iota'$. Denote the category of prismatic $F$-crystals (resp.\@ analytic prismatic $F$-crystals) on $\ms{X}$ by $\cat{Vect}^\varphi(\ms{X}_\Prism)$ (resp.\@ $\cat{Vect}^{\varphi,\an}(\ms{X}_\prism)$).\footnote{Note that, despite the notation, we are not claiming the existence of a site $\ms{X}_\Prism$.} It is evident that $\cat{Vect}^{\varphi,\an}(\ms{X}_\Prism)$ carries the structure of an exact $\Z_p$-linear $\otimes$-category, where exactness and the notion of tensor product are defined entry-by-entry.

The category $\GVect^{\varphi,\an}(\ms{X}_\Prism)$ of $\mc{G}$-objects of analytic prismatic $F$-crystals on $\ms{X}$ may be identified with the category of triples $(\omega_\et,\omega_\Prism,\iota)$ where 
\begin{itemize}
\item $\omega_\et$ is an object of $\GLoc^{\dR}_{\Z_p}(\ms{X}_E^\mr{an})$, 
\item $\omega_\Prism$ is an object of $\GVect^{\an,\varphi}(\wh{\ms{X}}_\Prism)$, 
\item and $\iota\colon T_\et\circ \omega_\Prism\isomto \omega_\et^\an|_{\wh{\ms{X}}_\eta}$ is an isomorphism,
\end{itemize}
and where morphisms are defined similarly to that of analytic prismatic $F$-crystals over $\ms{X}$. A similar statement holds for $\GVect^\varphi(\ms{X}_\Prism)$. We define $\GVect^{\varphi}_{\bm{\mu}}(\ms{X}_\Prism)$ to be the full subgroupoid of $\GVect^\varphi(\ms{X}_\Prism)$ consisting of those $(\omega_\eta,\omega_\prism,\iota)$ such that $\omega_\et$ and $\omega_\prism$ are both of type $\bm{\mu}$.

Using Theorem \ref{thm:prismatic-F-crystal-shtuka-relationship} we can define a shtuka realization functor also in this context.

\begin{construction} The association
\begin{equation*}
    (\omega_\et,\omega_\Prism,\iota)\mapsto (U_\sht(\omega_\et),T_\sht(\omega_\Prism),U_\sht (\iota) \circ \varrho_{\omega_\Prism})
\end{equation*}
defines a functor
\begin{equation*}
    T_\mr{sht}\colon \GVect^{\varphi,\an}(\ms{X}_\Prism)\to \GSht(\mathsf{t}(\ms{X}))
\end{equation*}
called the \emph{shtuka realization functor}, which restrictions to a functor
\begin{equation*}
    T_\sht\colon \GVect^\varphi_{\bm{\mu}}(\ms{X}_\Prism)\to \GShtmu(\ms{X}).
\end{equation*} 
\end{construction}


\appendix

\addtocontents{toc}{\protect\setcounter{tocdepth}{1}}%

\section{Tannakian formalism of torsors}\label{s:Tannakian-appendix}

In this appendix we collect some results concerning torsors and the Tannakian formalism.

\subsection{Basic definitions and results}\label{ss:basic-definitions-and-results} A topos $\ms{T}$ is the category of sheaves on a site (as in \stacks{03NH}), with the topology where $\{T_i\to T\}$ is a cover if it is a universal effective epimorphism (equiv.\@ $\bigsqcup T_i\to T$ is a surjection of sheaves). Denote the final object of $\ms{T}$ by $\ast$.

Fix $\mc{G}$ to be a group object of $\ms{T}$. An object $P$ of $\ms{T}$ equipped with a right action of $\mc{G}$ is a \emph{pseudo-torsor} for $\mc{G}$ if the following morphism is an isomorphism
\begin{equation*}
    P\times \mc{G}\to P\times P,\qquad (p,g)\mapsto (p,p\cdot g),
\end{equation*}
or, equivalently, $\mc{G}(S)$ acts simply transitively on $P(S)$ if the latter is non-empty. A pseudo-torsor $P$ is a \emph{torsor} if $P\to \ast$ is an epimorphism or, equivalently, $P$ is locally non-empty. Here, we say that an object $\mc{Q}$ of $\ms{T}$ is \emph{locally non-empty} if there exists a cover $\{U_i\to \ast\}$ with $\mc{Q}(U_i)$ non-empty for all $i$. A morphism of pseudo-torsors for $\mc{G}$ is a $\mc{G}$-equivariant morphism in $\ms{T}$, which is automatically an isomorphism if the source is a torsor. A torsor $P$ is trivial, if and only if $P(\ast)\ne\varnothing$. 

Denote the category of pseudo-torsors for $\mc{G}$ on $\ms{T}$ by $\cat{PseuTors}_\mc{G}(\ms{T})$, by $\cat{Tors}_\mc{G}(\ms{T})$ the full subcategory of torsors, and by $H^1(\ms{T},\mc{G})$ the set of isomorphism classes in $\cat{Tors}_\mc{G}(\ms{T})$. For an object $T$ of $\ms{T}$ with localized topos $\ms{T}/T$ (see \stacks{04GY}), for $\mc{G}_T\defeq \mc{G}|_{\ms{S}/T}$ we have 
\begin{equation*}
    \cat{Tors}_\mc{G}(\ms{T})\to \cat{Tors}_{\mc{G}_T}(\ms{T}/T),\qquad P\mapsto P_T\defeq P\times T, 
\end{equation*}
(we shorten the target to $\cat{Tors}_\mc{G}(\ms{T}/T)$). The association of $\cat{Tors}_\mc{G}(\ms{T}/T)$ to $T$ is a stack on $\ms{T}$.

For a $\mc{G}$-torsor $P$ and an object $Q$ of $\ms{T}$ with a left action of $\mc{G}$, we denote by $P\wedge^\mc{G}Q$ the \emph{contracted product} obtained as the quotient of $P\times Q$ by the $\mc{G}$-action $g\cdot(p,q)\defeq (pg^{-1},gq)$. For a morphism $f\colon \mc{G}\to\mc{H}$ of group objects, $\mc{H}$ inherits a left $\mc{G}$-action and we have a functor
\begin{equation*}
    f_\ast\colon \cat{Tors}_\mc{G}(\ms{T})\to\cat{Tors}_\mc{H}(\ms{T}),\qquad P\mapsto f_\ast(P)\defeq P\wedge^\mc{G}\mc{H},
\end{equation*}
where $\mc{H}$ acts on $f_\ast(P)$ in the obvious way (see \cite[Chapitre III, Proposition 1.4.6]{Giraud}).

Let $\mc{C}$ be a site and set $\ms{C}\defeq \cat{Shv}(\mc{C})$ to be its category of sheaves. For an object $X$ of $\mc{C}$ denote by $h_X$ the associated representable presheaf and by $h_X^\sharp$, or just $X$, its sheafification. We shall freely abuse the identification $\ms{T}\isomto \cat{Shv}(\ms{T})$ (cf.\@ \cite[Expos\'{e} IV, Corollaire 1.2.1]{SGA4-1}). 

\begin{lem}\label{lem:condition-for-covering-of-final-object} Let $\{X_i\}$ be a set of objects of $\mc{C}$ and $\mc{A}$ an object of $\ms{C}$. Then, a collection of elements $f_i$ of $\mc{A}(X_i)$ corresponds to a cover $\{h_{X_i}^\sharp\xrightarrow{f_i}\mc{A}\}$ if and only if for every object $X$ of $\mc{C}$ and element $f\in\mc{A}(X)$ there is a cover $\{U_j\xrightarrow{g_j} X\}$ so that for all $j$ there is a morphism $k_j\colon U_j\to X_i$ with $f_i\circ k_j=f\circ g_j$.
\end{lem}

\begin{lem}\label{lem:refine-cover-by-representable-cover} Let $X$ be an object of $\mc{C}$, and $\{\mc{A}_j\to h_X^\sharp\}$ a cover in $\ms{C}$. Then, there exists a cover $\{X_i\to X\}$ in $\mc{C}$ such that $\{h_{X_i}^\sharp\to h_X^\sharp\}$ refines $\{\mc{A}_j\to h_X^\sharp\}$.
\end{lem}

Set $\cat{PseuTors}_\mc{G}(\mc{C})$ to be $\cat{PseuTors}_\mc{G}(\ms{C})$, and define $\cat{Tors}_\mc{G}(\mc{C})$ and $H^1(\mc{C},\mc{G})$ similarly. By Lemma \ref{lem:condition-for-covering-of-final-object}, an object $\mc{A}$ of $\ms{C}$ with right $\mc{G}$-action belongs to $\cat{PseuTors}_\mc{G}(\mc{C})$ if and only if $\mc{G}(X)$ acts simply transtively on $\mc{A}(X)$ whenever $X$ is an object of $\mc{C}$ with $\mc{A}(X)\ne\varnothing$. By the following lemma an object $\mc{A}$ of $\cat{PseuTors}_\mc{G}(\mc{C})$ belongs to $\cat{Tors}_\mc{G}(\mc{C})$ if and only if for every object $X$ of $\mc{C}$, there exists a cover $\{X_i\to X\}$ in $\mc{C}$ with $\mc{A}(X_i)$ non-empty. 

\begin{lem}\label{lem:locally-nonempty-equiv} An object $\mc{A}$ of $\ms{C}$ is locally non-empty if and only if for all objects $X$ of $\mc{C}$ there exists a cover $\{X_i\to X\}$ in $\mc{C}$ with $\mc{A}(X_i)$ non-empty for all $i$.
\end{lem}

By \cite[Chapitre III, 1.7.3.3]{Giraud}, for any object $X$ of $\mc{C}$ there is a natural identification between $\cat{Tors}_\mc{G}(\ms{C}/h_X^\sharp)$ and $\cat{Tors}_\mc{G}(\mc{C}/X)$, with $\mc{C}/X$ as in \stacks{00XZ}. Thus, these objects are unambiguous in their definition, and so we use the latter notation in practice.

\subsection{Vector bundles and torsors}

Let $\mc{O}$ be a ring object of a topos $\ms{T}$. A \emph{vector bundle} on $(\ms{T},\mc{O})$ is an $\mc{O}$-module $\mc{E}$ for which there exists a cover $\{U_i\to \ast\}$ with $\mc{E}|_{U_i}$ isomorphic to $\mc{O}^{n_i}_{U_i}$ for some $n_i$.\footnote{For all ringed sites we consider this agrees with the notion of vector bundle defined in \cite[Notation 2.1]{BhattScholzeCrystals}.} Define $\cat{Vect}(\ms{T},\mc{O})$ to be the category of vector bundles on $(\ms{T},\mc{O})$, and $\cat{Vect}_n(\ms{T},\mc{O})$ the full subcategory where $n_i=n$ for all $i$. Let $\cat{Vect}^\mathrm{iso}_n(\ms{T},\mc{O})$ be the groupoid with the same objects as $\cat{Vect}_n(\ms{T},\mc{O})$ but with only the isomorphisms as morphisms. If $\ms{C}=\mathbf{Sh}(\mc{C})$ we use the notation $\cat{Vect}(\mc{C},\mc{O})$ and $\cat{Vect}_n(\mc{C},\mc{O})$, instead.

Define $\GL_{n,\mc{O}}$ to be the group object of $\ms{T}$ given by $\GL_{n,\mc{O}}(T)\defeq \Aut_{\mc{O}_T}(\mc{O}_T^n)$. Consider
\begin{equation*}
    \uIsom(\mc{O}^n,\mc{E})\colon \ms{T}\to \cat{Set},\qquad T\mapsto \Isom(\mc{O}^n_T,\mc{E}_T),
\end{equation*}
for $\mc{E}$ an object of $\cat{Vect}_n(\ms{T},\mc{O})$, which carries the natural structure of a $\GL_{n,\mc{O}}$-torsor. Conversely, for a  $\GL_{n,\mc{O}}$-torsor $P$ the contracted product $P\wedge^{\GL_{n,\mc{O}}}\mc{O}^n$, which inherits the structure of an $\mc{O}$-module from $\mc{O}^n$, is a vector bundle.

\begin{prop}[{cf.\@ \cite[Chapitre III, Th\'{e}or\`eme 2.5.1]{Giraud}}]\label{prop:gln-torsor-vector-bundle}
The functor
\begin{equation*}
    \cat{Vect}_n^\mathrm{iso}(\ms{T},\mc{O})\to \cat{Tors}_{\GL_{n,\mc{O}}}(\ms{T}),\qquad \mc{E}\mapsto \uIsom(\mc{E},\mc{O}^n)
\end{equation*}
is an equivalence with quasi-inverse given by 
\begin{equation*}
    \cat{Tors}_{\GL_{n,\mc{O}}}(\ms{T})\to \cat{Vect}^\mathrm{iso}_n(\ms{T},\mc{O}),\qquad P\mapsto P\wedge^{\GL_{n,\mc{O}}}\mc{O}^n.
\end{equation*}
\end{prop}

\subsection{Torsors and morphisms of topoi}
Let $\mc{C}$ (resp.\@ $\mc{D}$) be a site and set $\ms{C}$ (resp.\@ $\ms{D}$) to be the associated topos. Fix a morphism of topoi $(u_\ast,u^{-1})\colon \ms{C}\to \ms{D}$ (see \cite[Exposé IV, Definition 3.1]{SGA4-1} or \stacks{00XA}) and a group object $\mc{G}$ of $\ms{C}$. Observe that as $u_\ast$ is left exact it induces a morphism
\begin{equation*}
    u_\ast\colon \cat{PseuTors}_\mc{G}(\ms{C})\to \cat{PseuTors}_{u_\ast(\mc{G})}(\ms{D}).
\end{equation*}
On the other hand, if $\mc{H}$ is a group object of $\ms{D}$ then we similarly obtain a functor 
\begin{equation*}
    u^{-1}\colon \cat{PseuTors}_{\mc{H}}(\ms{D})\to\cat{PseuTors}_{u^{-1}(\mc{H})}(\ms{C}).
\end{equation*}
When $\mc{H}=u_\ast(\mc{G})$, the counit map $\epsilon: u^{-1}(u_\ast(\mc{G}))\to \mc{G}$ gives us a functor 
\begin{equation*}
    \epsilon_\ast\colon \cat{PseuTors}_{u^{-1}(u_\ast(\mc{G}))}(\ms{C})\to \cat{PseuTors}_{\mc{G}}(\ms{C}).
\end{equation*}
By composing these two functors we obtain a functor 
\begin{equation*}
u^\ast\defeq \epsilon_\ast \circ u^{-1}\colon \cat{PseuTors}_{u_\ast(\mc{G})}(\ms{D})\to \cat{PseuTors}_{\mc{G}}(\ms{C}).
\end{equation*}
We then obtain an adjoint pair $(u_\ast,u^\ast)\colon \cat{PseuTors}_\mc{G}(\ms{C}) \to\cat{PseuTors}_{u_\ast(\mc{G})}(\ms{D})$.

The following result follows quickly by applying the adjointness of $u^\ast$ and $u_\ast$.
\begin{prop}\label{prop:ominbus-torsor-comparison} Suppose that $u^{-1}(\mc{B})$ is locally non-empty for all objects $\mc{B}$ of $\cat{Tors}_{u_\ast(\mc{G})}(\ms{D})$. Then, $(u_\ast,u^\ast)\colon \cat{Tors}_\mc{G}(\ms{C})'\to \cat{Tors}_{u_\ast(\mc{G})}(\ms{D})$ is a pair of quasi-inverses, where $\cat{Tors}_\mc{G}(\ms{C})'$ is the full subcategory of $\cat{Tors}_\mc{G}(\ms{C})$ consisting of those $\mc{A}$ such that $u_\ast(\mc{A})$ is locally non-empty.
\end{prop}

If $u\colon \mathcal{D}\to \mathcal{C}$ is a continuous functor (see \cite[Expos\'e III, Definition 1.1]{SGA4-1} or \stacks{00WV}) then, by \stacks{00WU} we get an adjoint pair $(u_\ast,u^{-1})\colon \ms{C}\to\ms{D}$ where $u_\ast(\mc{A})$ is $\mc{A}\circ u$, and $u^{-1}(\mc{B})$ is the sheafification of 
\begin{equation*}
    (u^{-1})^\mathrm{pre}(\mathcal{B})(C)\defeq\colim_{\scriptscriptstyle(D,\psi)\in  \mc{I}_{\scriptscriptstyle C}^\mathrm{opp}}\mc{B}(D)
\end{equation*}
with $\mc{I}^\mathrm{opp}_C$ the category of pairs $(D,\psi)$ where $D$ is an object of $\mc{D}$ and $\psi\colon C\to u(D)$. If $u$ induces a morphism of sites (i.e.\@, that $u^{-1}$ is left exact), then $(u_\ast,u^{-1})$ is a morphism of topoi.

\begin{cor}[{cf.\@ \cite[Chapitre V, Proposition 3.1.1]{Giraud}}]\label{cor:torsor-equiv-continuous-morphism} If $u\colon \mc{D}\to\mc{C}$ induces a morphism of sites, then we obtain a pair of quasi-inverse functors $(u_\ast,u^\ast)\colon \cat{Tors}_\mc{G}(\ms{C})'\to \cat{Tors}_{u_\ast(\mc{G})}(\ms{D})$.
\end{cor}

If $u\colon \mc{C}\to\mc{D}$ is a cocontinuous functor (see \cite[Expos\'e III, \S2]{SGA4-1} or \stacks{00XI}), then by \stacks{00XN} $u$ induces a morphism of topoi $(u_\ast,u^{-1})\colon \ms{C}\to\ms{D}$. Here, $u^{-1}(\mc{B})$ is the sheafification of $\mc{B}\circ u$, and 
\begin{equation*}
    u_\ast(\mc{A})(D)=\lim_{(C,\psi)\in {}_D \mc{I}^\mathrm{opp}}\mc{A}(C)
\end{equation*}
where ${}_D \mc{I}^\mathrm{opp}$ is the the category of pairs $(C,\psi)$ where $C$ is an object of $\mc{C}$ and $\psi\colon u(C)\to D$. Combining Proposition \ref{prop:ominbus-torsor-comparison} and Lemma \ref{lem:cocont-inverse-image-locally-non-empty} below we obtain the following corollary.

\begin{cor}\label{cor:torsor-equiv-cocontinuous-morphism} Let $u\colon \mc{C}\to\mc{D}$ be a cocontinuous functor. Then, for any group object $\mc{G}$ of $\ms{C}$ we obtain a pair of quasi-inverse functors $(u_\ast,u^\ast)\colon \cat{Tors}_\mc{G}(\ms{C})'\to \cat{Tors}_{u_\ast(\mc{G})}(\ms{D})$.
\end{cor}

\begin{lem}\label{lem:cocont-inverse-image-locally-non-empty} Let $u\colon \mc{C}\to\mc{D}$ be a cocontinuous functor. Then, for any locally non-empty object $\mc{B}$ of $\ms{D}$, the object $u^{-1}(\mc{B})$ of $\ms{C}$ is locally non-empty. 
\end{lem}
\begin{proof} Take a cover $\{\mc{B}_i\to \ast\}$ in $\ms{D}$ with $\Hom(\mc{B}_i,\mc{B})\ne\varnothing$ for all $i$, and an arbitrary cover $\{h^\sharp_{Z_\gamma}\to *\}$ in $\ms{C}$. By Lemma \ref{lem:refine-cover-by-representable-cover}, for each $\gamma$ there exists a cover $\{D_{\beta\gamma}\to u(Z_\gamma)\}$ in $\mc{D}$ such that $\{h_{D_{\beta\gamma}}^\sharp\to h_{u(Z_\gamma)}^\sharp\}$ refines $\{\mc{B}_i\times h_{u(Z_\gamma)}^\sharp\to h_{u(Z_\gamma)}^\sharp\}$. By cocontinuity, for each $\gamma$ there exists a cover $\{X_{\alpha\gamma}\to Z_\gamma\}$ in $\mc{C}$ such that $\{u(X_{\alpha\gamma})\to u(Z_\gamma)\}$ refines $\{D_{\beta\gamma}\to u(Z_\gamma)\}$, and therefore $\mc B(u(X_{\alpha\gamma}))$ is non-empty, and there is a map $\mc B(u(X_{\alpha\gamma}))\to u^{-1}(\mc B)(X_{\alpha\gamma})$, so $u^{-1}(\mc{B})(X_{\alpha\gamma})$ is non-empty. As $\{h_{X_{\alpha\gamma}}^\sharp\to \ast\}$ is a cover in $\ms{C}$, the claim follows.
\end{proof}

Let $(u_\ast,u^{-1})\colon \ms{C}\to\ms{D}$ be a morphism of topoi defined by a (co)continuous functor $u$. For a ring object $\mc{O}$ of $\ms{C}$ one has an identification $u_\ast(\GL_{n,\mc{O}})\isomto \GL_{n,u_\ast(\mc{O})}$. Define $\cat{Vect}(\ms{C},\mc{O})'$ to be the full subcategory of $\cat{Vect}(\ms{C},\mc{O})$ of vector bundles $\mc{E}$ with $u_\ast(\mc{E})$ a vector bundle over $u_\ast(\mc{O})$. Moreover, define the functor
\begin{equation*}
    u^\ast\colon \cat{Vect}(\ms{D},u_\ast(\mc{O}))\to \cat{Vect}(\ms{C},\mc{O})',\quad u^\ast(\mc{W})\defeq u^{-1}(\mc{W})\otimes_{u^{-1}(u_\ast(\mc{O}))}\mc{O}
\end{equation*}
which is compatible under Proposition \ref{prop:gln-torsor-vector-bundle} with $u^\ast\colon \cat{Tors}_{u_\ast(\GL_{n,\mc{O}})}(\ms{D})\to \cat{Tors}_{\GL_{n,\mc{O}}}(\ms{C})'$. Using similar ideas to above, one obtains the following.

\begin{prop}\label{prop:morphisms-of-topos-vector-bundle-equiv} Suppose that $(u_\ast,u^{-1})\colon\ms{C}\to\ms{D}$ is a morphism of topoi induced by either continuous or cocontinuous morphism. Then, $u_\ast\colon \cat{Vect}(\mc{C},\mc{O})'\to\cat{Vect}(\mc{D},u_\ast(\mc{O}))$ is a rank preserving $\otimes$-equivalence with quasi-inverse $u^\ast$ (see \S\ref{section:torsors via tensors} for this terminology).
\end{prop}

\subsection{Torsors and vector bundles on formal schemes}\label{ss:torsors-and-vb-on-formal-schemes} Let $\mf{X}$ be a formal scheme, and denote by $\mf{X}_\fl$ the category consisting of morphisms of formal schemes $\mf{Y}\to\mf{X}$, morphisms being $\mf{X}$-morphisms, and endowed with the Grothendieck topology where $\{\mf{Y}_i\to\mf{Y}\}$ is a cover if $\coprod_i \mf{Y}_i\to\mf{Y}$ is adically faithfully flat (see \cite[Chapter I, Definition 4.8.12.(2)]{FujiwaraKato}) and finitary in the sense of (2) in \stacks{03NW}. This site is subcanonical by \cite[Chapter I, Proposition 6.1.5]{FujiwaraKato}. Denote by $\mf{X}^\mathrm{adic}_{\fl}$, $\mf{X}_\et$, and $\mf{X}_\Zar$ the full subcategories of $\mf{X}_\fl$ consisting of objects whose structure morphism is adic, \'etale, and an open embedding, respectively, with the induced topology. Denote by $\mf{X}_{\mr{ZAR}}$, the full subcategory of $\mf{X}_\fl$ whose covers are given by Zariski covers. When $\mf{X}$ is a scheme, we use the notation $\mf{X}_\mathrm{fpqc}$ for $\mf{X}^\mathrm{adic}_\fl$. Each of these has a variant consisting only of affine (formal) schemes, but as these variants give rise to the same topos, we often confuse the two. Each of these sites is ringed via the usual structure sheaf (see \cite[Chapter I, Proposition 6.1.2]{FujiwaraKato}).

Let $\mc{G}$ be a smooth affine group (formal) $\mf{X}$-scheme. As affine adic morphisms satisfy effective descent in $\mf{X}_\fl$ (see \cite[Chapter I, Corollary 6.1.13]{FujiwaraKato}), and smoothness can be checked locally in $\mf{X}_\fl$ (see \cite[Chapter I, Proposition 6.1.8]{FujiwaraKato}) one may observe the following.

\begin{lem}\label{lem:torsor-representable} A $\mc{G}$-torsor on $\mf{X}_\fl$ is representable by a smooth and affine surjection $\mf{P}\to\mf{X}$.
\end{lem}

As smooth surjections have \'etale local sections, we obtain the following from Corollary \ref{cor:torsor-equiv-continuous-morphism}, and we denote the common category of $\mc{G}$-torsors on these three sites by $\cat{Tors}_\mc{G}(\mf{X})$.

\begin{cor}\label{cor:flat-and-etale-torsors-agree} The inclusions $\mf{X}_\et\to\mf{X}_\fl^\mathrm{adic} \to \mf{X}_\fl$ give equivalences on categories of $\mc{G}$-torsors. 
\end{cor}

Suppose that $R$ is a ring which is $J$-adically complete with respect to a finitely generated ideal $J\subseteq R$. Consider the left exact functor
\begin{equation*}
    \wh{(-)}\colon \cat{PSh}(\Spec(R)_\fpqc)\to\cat{PSh}(\Spf(R)^\mathrm{adic}_\fl),\qquad \mc{F}\mapsto \left(\wh{\mc{F}}(\mf{Y})\defeq \varprojlim_n \mc{F}(\mf{Y}_n)\right),
\end{equation*}
where $\mf{Y}_n:=(|\mf{Y}|,\mc{O}_{\mf{Y}}/J^n\mc{O}_\mf{Y})$. If $\mc{F}$ is a sheaf, then $\wh{\mc{F}}$ is a sheaf as the inverse limit functor is left exact. If $\mc{F}=h_P$ for a morphism $P\to\Spec(R)$, then $\wh{\mc{F}}$ is represented by $\wh{P}\to \Spf(R)$.

Let $\mc{G}$ be a smooth affine group $R$-scheme. Note that $\wh{\mc{G}}(\Spf(S))=\mc{G}(\Spec(S))$ when $S$ is $J$-adically complete, and we denote this common group by $\mc{G}(S)$, and confuse $\mc{G}$ and $\wh{\mc{G}}$. As $\wh{(-)}$ commutes with products, it naturally sends pseudo-torsors for $\mc{G}$ to pseudo-torsors for $\wh{\mc{G}}$. Due to the following, we can unambiguously denote $\cat{Tors}_\mc{G}(\Spf(R))$ and $\cat{Tors}_\mc{G}(\Spec(R))$ by the common symbol $\cat{Tors}_\mc{G}(R)$.

\begin{prop}\label{prop:completion-is-equivalence} The functor $\wh{(-)}$ functor induces an equivalence
\begin{equation*}
    \wh{(-)}\colon \cat{Tors}_\mc{G}(\Spec(R)_{\fpqc})\to \cat{Tors}_\mc{G}(\Spf(R)^\mathrm{adic}_\fl). 
\end{equation*}
\end{prop}
\begin{proof} Let us first establish bijectivity on the sets of isomorphism classes. To do this, first observe that by Corollary \ref{cor:flat-and-etale-torsors-agree}, we are free to replace $\Spec(R)_{\mr{fpqc}}$ (resp.\@ $\Spf(R)^\mr{adic}_\mr{fl}$) with $\Spec(R)_\et$ (resp.\@ $\Spf(R)_\et$). Observe then that we have a commutative diagram
\begin{equation}\label{eq:torsor-diagram}
    \xymatrix{H^1(\Spec(R)_{\et},\mc{G})\ar[r]^{\wh{(-)}}\ar[d] &  H^1(\Spf(R)_\et,\mc{G}).\ar[dl]^-{\sim}\\ \displaystyle \varprojlim_n H^1(\Spec(R/J^n)_{\et},\mc{G}) & }
\end{equation}
As the completion of a $\mc{G}$-pseudo-torsor is a $\mc{G}$-pseudo-torsor, that the arrow labeled by $\widehat{(-)}$ is well-defined (i.e.\@, sends a torsor to a torsor) follows from the observation that the completion of an \'etale cover of $\Spec(R)$ is an \'etale cover of $\Spf(R)$. The arrow labeled as an isomorphism is obtained from the equivalence of categories
\begin{equation*}
    \cat{Tors}_\mc{G}(\Spf(R)_\et)\isomto \twolim_n \cat{Tors}_\mc{G}(\Spec(R/J^n)),
\end{equation*}
given by sending $\mc{A}$ to $(\mc{A}_{\Spec(R/J^n)})$ with quasi-inverse taking $(\mc{A}_n)$ to the $\mc{G}$-torsor sending $\Spf(S)$ to $\varprojlim \mc{A}_n(\Spec(S/J^nS))$. That this quasi-inverse is well-defined (i.e.\@, actually produces torsors) follows from the topological invariance of \'etale sites $\Spf(R)_\et\isomto\Spec(R/J)_\et$, and the smoothness of each $\mc{A}_n$, which shows that any \'etale cover of $\Spec(R/J)$ trivializing $\mc{A}_1$ lifts uniquely to an \'etale cover of $\Spf(R)$ trivializing $\varprojlim \mc{A}_n$. So, the claim follows as the vertical arrow in \eqref{eq:torsor-diagram} is bijective by \cite[Theorem 2.1.6.(b)]{BCLoopTorsors}.

To show fully faithfulness we must show that for any two $\mc{G}$-torsors $\mc{F}_1$ and $\mc{F}_2$ on $\Spec(R)_{\fpqc}$ that the induced map $\Hom_\mc{G}(\mc{F}_1,\mc{F}_2)\to \Hom_{\mc{G}}(\wh{\mc{F}}_1,\wh{\mc{F}}_2)$ is a bijection. As $\underline{\Aut}(\mc{F}_2)$ is locally isomorphic to $\mc{G}$, we deduce from effective descent for affine morphisms in $\Spec(R)_\fpqc$ that $\underline{\Aut}(\mc{F}_2)$ is represented by some smooth affine group $R$-scheme $H$. Thus, $\underline{\Aut}(\wh{\mc{F}}_2)$ is represented by $\wh{H}$. Moreover, we may assume that $\mc{F}_1$ is isomorphic to $\mc{F}_2$, and thus by the bijectivity of isomorphism classes, that $\wh{\mc{F}}_1$ is isomorphic to $\wh{\mc{F}}_2$. So, $\Hom_\mc{G}(\mc{F}_1,\mc{F}_2)$ (resp.\@ $\Hom_{\mc{G}}(\wh{\mc{F}}_1,\wh{\mc{F}}_2)$) is a torsor for $\Aut(\mc{F}_2)=H(R)$ (resp.\@ $\Aut(\wh{\mc{F}}_2)=\wh{H}(R)$). The claim follows as $\Hom_\mc{G}(\mc{F}_1,\mc{F}_2)\to \Hom_{\mc{G}}(\wh{\mc{F}}_1,\wh{\mc{F}}_2)$ is equivariant for the bijection $H(R)\to\wh{H}(R)$.
\end{proof}

Let $\cat{FPMod}(R)$ denote the category of finite projective $R$-modules. The following is a vector bundle analogue of Proposition \ref{prop:completion-is-equivalence}.
\begin{prop}\label{prop:global-sections-equivalence-vb} The global section functor $\cat{Vect}(\Spf(R)_{\fl},\mc{O}_{\Spf(R)})\to\cat{FPMod}(R)$ is a bi-exact $R$-linear $\otimes$-equivalence (see \S\ref{section:torsors via tensors} for this terminology) which preserves rank.
\end{prop}
\begin{proof} We claim the source is equal to $\cat{Vect}(\Spf(R)_\Zar,\mc{O}_{\Spf(R)})$. By Proposition \ref{prop:morphisms-of-topos-vector-bundle-equiv} it suffices to show that for an object $\mc{E}$ of $\cat{Vect}_n(\Spf(R),\mc{O}_{\Spf(R)})$ that $P=\underline{\Isom}(\mc{O}_{\Spf(R)}^n,\mc{E})$ has a section Zariski locally on $\Spf(R)$. Up to replacing $R$ by a completed localization, we may assume by \stacks{05VG} that $P(R/JR)$ is non-empty. But, as $P$ is represented by a smooth formal $R$-scheme by Lemma \ref{lem:torsor-representable} we deduce from Hensel's lemma that $P(R)$ is non-empty. Then, by \cite[Chapter I, Theorem 3.2.8]{FujiwaraKato}, and the fact that any finite projective $R$-module $M$ is automatically $J$-adically complete,\footnote{Find an $R$-module $N$ such that $M\oplus N\simeq R^m$ for some $m$. As $R^m$ is $J$-adically complete the morphism $M\oplus N\to \widehat{M\oplus N}=\wh{M}\oplus\wh{N}$ is an isomorphism, from where it follows that $M\to \wh{M}$ is an isomorphism.} it suffices to show that for an adically quasi-coherent sheaf $\mc{E}$ on $\Spf(R)$, that $M=\mc{E}(\Spf(R))$ is finite projective if and only if $\mc{E}$ is a vector bundle, and the only if direction is clear. So, suppose that $\mc{E}$ is a vector bundle. Then, by \cite[Chapter I, Theorem 3.2.8]{FujiwaraKato} $M$ is a finitely generated $J$-adically complete $R$-module. Moreover, as $\mc{E}|_{\Spec(R/J^m)}$ is a vector bundle for all $m$, we know from \stacks{05JM} that $M/J^mM$ is finite projective for all $m$. Then, $M$ is finite projective by \stacks{0D4B}.
\end{proof}

Because of Proposition \ref{prop:global-sections-equivalence-vb} and its proof, the category of vector bundles on a formal scheme $\mf{X}$ is independent of the above-defined sites. We denote the common category by $\cat{Vect}(\mf{X})$ (omitting the structure sheaf from the notation). If $\mf{X}=\Spf(R)$, we shorten this notation further to $\cat{Vect}(R)$, and abusively identity it with $\cat{FPMod}(R)$.

\subsection{Tannakian formalism}\label{section:torsors via tensors}

For a ring $R$,\footnote{In this article $R$ will almost always be $\Z_p$ or $\Q_p$, and which it is should always be clear from context.} we say $\mc{C}$ is an \emph{(exact) $R$-linear $\otimes$-category} if 
\begin{itemize}
\Item (it is an exact category (see \cite[Appendix A]{Keller}),)
\item $\mc{C}$ is an additive $R$-linear category (see \stacks{0104} and \stacks{09MI}), 
\item the underlying category is Karoubian (see \stacks{09SF}),
\item there is an $R$-bilinear symmetric monoidal structure $\otimes\colon \mc{C}\times \mc{C}\to \mc{C}$ (see \stacks{0FFJ}).
\end{itemize}
For the unit object $\mathbf{1}$ of $\mc{C}$ and an object $X$ of $\mc{C}$, an \emph{element} of $X$ means a morphism $\mathbf{1}\to X$. 

For (exact) $R$-linear $\otimes$-categories $\mc{C}$ and $\mc{D}$, a functor $F\colon \mc{C}\to\mc{D}$ is an \emph{(exact) $R$-linear $\otimes$-functor} if it (preserves exact sequences and it) is $R$-linear (see \stacks{09MK}) and symmetric monoidal (see \stacks{0FFL} and \stacks{0FFY}). From \cite[I. Proposition 4.4.2]{SaaCatT}, a quasi-inverse of an $R$-linear $\otimes$-functor $F$ is automatically an $R$-linear $\otimes$-functor, in which case we call $F$ an \emph{$R$-linear $\otimes$-equivalence}. If $F$ is exact then it is not guaranteed that the same holds for its quasi-inverse (for example, the restriction functor $\cat{Rflx}(X)\to\cat{Rflx}(U)$ for a large open subset $U\subsetneq X$ when both are endowed with the exact structure inherited from the usual one on the category of coherent modules: see \S\ref{ss:reflexive-pseudo-torsors}). If an exact $R$-linear $\otimes$-functor has an exact $R$-linear $\otimes$-functor quasi-inverse, we say that $F$ is a \emph{bi-exact $R$-linear $\otimes$-equivalence}.

Let $\mc{C}$ be an $R$-linear $\otimes$-category and $X$ a dualizable object of $\mc{C}$ (see \stacks{0FFP}). As $\mc{C}$ is Karoubian, and by \stacks{0FFU} and \stacks{0FFT}, we may construct in $\mc{C}$ an object obtained from $X$ by taking any finite combination of direct sums, duals, symmetric products, and alternating products. By a \emph{set of tensors} $\mathds{T}$ on $X$ we mean a finite set of elements in an object built in this way. We write this symbolically as $\mathds{T}\subseteq X^\otimes$.\footnote{We often implicitly interpret $X^\otimes$ as the direct sum of these finite constructions in a larger $R$-linear $\otimes$-category that is closed under arbitrary direct sums when such a larger category is naturally given.} For an $R$-linear $\otimes$-functor $F\colon\mc{C}\to\mc{D}$ and a set of tensors $\mathds{T}\subseteq X^\otimes$ we obtain the set $F(\mathds{T})\subseteq F(X)^\otimes$ of tensors on $F(X)$. 

We define a \emph{tensor package} over $R$ to be a pair $(\Lambda_0,\mathds{T}_0)$ where $\Lambda_0$ is a finite projective $R$-module and $\mathds{T}_0\subseteq \Lambda_0^\otimes$. Given a tensor package $(\Lambda_0,\mathds{T}_0)$ we have the group $R$-scheme
\begin{equation*}
\mathrm{Fix}(\mathds{T}_0)\colon \cat{Alg}_R\to \cat{Set},\qquad S\mapsto \left\{g\in \GL(\Lambda_0\otimes_R S): g(t)=t,\text{ for all }t\in \mathds{T}_0\right\},
\end{equation*}
a closed subgroup scheme of $\GL(\Lambda_0)$. 

\begin{thm}[{\cite[Theorem 1.1]{Broshi}}]\label{thm:broshi1} Suppose that $R$ is a Dedekind domain. Then, for every flat finite type affine group $R$-scheme $\mc{G}$, and faithful representation $\mc{G}\to \GL(\Lambda_0)$, there exists a tensor package $(\Lambda_0,\mathds{T}_0)$ with $\mc{G}=\mathrm{Fix}(\mathds{T}_0)$.
\end{thm}

\begin{rem} Theorem \ref{thm:broshi1} was previously proven by Kisin in \cite[Proposition 1.3.2]{KisIntShab} in the case when $R$ is a discrete valuation ring and $\mc{G}$ has reductive generic fiber. As pointed out by Deligne in \cite{DeligneLetter} it is possible in this case to only consider $\mathds{T}_0$ contained in $\bigoplus_{m,n} \Lambda_0^{\otimes m}\otimes_R (\Lambda_0^\vee)^{\otimes n}$. This observation also applies to the situation of Theorem \ref{thm:broshi1}.
\end{rem}

Suppose now that $(\Lambda_0,\mathds{T}_0)$ is a tensor package over $R$ with $\mc{G}=\mathrm{Fix}(\mathds{T}_0)$. Denote by $\cat{Rep}_R(\mc{G})$ the natural exact $R$-linear $\otimes$-category of representations $\mc{G}\to\GL(\Lambda)$, where $\Lambda$ is a finite projective $R$-module. 
\begin{defn}\label{defn: G-C the groupoid of exact tensor functors}
    For any exact $R$-linear $\otimes$-category $\mc{C}$ a \emph{$\mc{G}$-object} in $\mc{C}$ is an exact $R$-linear $\otimes$-functor $\omega\colon \cat{Rep}_R(\mc{G})\to \mc{C}$. An \emph{isomorphism} $\omega\to \omega'$ is an invertible natural transformation, and we denote the groupoid of $\mc{G}$-objects in $\mc{C}$ by $\mc{G}\text{-}\mc{C}$. 
\end{defn}

Let $\ms{X}$ be a topos and $\mc{O}$ an $R$-algebra object of $\ms{X}$. Then, $\cat{Vect}(\ms{X},\mc{O})$ is an exact $R$-linear $\otimes$-category, with exactness inherited from the category of $\mc{O}$-modules. The following is a sheaf
\begin{equation}\label{eq: definition of mc G sub O}
\mc{G}_\mc{O}\colon \ms{X}\to \cat{Grp},\qquad T\mapsto \mc{G}(\mc{O}(T)),
\end{equation}
as $\mc{G}$ preserves all limits of rings. Observe that $(\GL_{n,R})_\mc{O}=\GL_{n,\mc{O}}$. We write $\cat{PseuTors}_{\mc{G}}(\ms{X})$ (resp.\@ $\cat{Tors}_{\mc{G}}(\ms{X})$) for $\cat{PseuTors}_{\mc{G}_\mc{O}}(\ms{X})$ (resp.\@ $\cat{Tors}_{\mc{G}_\mc{O}}(\ms{X})$), when $\mc{O}$ is clear from context. For any object $T$ of $\ms{X}$ there is, with the notation in Definition \ref{defn: G-C the groupoid of exact tensor functors}, a restriction functor 
\begin{equation*}
\GVect(\ms{X},\mc{O})\to \GVect(\ms{X}/T,\mc{O}),\qquad \omega\mapsto \omega_T\defeq (-)|_T\circ \omega,
\end{equation*}
where $(-)|_T$ denotes restriction. Denote by $\omega_\triv$ the $\mc{G}$-object given by $\omega_\triv(\Lambda)=\Lambda\otimes_R\mc{O}$. 

For an object $\mc{P}$ of $\cat{Tors}_\mc{G}(\ms{X})$ we obtain the object $\omega_\mc{P}$ of $\GVect(\ms{X},\mc{O})$ given by 
\begin{equation*}
    \omega_\mc{P}\colon \cat{Rep}_R(\mc{G})\to \cat{Vect}(\ms{X},\mc{O}),\qquad \Lambda\mapsto \mc{P}\wedge^\mc{G}(\Lambda\otimes_R \mc{O}),
\end{equation*}
which is locally on $\ms{X}$ isomorphic to $\omega_\triv$. Observe that for a representation $\rho\colon \mc{G}\to\GL(\Lambda)$, the vector bundle $\omega_\mc{P}(\Lambda)$ agrees, functorially in $\mc{P}$ and $\Lambda$, with the vector bundle associated to $\rho_\ast(\mc{P})$ by Proposition \ref{prop:gln-torsor-vector-bundle}, and so we sometimes confuse the two.

For a pair $(\mc{E},\mathds{T})$, where $\mc{E}$ is an object of $\cat{Vect}(\ms{X},\mc{O})$ and $\mathds{T}\subseteq \mc{E}^\otimes$, the functor
\begin{equation*}
\begin{aligned}\underline{\Isom}\left((\Lambda_0\otimes_R \mc{O},\mathds{T}_0\otimes 1),(\mc{E},\mathds{T})\right)\colon &\ms{X}\to \cat{Set},\\ &T\mapsto \left\{f\colon\Lambda_0\otimes_R  \mc{O}_T\isomto  \mc{E}_T: f(\mathds{T}_0\otimes 1)=\mathds{T}\right\},\end{aligned}
\end{equation*}
has the structure of an object of $\cat{PseuTors}_\mc{G}(\ms{X})$. We call $(\mc{E},\mathds{T})$ a \emph{twist} of $(\Lambda_0,\mathds{T}_0)$ if this pseudo-torsor is a torsor. By an \emph{isomorphism} of twists $(\mc{E},\mathds{T})\to (\mc{E}',\mathds{T}')$ we mean an isomorphism $\mc{E}\to \mc{E}'$ carrying $\mathds{T}$ to $\mathds{T}'$. Denote by $\cat{Twist}_\mc{O}(\Lambda_0,\mathds{T}_0)$ the groupoid of twists of $(\Lambda_0,\mathds{T}_0)$.

\begin{prop}\label{prop:torsors-and-twists}The functor 
\begin{equation*}
\cat{Twist}_\mc{O}(\Lambda_0,\mathds{T}_0)\to \cat{Tors}_{\mc{G}_\mc{O}}(\ms{X}),\qquad (\mc{E},\mathds{T})\mapsto \underline{\Isom}\left((\Lambda_0\otimes_R \mc{O},\mathds{T}_0\otimes 1),(\mc{E},\mathds{T})\right),
\end{equation*}
is an equivalence of groupoids with quasi-inverse given by sending $\mc{P}$ to $(\omega_\mc{P}(\Lambda_0),\omega_\mc{P}(\mathds{T}_0))$.
\end{prop}
\begin{proof} The association of $T$ in $\ms{X}$ to the groupoid of pairs $(\mc{E},\mathds{T})$ of an object $\mc{E}$ of $\cat{Vect}(\ms{X},\mc{O}/T)$ and $\mathds{T}\subseteq \mc{E}^{\otimes}$, forms a stack over $\ms{X}$ which we denote $C$. This proposition is then a special case of \cite[Chapitre III, Th\'eor\`eme 2.5.1]{Giraud} as the natural map $\mc{G}_\mc{O}\to \underline{\Aut}(\Lambda_0\otimes_R \mc{O}_X,\mathds{T}_0\otimes 1)$ is an isomorphism, and (with notation in loc.\@ cit.\@) $C(\Lambda_0\otimes_R\mc{O},\mathds{T}_0\otimes 1)=\cat{Twist}_\mc{O}(\Lambda_0,\mathds{T}_0)$.
\end{proof}

For an object $\omega$ of $\GVect(\ms{X},\mc{O})$ we have a pseudo-torsor 
\begin{equation*}
    \uIsom(\omega_\triv,\omega)\colon \ms{X}\to \cat{Set},\qquad T\mapsto \Isom(\omega_{\triv,T},\omega_T),
\end{equation*}
for the group sheaf $\underline{\Aut}(\omega_\triv)$. Call $\omega$ \emph{locally trivial} if this pseudo-torsor is a torsor, and by $\GVect^\lt(\ms{X},\mc{O})$ the full subgroupoid of $\GVect(\ms{X},\mc{O})$ of locally trivial objects. Say that $\mc{G}$ is \emph{reconstructible in $(\ms{X},\mc{O})$} if the natural map $\mc{G}_\mc{O}\to \underline{\Aut}(\omega_\triv)$ is an isomorphism. In this case, there is a natural equivalence $\cat{Tors}_{\mc{G}}(\ms{X})\to \GVect^\lt(\ms{X},\mc{O})$ given by sending $\mc{P}$ to $\omega_\mc{P}$, with quasi-inverse sending $\omega$ to $\underline{\Isom}(\omega_\triv,\omega)$.

If $\mc{G}$ is reconstructible in $(\ms{X},\mc{O})$, then for an object $\omega$ of $\GVect^\lt(\ms{X},\mc{O})$, the map
\begin{equation*}
\underline{\Isom}(\omega_\triv,\omega)\to \underline{\Isom}\left((\Lambda_0\otimes_R\mc{O},\mathds{T}_0\otimes 1),(\omega(\Lambda_0),\omega(\mathds{T}_0))\right)
\end{equation*}
given by evaluation is a morphism of pseudo-torsors where the source is a torsor, and so an isomorphism. Thus, $(\omega(\Lambda_0),\omega(\mathds{T}_0))$ is an object of $\cat{Twist}_\mc{O}(\Lambda_0,\mathds{T}_0)$. We deduce the following.

\begin{prop}\label{prop:torsors-twists-functors} Suppose that $\mc{G}$ is reconstructible in $(\ms{X},\mc{O})$. Then, 
\begin{equation*}
\xymatrixcolsep{5.5pc}\xymatrixrowsep{2.5pc}\xymatrix{\cat{Twist}_\mc{O}(\Lambda_0,\mathds{T}_0)\ar[rr]^-{(\mc{E},\mathds{T})\mapsto \underline{\Isom}((\Lambda_0\otimes_R\mc{O},\mathds{T}_0\otimes 1),(\mc{E},\mathds{T}))} & & \cat{Tors}_{\mc{G}}(\ms{X})\ar[d]^{\mc{P}\mapsto \omega_\mc{P}}\\ & & \GVect^\lt(\ms{X},\mc{O})\ar[ull]^{\omega\mapsto (\omega(\Lambda_0),\omega(\mathds{T}_0))\qquad\quad}}
\end{equation*}
is a commuting triangle of equivalences.
\end{prop}

The following shows that the assumptions in Proposition \ref{prop:torsors-twists-functors} are often satisfied.

\begin{thm}[{\cite[Theorem 1.2]{Broshi}, \cite[Theorem 19.5.1]{ScholzeBerkeley}}]\label{thm:broshi2}
Assume $R$ is a Dedekind domain and that $\mc{G}$ is $R$-flat. Then for an $R$-scheme $X$, $\mc{G}$ is reconstructible in $(X_\fpqc,\mc{O}_X)$ and every object of $\GVect(X)$ is locally trivial.
\end{thm}
\begin{proof} The only thing not contained in loc.\@ cit.\@ is the reconstructibility claim, but this follows from the fully faithfulness of the the functors in loc.\@ cit.\@ applied to the trivial objects.
\end{proof}

\begin{rem} Theorem \ref{thm:broshi2} actually implies that if $(\ms{T},\mc{O})$ is a ringed topos over $R$, then $\mc{G}$ is reconstructible in $(\ms{T},\mc{O})$. Indeed, for an object $T$ of $\ms{T}$, we must check that the map $\mc{G}(\mc{O}(T))\to \Aut(\omega_\mr{triv}|_T)$ is an isomorphism. But, observe that $\Aut(\omega_\mr{triv}|_T)$ boils down to understanding compatible automorphisms of $\Lambda\otimes_R \mc{O}|_T$ for $\Lambda$ a representation of $\mc{G}$. But, such automorphisms are precisely the $\mc{O}(T)$-linear automorphisms of $\Lambda\otimes_R \mc{O}(T)$. Thus, as $\mc{O}(T)$ is an $R$-algebra, we know from Theorem \ref{thm:broshi2} that the natural map from $\mc{G}(\mc{O}(T))$ is an isomorphism. 

In particular, using this one obtains simpler proofs for parts of  Propositions \ref{prop:tannakian-formalism-proetale-scheme}, \ref{prop:G-for-proet-adic}, \ref{prop:G-for-proet-diamonds}, and \ref{prop:rflx-broshi}, as well as the claim at the beginning of \S\ref{sss:Tannakian-prismatic-F-crystals}. We leave these proofs unchanged though, in case they are helpful. The fact that a simplification should exist was pointed out to us by Mark Kisin.
\end{rem}

\begin{rem} There is an error in \cite[Lemma 4.4 (iii)]{Broshi}, which is necessary for \cite[Theorem 1.2]{Broshi}. Namely, the inclusion $\mathcal{O}_X\to \mathcal{O}_G$\footnote{In this remark we use notation as in loc.\@ cit. In particular, our $X$ (resp.\@ $\Spec(R)$) is $Y$ (resp.\@ $X$) in loc.\@ cit.\@} is not split as $\mc{O}_G$-comodules, and thus one cannot use additivity to conclude that $F(\mc{O}_X)$ is a summand of $F(\mc{O}_G)$. But, observe that
\begin{equation*}
    0\to \mc{O}_X\to\mc{O}_G\to \mc{O}_G/\mc{O}_X\to 0,
\end{equation*}
is an exact sequence in $\cat{Rep}'(G)$, where $\mc{O}_G/\mc{O}_X$ is flat as $\mc{O}_X\to\mc{O}_G$ is split as $\mc O_X$-modules. By the exactness of $F$, and the flatness of $F(\mc{O}_G/\mc{O}_X)$, we obtain the (universally) exact sequence
\begin{equation*}
    0\to F(\mc{O}_X)\to F(\mc{O}_G)\to F(\mc{O}_G/\mc{O}_X)\to 0.
\end{equation*}
As $F(\mc{O}_X)=\mathcal{O}_Y$ universally injects into $F(\mc{O}_G)$, we deduce that $F(\mc{O}_G)$ is faithfully flat.
\end{rem}

\begin{rem}If $\mc{G}$ is $R$-smooth we may replace $X_\fl$ in Theorem \ref{thm:broshi2} with $X_\et$. When $X=\Spec(A)$, for $A$ complete with respect to a finitely generated ideal, we may replace all instances of $X$ in Theorem \ref{thm:broshi2} with $\Spf(A)$ by Proposition \ref{prop:completion-is-equivalence} and Proposition \ref{prop:global-sections-equivalence-vb}.
\end{rem}

\subsection{Reflexive pseudo-torsors}\label{ss:reflexive-pseudo-torsors} 
Let $R$ be a ring and $X$ an integral locally Noetherian normal $R$-scheme. An open embedding $j\colon U\hookrightarrow X$ is \emph{large} if $j(U)$ contains all points of codimension $1$. If $X'\to X$ is an \'etale map, then $X'$ is normal and $U\times_X X'\hookrightarrow X'$ is large. Denote by $\cat{A}(U_\et)$ the full subcategory of $\cat{Shv}(U_\et)$ of sheaves represented by an affine $U$-scheme. Let $\cat{A}_U(X_\et)$ be the full subcategory of $\cat{Shv}(X_\et)$ of sheaves represented by an affine $X$-scheme $Y$ such that $Y \times_X U$ is dense in $Y$. 

\begin{prop}[{cf.\@ \cite[Lemme 2.1]{ColliotTheleneSansuc}}]\label{prop:affine-sheaf-adjunction}
     Let $j\colon U\hookrightarrow X$ be a large open embedding. Then, $(j_\ast,j^\ast)\colon \cat{A}(U_\et) \to \cat{A}_U(X_\et)$ is a pair of quasi-inverse functors.
\end{prop}
\begin{proof} Suppose that $\mc{F}=\underline{\Spec}(\mc{A})$, where $\mc{A}$ is a quasi-coherent $\mc{O}_U$-algebra. As $\mc{O}_X\to j_\ast j^\ast \mc{O}_X=j_\ast\mc{O}_U$ is an isomorphism (see \cite[Lemme 2.1]{ColliotTheleneSansuc}), applying \stacks{01LQ} shows that $j_\ast\mc{F}$ is represented by $\underline{\Spec}(j_\ast\mc{A})$. To see $j_\ast\mc{F} \in \cat{A}_U(X_\et)$, we may assume that $X$ is affine. Take an affine covering $\{V_i\}_{i \in I}$ of $U$. Then $\mc{A}(U) \to \prod_{i \in I}\mc{A}(V_i)$ is injective. Hence $\underline{\Spec}(\mc{A}) \subset \Spec(\mc{A}(U))$ is dense. Therefore $j_\ast$ is well-defined. 
It is clear that $j^*j_* \simeq \id$. 
We show that $\id \simeq j_* j^*$. We may assume that $X$ is affine. Let $Y$ be an affine $X$-scheme. Take a closed immersion $Y \hookrightarrow \underline{\mathrm{Spec}}(\mc{O}_X [T_{\lambda}]_{\lambda \in \Lambda})$ over $X$, where $\Lambda$ is a set. Then a $U$-section of $Y$ uniquely extends to an $X$-section of $\underline{\mathrm{Spec}}(\mc{O}_X [T_{\lambda}]_{\lambda \in \Lambda})$ by \cite[Lemme 2.1 (i)]{ColliotTheleneSansuc}. Further it factors through $Y$ by the density of $U \subset X$. Hence the claim follows. 
\end{proof}

Recall that a coherent $\mc{O}_X$-module $\mc{F}$ is called \emph{reflexive} if the natural map 
\begin{equation*}
\mc{F}\to \mc{F}^{\ast\ast}\defeq \mc{H}om(\mc{H}om(\mc{F},\mc{O}_X),\mc{O}_X)
\end{equation*}
is an isomorphism. Equivalently, $\mc{F}$ is reflexive if there exists a large open embedding $j\colon U\hookrightarrow X$ such that $j^\ast\mc{F}$ is a vector bundle and for which the unit map $\mc{F}\to j_\ast j^\ast \mc{F}$ is an isomorphism (see \stacks{0AY6}). Denote by $\cat{Rflx}(X)$ the category of reflexive $\mc{O}_X$-modules, which has the structure of an $R$-linear $\otimes$-category where the tensor product is as in \stacks{0EBH}. We endow $\cat{Rflx}(X)$ with an exact structure by declaring an sequence exact if it is exact at every codimension $1$ point. Thus, $\cat{Rflx}(X)$ contains $\cat{Vect}(X)$ as a full $R$-linear tensor subcategory but \emph{not} as an exact subcategory. With this exact structure, if $j\colon U\hookrightarrow X$ is a large open embedding, then $j_\ast \colon\cat{Rflx}(U)\to\cat{Rflx}(X)$ is a bi-exact $R$-linear $\otimes$-equivalence (see \stacks{0EBJ}). Let $\cat{Rflx}_n^\mathrm{iso}(X)$ be the category of reflexive $\mc{O}_X$-modules $\mc{F}$ for which there is a large open embedding $j\colon U\to X$ with $j^\ast\mc{F}$ a rank $n$ vector bundle, with morphisms being isomorphisms of $\mc{O}_X$-modules.

For a smooth group $X$-scheme $\mc{G}$, a pseudo-torsor $\mc{Q}$ for $\mc{G}$ is called \emph{reflexive} if there is a large open embedding $j\colon U\hookrightarrow X$ with $j^\ast\mc{Q}$ a $\mc{G}$-torsor and the unit map $\mc{Q}\to j_\ast j^\ast\mc{Q}$ an isomorphism. Denote by $\cat{Rflx}_\mc{G}(X)$ the full subcategory of $\cat{PseuTors}_\mc{G}(X_\et)$ of reflexive pseudo-torsors.

\begin{prop}\label{prop:rflx-omnibus} Suppose $\mc{G}$ is a smooth group $X$-scheme. Then, the following is true.
\begin{enumerate}
\item An object $\mc{Q}$ of $\cat{PseuTors}_\mc{G}(X_\et)$ belongs to $\cat{Rflx}_\mc{G}(X)$ if and only if $\mc{Q}$ belongs to $\cat{A}(X_\et)$, and $\mc{Q}_x$ belongs to $\cat{Tors}_\mc{G}(\mc{O}_{X,x})$ for all codimension $1$ points $x$ of $X$.
\item For a large open embedding $j\colon U\hookrightarrow X$, the pair $(j_\ast,j^\ast)\colon \cat{Rflx}_\mc{G}(U)\to \cat{Rflx}_{\mc{G}}(X)$ are quasi-inverse.
\item  The natural functors $\mc{Q}\mapsto \mc{Q}\wedge^{\GL_{n,\mc{O}_X}}\mc{O}_X^n$ is an equivalence of categories 
\begin{equation*}\cat{Rflx}_{\GL_{n,\mc{O}_X}}(X)\to \cat{Rflx}^\mathrm{iso}_n(X),
\end{equation*}
with quasi-inverse given by $\mc{E}\mapsto \underline{\Isom}(\mc{E},\mc{O}_X^n)$.
\end{enumerate}
\end{prop}
\begin{proof} By Proposition \ref{prop:affine-sheaf-adjunction} and Proposition \ref{prop:gln-torsor-vector-bundle}, it only remains to show the if part of (1). Let $Y\to X$ be a finite type affine $X$-scheme representing $\mc{Q}$. By Proposition \ref{prop:affine-sheaf-adjunction}, it suffices to show that $Y_U\to U$ is faithfully flat for some large open $U$. As $Y_{\mc{O}_{X,x}}$ is faithfully flat over $\mc{O}_{X,x}$ for all codimension $1$ points $x$, we may conclude by \stacks{04AI} and \stacks{07RR}.
\end{proof}

For any \'etale map $X'\to X$ there is a restriction functor 
\begin{equation*}
\GRflx(X)\to \GRflx(X'),\qquad \omega\mapsto \omega_{X'
}\defeq (-)|_{X'}\circ \omega,
\end{equation*}
where $(-)|_{X'}$ denotes restriction. There is a fully faithful embedding of $\GVect(X)$ into $\GRflx(X)$. For an object $\mc{Q}$ of $\cat{Rflx}_\mc{G}(\ms{X})$ we obtain the object $\omega_\mc{Q}$ of $\GRflx(X)$ given by 
\begin{equation*}
    \omega_\mc{Q}\colon \cat{Rep}_R(\mc{G})\to \cat{Rflx}(X),\qquad \Lambda\mapsto \mc{Q}\wedge^\mc{G}(\Lambda\otimes_R \mc{O}_X).
\end{equation*}
For $\rho\colon \mc{G}\to\GL(\Lambda)$ one checks that $\omega_\mc{Q}(\Lambda)$ agrees, functorially in $\mc{Q}$ and $\Lambda$, with the reflexive module associated to $\rho_\ast(\mc{Q})$ by Proposition \ref{prop:rflx-omnibus}. If $\mc{G}$ is reconstructible in $(X_\et,\mc{O}_X)$, an object $\omega$ of $\GRflx(X)$ is called \emph{locally trivial} if the pseudo-torsor $\underline{\Isom}(\omega_\triv,\omega)$ is reflexive. Denote by $\GRflx^\mathrm{lt}(X,\mc{O}_X)$ the full subcategory of locally trivial objects. 

Suppose $(\Lambda_0,\mathds{T}_0)$ is a tensor package over $R$ and $\mc{G}$ is the base change to $X$ of $\mathrm{Fix}(\mathds{T}_0)$, which we abusively also denote $\mc{G}$. The category $\cat{Twist}^\rflx_{X}(\Lambda_0,\mathds{T}_0)$ of \emph{reflexive twists} consists of pairs $(\mc{E},\mathds{T})$ where $\mc{E}$ is an object of $\cat{Rflx}(X)$ and $\mathds{T}\subseteq \mc{E}^{\otimes}$ such that $\underline{\Isom}\left((\Lambda_0\otimes_R\mc{O}_X,\mathds{T}_0\otimes 1),(\mc{E},\mathds{T})\right)$ is reflexive. If $\mc{Q}$ is a reflexive pseudo-torsor for $\mc{G}$, then $(\omega_\mc{Q}(\Lambda_0),\omega_\mc{Q}(\mathds{T}_0))$ is reflexive.

Combining Proposition \ref{prop:torsors-twists-functors} and Proposition \ref{prop:rflx-omnibus}, one deduces the following.

\begin{prop}\label{prop:rflx-triangle} Suppose that $\mc{G}$ is reconstructible in $(X_\et,\mc{O}_X)$. Then,
\begin{equation*}
\xymatrixcolsep{5.5pc}\xymatrixrowsep{2.5pc}\xymatrix{\cat{Twist}^\rflx_{X}(\Lambda_0,\mathds{T}_0)\ar[rr]^-{(\mc{E},\mathds{T})\mapsto \underline{\Isom}((\Lambda_0\otimes_R\mc{O},\mathds{T}_0\otimes 1),(\mc{E},\mathds{T}))} & & \cat{Rflx}_\mc{G}(X)\ar[d]^{\mc{P}\mapsto \omega_\mc{P}}\\ & & \GRflx^\lt(X)\ar[ull]^{\omega\mapsto (\omega(\Lambda_0),\omega(\mathds{T}_0))\qquad\quad}}
\end{equation*}
is a commuting triangle of equivalences.
\end{prop}

If $R$ is a Dedekind domain, then we have an analogue of Theorem \ref{thm:broshi2}.
\begin{prop}\label{prop:rflx-broshi}
    Assume that $R$ is a Dedekind domain. Then, $\mc{G}$ is reconstructible in $(X_\et,\mc{O}_X)$ and every object of $\GRflx(X)$ is locally trivial.
\end{prop}
\begin{proof} Every representation $\Lambda$ of $\mathcal{G}$ occurs as a subquotient of some $T^{m,n}\defeq \Lambda_0^{\otimes m}\otimes_R (\Lambda_0^\vee)^{\otimes n}$ (cf.\@ \cite[Proposition 12]{dosSantos}). Thus, if $\omega$ is an object of $\GRflx(X)$ then the natural map $\underline{\Isom}(\omega_\triv,\omega)\to \underline{\Isom}(\Lambda_0\otimes_R \mc{O}_X,\omega(\Lambda_0))$ is an isomorphism onto the closed subscheme of those $f$ such that for all $m$ and $n$, and all subrepresentations $\Lambda\subseteq T^{m,n}$, the induced isomorphism $f^{m,n}\colon T^{m,n}\otimes_R \mc{O}_X\to \omega(T^{m,n})$ satisfies $f^{m,n}(\Lambda\otimes_R \mc{O}_X)\subseteq \Lambda$ and $(f^{m,n})^{-1}(\omega(\Lambda))\subseteq \Lambda\otimes_R\mc{O}_X$. As $\underline{\Isom}(\Lambda_0\otimes_R \mc{O}_X,\omega(\Lambda_0))$ is an affine finite type $X$-scheme, the same is true for $\underline{\Isom}(\omega_\triv,\omega)$. By Proposition \ref{prop:rflx-omnibus}, it then suffices to show that for all codimension $1$ points $x$ that $\underline{\Isom}(\omega_\triv,\omega)_x$ is a torsor. But, this corresponds to the exact $R$-linear $\otimes$-functor $\Lambda\mapsto \omega(\Lambda)_x$. As $\mc{O}_{X,x}$ is dimension $1$ all reflexive modules are vector bundles, and thus this an object of $\GVect(\mc{O}_{X,x})$. The claim then follows from Theorem \ref{thm:broshi2}.
\end{proof}

We end with some results inspired by \cite{ColliotTheleneSansuc}. Suppose that $\rho\colon \mc{G}\hookrightarrow \mc{H}$ is a closed embedding of reductive group $X$-schemes. Denote by $p\colon \mc{H}\to \mc{H}/\mc{G}$ the quotient sheaf in the fppf topology. Combining \cite[Corollary 9.7.7]{Alper} and \stacks{02KK} shows that $\mc{H}/\mc{G}$ belongs to $\cat{A}(X_\et)$. 

\begin{prop}\label{prop:pseudo-torsors-large-opens-equiv} Let $\mc{Q}$ be an object of $\cat{Rflx}_\mc{G}(X)$. Then, for any large open embedding $j\colon U\hookrightarrow X$, the natural map $\rho_\ast\mc{Q}\to j_\ast  \rho_\ast j^\ast \mc{Q}$ is an isomorphism.
\end{prop}
\begin{proof} For $X'\to X$ \'etale and $U'\defeq U\times_X X'$, we show that $\rho_\ast\mc{Q}(X')\to \rho_\ast j^\ast \mc{Q}(U')$ is bijective. By \cite[Chapitre III, Proposition 3.1.2]{Giraud}, the source (resp.\@ target) is identified with the set of $\mc{G}$-subtorsors $\mc{A}$ (resp.\@ $\mc{B}$) of $\mc{H}_{X'}\times \mc{Q}_{X'}$ (resp.\@ $\mc{H}_{U'}\times \mc{Q}_{U'}$). By \cite[Chapitre III, Proposition 1.3.6]{Giraud}, $\Hom_{\mc{H}_{X'}}(\rho_\ast(\mc{A}),\mc{H}_{X'})$ is in bijection with $\Hom_{\mc{G}_{X'}}(\mc{A},\mc{H}_{X'})$ which is non-empty by composing $\mc{A}\to \mc{H}_{X'}\times\mc{Q}_{X'}$ with the projection to $\mc{H}_{X'}$. So, $\rho_\ast\mc{A}$, and by a similar argument $\rho_\ast(\mc{B})$, are trivial. Then, \cite[Chapitre III, Proposition 3.2.2]{Giraud} implies that $\mc{A}$ (resp.\@ $\mc{B}$) is of the form $p^{-1}(s)$ (resp.\@ $p^{-1}(t)$) where $s$ (resp.\@ $t$) is an element of $(\mc{H}/\mc{G})(X')$ (resp.\@ $(\mc{H}/\mc{G})(U')$). Consider
\begin{equation*}
\xymatrixcolsep{4pc}\xymatrixrowsep{1pc}\xymatrix{\rho_\ast\mc{Q}(X')\ar[r]\ar[d] & \rho_\ast j^\ast\mc{Q}(U')\ar[d]\\ \{p^{-1}(s)\hookrightarrow \mc{H}_{X'}\times\mc{Q}_{X'}\}\ar[r]^a & \{p^{-1}(t)\hookrightarrow \mc{H}_{U'}\times\mc{Q}_{U'}\} }
\end{equation*}
where the vertical arrows are bijections, and $a$ is the obvious map. By Proposition \ref{prop:affine-sheaf-adjunction}, $(\mc{H}/\mc{G})(X')\to (\mc{H}/\mc{G})(U')$ is bijective, and so the $s$ and $t$ occurring in this diagram are in bijective correspondence, and applying Proposition \ref{prop:affine-sheaf-adjunction} and Proposition \ref{prop:rflx-omnibus} shows $a$ is bijective.
\end{proof}

\begin{prop}\label{prop:reflexive-pseu-tors-is-tors-criterion}
   Suppose that $(\Lambda_0,\mathds{T}_0)$ is a tensor package with $\mc{G}\defeq\mathrm{Fix}(\mathds{T}_0)$ reductive. Denote by $\rho_0\colon \mc{G}\to \GL(\Lambda_0)$ the tautological map and let $(\mc{E},\mathds{T})$ be an object of $\cat{Twist}^\rflx_{X}(\Lambda_0,\mathds{T}_0)$. Then,
    \begin{equation*}
    (\rho_0)_\ast\mc{Q}=\underline{\Isom}\left(\Lambda_0\otimes_R\mc{O}_X,\mc{E}\right),\quad \text{where }\mc{Q}\defeq \underline{\Isom}\left((\Lambda_0\otimes_R\mc{O}_X,\mathds{T}_0\otimes 1),(\mc{E},\mathds{T})\right).
    \end{equation*} 
    In particular, $\mc{Q}$ is a torsor if and only if $\mc{E}$ is locally free.
\end{prop}
\begin{proof} Let $j\colon U\hookrightarrow X$ be a large open embedding such that $j^\ast\mc{Q}$ is a torsor. Applying \cite[Chapitre III, Proposition 1.3.6]{Giraud}, we obtain a map $(\rho_0)_\ast j^\ast \mc{Q}\to  \underline{\Isom}\left(\Lambda_0\otimes_R\mc{O}_U,\mc{E}|_U\right)$ of torsors. As the natural map $\underline{\Isom}\left(\Lambda_0\otimes_R\mc{O}_X,\mc{E}\right)\to j_\ast\underline{\Isom}\left(\Lambda_0\otimes_R\mc{O}_U,\mc{E}|_U\right)$ is an isomorphism from Proposition \ref{prop:affine-sheaf-adjunction} and Proposition \ref{prop:rflx-omnibus}, we conclude by Proposition \ref{prop:pseudo-torsors-large-opens-equiv}. For the second claim, it suffices to show the if statement, which follows as there is a tautological surjection of sheaves $\GL_{n,X}\times \mc{Q}\to (\rho_0)_\ast \mc{Q}$ and thus if the target is locally non-empty, so then must be the source.
\end{proof}

\begin{rem}\label{rem:alternative-to-rflx-pseudo-torsors} We introduced reflexive pseudo-torsors because we feel they are natural extensions of ideas present in the current article, that may be useful in a Tannakian formalism of analytic prismatic $F$-torsors. That said, while we do use reflexive pseudo-torsors in the body of this article, this is mainly through the final claim in Proposition \ref{prop:reflexive-pseu-tors-is-tors-criterion}. The reader uninterested in the formalism of reflexive pseudo-torsors should note that such a result is already obtained by the method of proof in \cite[Th\'eor\`eme 6.13]{ColliotTheleneSansuc}, which we briefly sketch. 

Let $j\colon U\hookrightarrow X$ be a large open such that $j^\ast\mc{Q}$ is a torsor. Moving to an \'etale extension of $X$, we may assume that $\mathcal{E}$ is trivial. Thus $\rho_\ast j^\ast\mc{Q}$ is trivial, and so comes from an element of $(\GL(\Lambda_0)/\mc{G})(U)$. As $(\GL(\Lambda_0)/\mc{G})_X$ is affine over $X$, this can be extended to an element of $(\GL(\Lambda_0)/\mc{G})(X)$ by Proposition \ref{prop:affine-sheaf-adjunction}, which gives rise to a $\mc{G}$-torsor $\mc{Q}'$. Evidently there exists an isomorphism of $\mc{G}$-torsors $f\colon j^\ast\mc{Q}\to j^\ast\mc{Q}'$ which as $\mc{Q}$ and $\mc{Q}'$ are affine (the former as it is a closed subscheme of the affine scheme $\underline{\Isom}(\Lambda_0\otimes_R\mc{O}_X,\mc{E})$, and the latter by Lemma \ref{lem:torsor-representable}), extends to an isomorphism of sheaves $\mc{Q}\to\mc{Q}'$ by \ref{prop:affine-sheaf-adjunction}. That this is $\mc{G}$-equivariant, and thus an isomorphism of pseudo-torsors follows as $f|_U$ is $\mc{G}$-equivariant, and $j(U)$ is Zariski dense in $X$.
\end{rem}

\end{document}